\documentclass{whrartcl}

\usepackage[inference]{semantic}
\usepackage{mathpartir}
\usepackage{xcolor}
\usepackage{url}
\usepackage{bussproofs}
\usepackage{varwidth}
\usepackage{cmll}
\usepackage{wasysym}

\usepackage{caption}
\usepackage{subcaption}

\DeclareUnicodeCharacter{03BC}{$\mu$}
\usepackage[
  style=numeric-comp,
  sorting=nyt,
  maxbibnames=99,
  backref=false,
  doi=false,
  isbn=false,
  url=false,
  eprint=false,
  backend=biber
  ]{biblatex}
\bibliography{bibliography}

\newcommand{\RR}{\mathcal{R}}
\newcommand{\BB}{\mathbb{B}}
\newcommand{\FF}{\mathbb{F}}
\newcommand{\PP}{\mathcal{P}}

\newcommand{\TT}{\mathcal{T}}
\newcommand{\TS}{\mathcal{S}}
\newcommand{\TF}{\mathcal{F}}
\newcommand{\Aa}{\mathcal{A}}

\newcommand{\CC}{\mathcal{C}}
\newcommand{\pow}{\mathcal{P}}

\newcommand{\ol}[1]{\overline{#1}}

\newcommand{\abs}[1]{|#1|}
\newcommand{\uh}{\upharpoonright}

\newcommand{\pto}{\dot{\to}}
\newcommand{\addr}{\,@\,}

\newcommand{\Pp}{\textsc{Pp}}

\newcommand{\im}{\mathrm{im}} 
\newcommand{\dom}{\mathrm{dom}}
\newcommand{\cod}{\mathrm{cod}}
\newcommand{\Ob}{\mathrm{Ob}}
\newcommand{\var}{\mathrm{Var}}

\newcommand{\prop}{\mathrm{Prop}}
\newcommand{\form}{\textsc{Form}}
\newcommand{\term}{\textsc{Term}}
\newcommand{\type}{\textsc{Type}}
\newcommand{\Hom}{\textsc{Hom}}
\newcommand{\Leaf}{\mathrm{Leaf}}
\newcommand{\Inner}{\mathrm{Inner}}
\newcommand{\Chld}{\mathrm{Chld}}
\newcommand{\Inf}{\mathrm{Inf}}
\newcommand{\Occ}{\mathrm{Occ}}
\newcommand{\Up}{\mathrm{Up}}
\newcommand{\Down}{\mathrm{Down}}
\newcommand{\FV}{\mathrm{FV}}

\newcommand{\merge}{\mathrm{merge}}

\newcommand{\Sb}{\textsc{Sb}}
\newcommand{\Seq}{\textsc{Seq}}
\newcommand{\RA}{\mathit{RA}}
\newcommand{\CA}{\mathit{CA}}
\newcommand{\PA}{\mathit{PA}}
\newcommand{\CGT}{\mathit{CGT}}
\newcommand{\RGT}{\mathit{RGT}}
\newcommand{\Rrr}{\mathrm{R}}
\newcommand{\Sss}{\mathrm{S}}
\newcommand{\mML}{\mathit{K}_\mu}
\newcommand{\BRmML}{\mathit{RK}_\mu^\BB}
\newcommand{\FRmML}{\mathit{RK}_\mu^\FF}
\newcommand{\JS}{\mathit{JS}}
\newcommand{\sdash}{\Rightarrow}
\newcommand{\strip}{\mathit{strip}}

\newcommand{\elab}{\mathit{expand}}
\newcommand{\embed}{\mathit{embed}}
\newcommand{\search}{\mathit{search}}

\newcommand{\SA}{\mathfrak{A}}
\newcommand{\SB}{\mathfrak{B}}
\newcommand{\SSB}{\mathfrak{S}}
\newcommand{\SC}{\textsc{Pfs}}

\newcommand{\step}[1]{\stackrel{#1}{\leadsto}}

\newcommand{\gstep}[1]{\stackrel{#1}{\leadsto_g}}

\newcommand{\sreset}[2]{#1 \uh #2}
\newcommand{\rsep}{\,|\,}

\newcommand{\RAx}{\textsc{Ax}}
\newcommand{\RCtr}{\textsc{Ctr}}
\newcommand{\REx}{\textsc{Ex}}

\newcommand{\RThin}{\textsc{Thin}}
\newcommand{\RWk}{\textsc{Wk}}
\newcommand{\RWeak}{\textsc{Weak}}
\newcommand{\RCut}{\textsc{Cut}}
\newcommand{\RSub}{\textsc{Sub}}
\newcommand{\RFocus}{\textsc{Focus}}
\newcommand{\RPop}{\textsc{Pop}}
\newcommand{\RReset}{\textsc{Reset}}
\newcommand{\RCond}{\textsc{Cond}}
\newcommand{\RRec}{\textsc{Rec}}
\newcommand{\RId}{\textsc{Id}}
\newcommand{\RMerge}{\textsc{Merge}}
\newcommand{\RMod}{\textsc{Mod}}

\newcommand*{\langeq}{\mathrel{\vcenter{\baselineskip0.5ex \lineskiplimit0pt
      \hbox{\scriptsize.}\hbox{\scriptsize.}}}%
  \mathrel{\vcenter{\baselineskip0.5ex \lineskiplimit0pt
      \hbox{\scriptsize.}\hbox{\scriptsize.}}}%
  =}

\newsavebox{\mypt}

\numberwithin{lemma}{section} 

\colorlet{dwnote}{wpurple}
\colorlet{gelnote}{cyan}
\colorlet{insertcolour}{cyan}
\colorlet{deletecolour}{red}
\usepackage[normalem]{ulem} 

\newcommand{\dreplace}[2]{\textcolor{dwnote}{{#2}}\textcolor{deletecolour}{\sout{#1}}} 
\newcommand{\dadden}[1]{\dreplace{}{#1}}
\newcommand\blfootnote[1]{%
  \begingroup
  \renewcommand\thefootnote{}\footnote{#1}%
  \addtocounter{footnote}{-1}%
  \endgroup
}

\title{From GTC to $\RReset$: Generating Reset Proof Systems from Cyclic Proof Systems}
\author{Graham E. Leigh, University of Gothenburg \\ Dominik Wehr, University of Gothenburg}

\begin{document}

\maketitle

\blfootnote{This work was supported by the Knut and Alice Wallenberg Foundation [2020.0199,
  2015.0179].}

\begin{abstract}
  We consider cyclic proof systems in which derivations are graphs rather than
  trees.
  Such systems typically come with a condition that isolates which derivations are admitted as `proofs', known as a the \emph{soundness condition}.
  This soundness condition frequently takes the form of either a \emph{global trace} condition, a property dependent on all infinite paths in the proof-graph, or a \emph{reset}
  condition, a `local' condition depending on the simple cycles only which, as a result, is typically stable under more proof transformations.

  In this article we present a general method for constructing cyclic proof systems with reset condition from cyclic proof with global trace conditions.
  In contrast to previous approaches, this method of generation is entirely
  independent of logic's semantics, only relying on
  combinatorial aspects of the notion of `trace' and `progress'.
  We apply this method to present reset proof systems for three cyclic proof
  systems from the literature: cyclic
  arithmetic~\cite{simpsonCyclicArithmeticEquivalent2017}, cyclic Gödel's
  T~\cite{dasCircularVersionGodel2021} and cyclic tableaux for the modal
  $\mu$-calculus~\cite{niwinskiGamesMcalculus1996}.
\end{abstract}


\tableofcontents


\section{Introduction}
\label{sec:intro}

In cyclic proofs, leaves may be annotated with `recursive' references to previous deduction 
steps instead of axioms, yielding derivations shaped like finite graphs rather
than merely (finite) trees. Proof systems which allow for such proofs have proven particularly
well-suited to logics which feature fixed-points or 
(co-)inductively defined concepts (see e.g.~\cite{sprengerGlobalInductionMechanisms2003,niwinskiGamesMcalculus1996,simpsonCyclicArithmeticEquivalent2017,brotherstonSequentCalculusProof2006}). Because cyclic proofs
may have infinite
branches, their soundness usually cannot be reduced to the soundness of the
system's axioms and truth-preservation of its rules.
In such cases, a further soundness condition must be imposed.
The most common such condition is known as the \emph{global trace condition}: A cyclic
derivation is sound if every infinite branch through it has an infinitely
progressing trace. The precise notion of trace and progress vary
between proof systems. In Simpson's Cyclic Arithmetic~\cite{simpsonCyclicArithmeticEquivalent2017} traces are sequences of
terms following a branch of the derivation and are considered to have progressed when the term decreases in value.
By contrast, in cyclic proof systems for logics featuring fixed points, traces are typically sequences of
formulas and progress is based on fixed point unfoldings.

The global trace condition is widespread in the literature
because it is often very simple to adapt to new cyclic proof
systems and allows for direct proofs of soundness. However, it also brings with
it some disadvantages. Verifying whether a cyclic derivation satisfies the
global trace condition, and thus verifying whether it constitutes a proof, is in
general PSPACE-complete~\cite{afshariAbstractCyclicProofs2022}. Furthermore, the
`global' nature of the global trace condition -- in which `local' changes in a
proof can interfere with `global' soundness -- often makes it ill-suited to
proof theoretic investigations. Indeed, most results
of cyclic proof theory besides soundness and completeness
are derived for systems with alternative soundness conditions, such as
reset
proofs~\cite{afshariCutfreeCompletenessModal2017,afshariLyndonInterpolationModal2022,afshariUniformInterpolationCyclic2021,martiFocusSystemAlternationFree2021},
induction orders~\cite{sprengerStructureInductiveReasoning2003a} and bouncing
threads~\cite{baeldeBouncingThreadsCircular2022}.

Reset proofs are cyclic proofs that, in place of a global trace
  condition, employ a mechanism of annotating sequents and a
specific `\RReset{}' rule of inference marking `progress'. Such a reset condition is local to each simple cycle of the proof's underlying graph. This
eases both proof checking, simplifying the problem to polynomial time (and frequently linear time),
as well as making it easier to reason about proof transformations.
The first instance of a reset proof system in the literature is the tableau
system for the modal $\mu$-calculus put forward by
Jungteerapanich~\cite{jungteerapanichTableauSystemsModal2010}. Since then,
similar\footnote{Strictly speaking, not all of these systems are reset proof
  systems in our sense of the term. See the discussion of related work in \Cref{sec:conclusion} for more detail.}
proof systems have been
designed for the alternation-free fragment of the
modal $\mu$-calculus~\cite{martiFocusSystemAlternationFree2021}, the first-order
$\mu$-calculus~\cite{afshariCyclicProofsFirstorder2022a}, modal logics with master modalities~\cite{rooduijnCyclicHypersequentCalculi2021} and
full Computation Tree Logic $\text{CTL}^*$~\cite{afshariCyclicProofSystem2023}.
While it has been observed that reset proof systems can be obtained from
global-trace-based cyclic proof systems by annotating sequents with Safra
automata~\cite{safraComplexityOmegaAutomata1988} (see e.g.
\cite{jungteerapanichTableauSystemsModal2010,martiFocusSystemAlternationFree2021}),
all aforementioned reset proof systems were designed as variations of
Jungteerapanich's system, rather than `directly' via the Safra construction.

In this article, we make the connection between Safra automata and reset
proof systems formal. By adopting suitable abstract notions of trace conditions and
their induced cyclic proof systems, we show that any
cyclic proof system given by a global trace
condition naturally induces an equivalent reset proof system.
This result is not solely abstract but also provides a `recipe' for deriving
a corresponding reset proof system from any suitable cyclic proof system, even
those unrelated to the modal $\mu$-calculus.
The abstract notions of trace condition and cyclic proof system we employ are
general enough to cover the majority of proof systems studied in the literature on cyclic
proof systems.
We demonstrate the method on two cyclic proof systems from the literature:
Cyclic Arithmetic~\cite{simpsonCyclicArithmeticEquivalent2017} and Cyclic Gödel's
T~\cite{dasCircularVersionGodel2021}, obtaining in each case an equivalent reset-style proof system. 
Applying the construction to cyclic proofs for the modal \( \mu \)-calculus induces
a different reset proof system depending on the notion of trace employed.
These latter systems serve to illustrate the difference
between our method and that employed by Jungteerapanich and Stirling.

\section{Outline of the article}
\label{sec:overview}

In \Cref{sec:prelims} we give definitions that underpin and
motivate the remainder of the article. The two central concepts we rely on are an abstract
rendition of cyclic proof systems (\Cref{sec:abs-ps}) and a method of specifying
a global trace condition for a cyclic proof system in terms of so-called
activation algebras (\Cref{sec:gtc}). We recall some
definitions and results related to infinite word and tree automata in
\Cref{sec:automata} which we rely on throughout the article.

A \emph{cyclic proof system} consists of two components: a set of
\emph{derivation rules} and a \emph{soundness condition}. Given a set of derivation rules, a
pre-proof is obtained by annotating a `tree-shaped' cyclic graph with instances of said
rules. The soundness condition is a condition on such pre-proofs which
distinguishes proofs in the cyclic proof system from mere pre-proofs. Another
important concept is the cyclic proof system \emph{homomorphism}: Given two cyclic proof
systems $\RR$ and $\TS$, a homomorphism $f \colon \RR \to \TS$ roughly consists of a
translation of $\RR$-sequents to $\TS$-sequents such that, under this
translation, the derivation rules of $\RR$ are admissible in $\TS$. Crucially,
such a homomorphism $f \colon \RR \to \TS$ allows $\RR$-proofs to be translated into
$\TS$-proofs, thereby relating the two proof systems. The central conceit of
this article is\dadden{,} for certain cyclic proof systems $\RR$\dadden{,} to construct a reset
proof system $\Rrr(\RR)$ `corresponding to' $\RR$. We formalize this notion of
`correspondence' by constructing suitable homomorphisms in both directions.

The second central concept of this article -- a method of
  specifying global trace conditions for abstract cyclic proofs -- is recalled
  in \Cref{sec:gtc}.
A trace condition on pre-proofs
identifies a collection of \emph{trace objects} in each sequent
and a collection of \emph{traces} connecting trace objects along each infinite
branch of a cyclic proof. Along these traces there exists an accumulative
notion of `progress', specified via \emph{activation algebras}. A pre-proof is
considered a proof if every infinite branch carries a trace which progresses infinitely often.
The vast majority of soundness conditions for cyclic proof systems
found in the literature are global trace conditions of this kind.

In this article,
we show that every cyclic proof system whose soundness condition can be realised
as such a global trace condition can be associated a reset proof system, i.e., a cyclic proof system in which the soundness condition is wholly determined by simple cycles.
Given a cyclic proof system $\RR$ whose soundness condition is specified by an activation algebra $\Aa$, the reset system $\Rrr(\RR)$ is obtained by annotating the deduction elements of the cyclic proof system (the `sequents') with \emph{Safra boards}.
Introduced in
\Cref{sec:safra-boards}, Safra boards are inspired by the Safra construction~\cite{safraComplexityOmegaAutomata1988} used
in determinising infinite word automata, specifically their presentation in~\cite{kozenSafraConstruction2006}.
Roughly, a Safra board for a given sequent
with trace objects $X$ consists of `squares' $(x, a) \in X \times \Aa$ on which
stacks of playing chips are resting, `tracking' the progress of the trace
values. Given a $\RR$-derivation rule with conclusion $\Gamma$ and $\Gamma'$ as
one of its premises, there are rules describing how to move and extend the
stacks on a board for $\Gamma$ to obtain a board for $\Gamma'$ which takes into
account the `progress' made in the trace step from $\Gamma$ to $\Gamma'$.
Additionally, there are certain bookkeeping operations that may be performed on
such Safra boards, including a reset operation which `resets' some of the
progress tallied on a Safra board. This machinery allows for a simpler trace
condition: A pre-proof is a proof if along every infinite branch, infinitely
many reset-steps take place, indicating that a trace value has progressed
infinitely along said branch. 
Differing from the global trace condition, this property can be
  established by simply verifying that appropriate resets are part of each
  individual simple cycle of a pre-proof, yielding a `local' soundness condition.

For cyclic proof systems $\RR$ with a global
trace condition specified in terms of an activation algebra $\Aa$, the reset system $\Rrr(\RR)$ is defined in
\Cref{sec:acd_pwn}. There are two key properties we prove relating $\RR$ and
$\Rrr(\RR)$: \emph{soundness} and \emph{completeness}. 
Soundness
states that any annotated sequent provable in the reset system $\Rrr(\RR)$ is provable, without annotations, in the original system $\RR$. Conversely, completeness is the property that any
sequent provable in $\RR$ is provable in $\Rrr(\RR)$.
The names of these two properties are apt because they allow soundness and
completeness of $\RR$, relative to some semantics, to be `lifted' to
$\Rrr(\RR)$.
Both results are established by providing suitable cyclic
proof system homomorphisms. In \Cref{sec:sound}, we consider the translation
$\strip \colon \Rrr(\RR) \to \RR$ which simply strips an $\Rrr(\RR)$-proof of its Safra
board annotations and removes the derivation steps corresponding to the various
bookkeeping operations on Safra boards. By showing this to be a homomorphism, we can
conclude that every $\Rrr(\RR)$-proof induces a naturally corresponding $\RR$-proof
of the `same' sequent, yielding soundness.
The proof of completeness in \Cref{sec:complete} is less direct: For every finite subsystem $\TF$ of
$\RR$ we define a finite subsystem $\Sss(\TF)$ of
$\Rrr(\RR)$ which enjoys the proof search property: For every sequent $\Gamma$
provable in $\TF$ via cyclic proof $\Pi$, there exists an annotated finite
unfolding of $\Pi$ in $\Sss(\TF)$ which is a proof of $\Gamma$. As $\Sss(\TF)$ can be
embedded into $\Rrr(\RR)$ via a homomorphism $\embed \colon \Sss(\TF) \to \Rrr(\RR)$ this
yields a proof of $\Gamma$ in the reset system.

\Cref{sec:concrete} applies the above results to obtain reset systems for
various cyclic proof systems from the literature: Peano arithmetic
(\Cref{sec:pa}), Gödel's T (\Cref{sec:godel-t})
and the modal $\mu$-calculus (\Cref{sec:modal-mu}). While the system $\Rrr(\RR)$ is
sound and complete for any suitable cyclic system, it tends to not be very
pleasant to `use'. This state of affairs can usually be assuaged with a few
ergonomic adjustments. This is precisely what we do in \Cref{sec:concrete}: For
each of the concrete cyclic systems $\RR$ above we design a bespoke reset
system $\TS$ `inspired' by $\Rrr(\RR)$.
Soundness and completeness of \( \TS \) is obtained via a pair of homomorphisms $\embed \colon \TS \to \Rrr(\RR)$ and $\elab \colon \Sss(\TF) \to \TS$.
Importantly,
the construction of these bespoke systems $\TS$ and the homomorphisms $\embed$
and $\elab$ requires very little work when relying on the results of the
previous sections. We hope these examples prove illuminating enough for readers
to be able to do the same with any suitable cyclic proof system of their
choosing.

We close in \Cref{sec:conclusion} with a short conclusion, an overview of
related work and an outlook of future investigations.


\section{Preliminaries}
\label{sec:prelims}

\subsection{Cyclic Proof Systems}
\label{sec:abs-ps}

We begin by giving a suitable abstract account of cyclic proof systems. Because
we employ a very broad notion of soundness condition, every cyclic proof system
we are aware of is an instance of this notion of cyclic proof system.

Cyclic proofs and preproofs are certain finite graphs whose nodes are labeled by
sequents, according to a derivation system. Instead of general graphs, we use
cyclic trees as the data structure underlying our notion of cyclic proof. They
have proven slightly more convenient in some of our definitions and proofs.
A \emph{tree} is a
non-empty set $T \subseteq \omega^*$ which is closed under taking prefixes. Each
$t \in T$ is called a \emph{node} and the nodes in $\Chld(t) \coloneq \{ti \in T
~|~ i \in \omega\}$ are called its \emph{children}. A node $t$ is a \emph{leaf} of $T$ if
$\Chld(t) = \emptyset$ and an \emph{inner node} otherwise.
A \emph{cyclic tree} is a pair $(T, \beta)$ of a finite tree $T$ and
a partial function $\beta \colon \Leaf(T) \pto \Inner(T)$ mapping some leaves of $T$
onto inner nodes of $T$ such that $\beta(t) < t$ by the prefix ordering for
every $t \in \dom(\beta)$. 
If $t \in \dom(\beta)$ one calls it a \emph{bud} and $\beta(t)$ its \emph{companion}.

\begin{definition}
  A \emph{derivation system} is a triple $(\Seq, \RR, \rho)$ consisting of a pair of
  sequents $\Seq$ and a set $\RR$ of derivation rules and a rule-interpretation
  $\rho \colon \RR \to \Seq^*$ such that for each $R \in \RR$, $\rho(R) = 
  (\Gamma, \Delta_1, \ldots, \Delta_{n - 1}) \in \Seq^n$ for $n > 0$. 
  The sequent $\Gamma$ is \emph{conclusion} of $R$ and the $\Delta_i$ its
  \emph{premises}. Henceforth, we refer to a derivation system $(\Seq, \RR,
  \rho)$ simply by $\RR$.

  An \emph{$\RR$-preproof} is a triple
  $\Pi = (C, \lambda, \delta)$ consisting of a cyclic
  tree $C = (T, \beta)$ together with a labeling $\lambda \colon T \to \Seq$ such
  that for every $t \in \dom(\beta)$ one has $\lambda(t) = \lambda(\beta(t))$ and
  a partial function $\delta \colon (T \setminus \dom(\beta)) \pto \RR$ such that for
  each $t \in T \setminus \dom(\beta)$
  \begin{itemize}
  \item either $t \in \dom(\delta)$ with $\rho(\delta(t)) = (\Gamma, \Delta_1, \ldots,
    \Delta_n)$ and $\lambda(t) = \Gamma$ and furthermore $\Chld(t) = \{t1,
    \ldots, tn\}$ and $\lambda(ti) = \Delta_i$ 
  \item or $t \in \Leaf(T)$.
  \end{itemize}
  Denote by $\Pp(\RR)$ for the set of $\RR$-preproofs.
  The sequent $\lambda(\varepsilon)$ is called the \emph{endsequent} of $\Pi$. Each leaf $o \in
  \Leaf(T) \setminus \dom(\delta)$ is called \emph{open} and its associated
  sequent $\lambda(o)$ is a \emph{assumption} of $\Pi$.

  A \emph{cyclic proof system} is a tuple $(\Seq, \RR, \rho, \SC)$ consisting of a
  derivation system $(\Seq, \RR, \rho)$ and the set $\SC \subseteq \Pp(\RR)$ of $\RR$-preproofs
  without assumptions called \emph{$\RR$-proofs}. Any $\Pi \in \SC$ with endsequent $\Gamma$ is called
  \emph{a proof of ~$\Gamma$}. Such a preproof is said to satisfy the
  \emph{soundness condition} of $\RR$. We extend the naming convention for
  derivation systems to cyclic derivation systems, referring to $(\Seq, \RR,
  \rho, \SC)$ by $\RR$.
\end{definition}

Proof- and preproof morphisms
between cyclic proof systems play a key role in the results of
this article. In essence, a preproof morphism $f \colon \RR \to \RR'$ witnesses that the
derivation rules of $\RR$ are admissible in $\RR'$. This gives rise to a method for
translating $\RR$-preproofs into $\RR'$-preproofs: Simply replace each application of
a derivation rule in the $\RR$-proof by the $\RR'$-preproof witnessing its
admissibility to obtain a $\RR'$-preproof. If this method translates all
$\RR$-proofs into $\RR'$-proofs, $f \colon \RR \to \RR'$ is considered a
proof morphism.

\begin{definition}
  Let $(\Seq, \RR, \rho, \SC)$ and $(\Seq', \RR', \rho', \SC')$ be cyclic proof
  systems. A \emph{preproof morphism} $f \colon \RR \to \RR'$ consists of a
  function $f_0 \colon \Seq \to \Seq'$ mapping $\RR$-sequents to $\RR'$-sequents and
  a function $f_1 \colon \RR \to \Pp(\RR')$ assigning to each $\RR$-rule a
  $\RR'$-preproof. Furthermore, these two functions must agree: For $R \in
  \RR$ with $\rho(R) = (\Gamma, \Delta_1, \ldots, \Delta_n)$, the preproof
  $f_1(R)$ must have $f_0(\Gamma)$ as its endsequent and $f_0(\Delta_1),
  \ldots, f_0(\Delta_n)$ as its assumptions. Henceforth, we denote both $f_0 \colon
  \Seq \to \Seq'$ and $f_1 \colon \RR \to \Pp(\RR')$ by $f$.
\end{definition}

It is easiest formally describe the method for translating $\RR$-preproofs into
$\RR'$-preproofs induced by $f \colon \RR \to \RR'$ in terms of preproof composition.
Thus, suppose $\RR$ was a cyclic derivation system and $\Pi = ((C, \beta), \lambda,
\delta)$ was an $\RR$-preproof with open leaves $o_1, \ldots, o_n$. Furthermore,
suppose there were $\RR$-preproofs $\Pi_1=((T_1, \beta_1), \lambda_1, \rho_1),
\ldots, \Pi_n = ((T_n, \beta_n), \lambda_n, \rho_n)$ such that the endsequent of
$\Pi_i$ is $\lambda(o_i)$ and its assumptions are $\Xi^i_1, \ldots, \Xi^i_{m_i}$.
Then one may compose this material into a preproof $\Pi[\Pi_1, \ldots, \Pi_n] = ((T_c,
\beta_c), \lambda_c, \delta_c)$ with endsequent $\Gamma$ and assumptions 
$\Xi^1_1, \ldots \Xi^1_{m_1}, \Xi^2_1, \ldots,
\Xi^2_{m_2}, \ldots, \Xi^n_{1}, \ldots, \Xi^n_{m_n}$ as follows:
\begin{align*}
  T_c & \coloneq T \cup \bigcup_{i = 1}^n \{l_i t ~|~ t \in T_i\} \\
  \lambda_c(t) & \coloneq
                 \begin{cases}
                   \lambda(t) & t \in T \\
                   \lambda_i(s) & t = o_i s \text{ for } s \in T_i
                 \end{cases} \\
  \delta_c(t) & \coloneq
                \begin{cases}
                  \delta(t) & t \in \dom(\delta) \\
                  \delta_i(s) & t = o_i s \text{ for } s \in T_i \text{ and } s \in \dom(\delta_i)
                \end{cases} \\
  \beta_c(t) & \coloneq
                \begin{cases}
                  \beta(t) & t \in \dom(\beta) \\
                  o_i \beta_i(s) & t = o_i s \text{ for } s \in T_i \text{ and } s \in \dom(\beta_i)
                \end{cases}
\end{align*}

Suppose there was a preproof morphism $f \colon \RR \to \RR'$ and a $\RR$-preproof $\Pi =
((T, \beta), \lambda, \delta)$ with endsequent $\Gamma$ and assumptions $\Delta_1, \ldots,
\Delta_n$. This induces an $\RR'$-preproof $f(\Pi)$ with endsequent $f(\Gamma)$ and
assumptions $f(\Delta_1), \ldots, f(\Delta_n)$.
It is defined recursively on
$T$ as by associating to each
node $t \in T$ a preproof $\Pi_t$ of $\lambda(t)$. Start setting for each $t \in
\Leaf(T) \setminus \dom(\delta)$ the preproof $\Pi_t \coloneq
(\{\varepsilon\}, \varepsilon \mapsto \lambda(t), \emptyset)$ i.e. the
preproof deriving $\lambda(t)$ as an open leaf. Now for each $t \in \dom(\delta)$ such
that all $\{t_1, \ldots, t_n\} = \Chld(t)$ have associated preproofs
$\Pi_{t_i}$, define $\Pi'_t \coloneq \Pi_{\delta(t)}[\Pi_{t_1}, \ldots, \Pi_{t_n}]$
(where $\Pi_{\delta(t)}$ is given by the morphism $f \colon \RR \to \RR'$). If $t
\not\in \im(\beta)$ then $\Pi_t = \Pi'_t$. Otherwise, $\Pi_t$ is obtained from
$\Pi_t'$ by adding $\beta$-cycles from each open leaf of $\Pi_t'$ corresponding to
a leaf in $\beta^{-1}(t)$ to $\varepsilon$. Then
$f(\Pi) \coloneq \Pi_{\varepsilon}$.

\begin{definition}
  Let $(\Seq, \RR, \rho, \SC)$ and $(\Seq', \RR', \rho', \SC')$ be cyclic proof
  systems. A preproof morphism $f \colon \RR \to \RR'$ is a \emph{proof morphism} if it
  preserves the soundness condition of $\RR$. That is, if for every $\Pi \in \SC$
  one has $f(\Pi) \in \SC'$.
\end{definition}

\subsection{Trace Categories}
\label{sec:gtc}

In the previous section, we left the soundness conditions quite vague. In this
section we describe one kind of soundness condition: the global trace condition.
More specifically, we describe a generic way of specifying the global trace
condition in terms of certain categories. We then go on to define a family of
such categories that are sufficient to specify most global trace conditions from
the literature. The definitions we give in this
section are adapted from~\cite{afshariAbstractCyclicProofs2022}
and~\cite{wehrAbstractFrameworkAnalysis2021}.

Denote by $\omega$ preorder category (semi-category) induced by $\omega$
ordered by $\leq$ ($\omega$ ordered by $<$), writing $n < m$ for their
(non-identity) morphisms. Fix a category $\TT$.
A \emph{path} through $\TT$ is a functor $P \colon \omega \to \TT$.
Given paths $P, P' \colon \omega \to \TT$ one calls $P$ a \emph{subpath} of $P'$,
written  $P \subseteq P'$, if there is a semi-functor $S \colon \omega \to
\omega$ (i.e. a strictly monotone map $S \colon \omega \to \omega$) such that $P = P'
\circ S$. In other words, a subpath $P$ of $P'$ may (a) `drop' a finite prefix
of $P'$ (e.g. $P(0) = P'(k)$) and (b) compose multiple `steps' of $P'$ (e.g.
$P(i < i + 1) = P'(j + n - 1 < j + n) \circ \ldots \circ P'(j < j + 1)$).
The transitive, symmetric closure of $\subseteq$ is denoted $\sim$.

A trace category is a category with a condition on paths which is invariant
under subpaths. This general notion captures most notions of `trace' found in
the literature of cyclic proof theory.

\begin{definition}\label{def:paths}
  A \emph{trace category} is a category $\TT$ together with a condition \( C_\TT
  \) on paths, called the \emph{trace condition}, invariant under $\sim$, i.e.
  if $P \sim P'$ then $C_\TT(P)$ if and only if $C_\TT(P')$ holds.
\end{definition}

A trace interpretation specifies notion of trace and progress of a derivation
system $\RR$ in terms of a trace category $\TT$. The sequents of $\RR$ are
identified with objects of $\TT$ and each `step' from a conclusion to a premise
in a derivation rule of $\RR$ is associated with a morphism between the objects
associated to said conclusion and premise. Under this interpretation, every
branch through a preproof induces a path $\omega \to \TT$ which allows for a
general specification of the global trace condition in terms of the trace
condition of $\TT$. Multiple examples of such trace interpretations for
cyclic proof systems from the literature can be found in \Cref{sec:concrete} and
in \cite[Chapter 6]{wehrAbstractFrameworkAnalysis2021}.

\begin{definition}
  Let $(\Seq, \RR, \rho)$ be a derivation system and $\TT$ a trace category. A
  \emph{trace interpretation} $\iota \colon \RR \to \TT$ consists of a map $\iota \colon \Seq
  \to \Ob(\TT)$ and for each rule $r \in \RR$ with $\rho(r) = (\Gamma, \Delta_1,
  \ldots, \Delta_n)$ morphisms $r_i \colon \iota(\Gamma) \to \iota(\Delta_i)$ for each \( 1 \le i \le n \).
\end{definition}

Fix a cyclic tree $C = (T, \beta)$. A sequence $\pi \in T^\omega$ is a
\emph{branch} through $C$ if $\pi_0 = \varepsilon$ and it satisfies the
following properties at every index $i \in \omega$: (a) if $\pi_i \not\in
\Leaf(T)$ then $\pi_{i + 1} \in \Chld(\pi_i)$ and (b) if $\pi_i \in \Leaf(T)$
then $\pi_i \in \dom(\beta)$ and $\pi_{i + 1} = \beta(\pi_i)$.

\begin{definition}\label{def:tc-induced}
  Let $(\Seq, \RR, \rho)$ be a derivation system with a trace interpretation 
  \( \iota \colon \RR \to \TT\). Let $\Pi = (C, \lambda, \delta)$ be a preproof and $\pi$ be a a path
  through $C$. This induces a path $\hat{\pi} \colon \omega \to \TT$ given by
  \[\hat{\pi}(i) \coloneq \iota(\lambda(\pi_i)) \qquad \hat{\pi}(i < i + 1) \coloneq
    \begin{cases}
      \delta(\pi_i)_j \colon \iota(\lambda(\pi_i)) \to \iota(\lambda(\pi_{i + 1})) & \pi_i
      \not\in \Leaf(T) \text{ and } \pi_{i + 1} = \pi_{i + 1} j \\
      1_{\hat{\pi}(i)} & \pi_i \in \dom(\beta)
    \end{cases}
  \]

  This induces a cyclic proof system $\iota(\RR) \coloneq (\Seq, \RR, \rho,
  \SC)$ with
  \[\SC \coloneq \{\Pi \in \Pp(\RR) ~|~ \text{ for every path } \pi \text{ through
    } \Pi \text{ the trace condition } C_\TT(\hat{\pi}) \text{ holds}\}\]
\end{definition}

The kind of soundness condition described in the previous definition is a
\emph{global trace condition}: Prima facie, one needs to check whether \emph{every}
infinite branch --- of which there are continuum many in nontrivial cases ---
satisfies the trace condition of $\TT$. In `sufficiently finitary' instances,
including essentially all those in the cyclic proof theory literature, the scope
of this verification can be restricted to a finite set of periodic paths via
Ramsey's theorem (see e.g.~\cite[Theorem 3]{afshariAbstractCyclicProofs2022}). Even then, this is a complex
verification process whose complexity quickly exceeds the human capabilities for
checking proofs. \emph{Reset proof systems} are cyclic proof systems with a
different kind of soundness condition. While we are not aware of an abstract
account of the reset soundness condition analogous to \Cref{def:tc-induced},
reset proof systems nonetheless share recognizable features. Reset conditions
generally work by `tracking progress' using an annotation mechanism for
sequents. The condition requires that a \emph{reset rule}, which resets some of
this tracked progress, is applied on each simple cycle, i.e. along the path
between $\beta(s)$ and $s$ for every $s \in \dom(\beta)$ in a preproof. To
verify that a preproof satisfies such a reset condition, it thus suffices to
analyze each such simple cycle individually, giving rise to a `local' soundness
condition as opposed to the global trace condition defined above. However, reset
proof systems often require much larger proof than proof systems with a global
trace condition (see \Cref{lem:complexity} for a somewhat general account of this).



We continue by giving a family of trace categories $\TT_\Aa$ induced by
activation algebras. Most global trace conditions found in the literature can be
specified in terms of $\TT_\Aa$ for a suitable algebra $\Aa$.

\begin{definition}\label{def:act-algeb}
  An \emph{activation algebra} $\Aa = (A, \leq, \vee, 0, \alpha)$ is a
  finite semilattice $(A, \leq, \vee, 0)$ together with a fixed \emph{activation
    element} $\alpha \in A$ where $0 \neq \alpha$.

  The \emph{$\Aa$-activated trace \emph{category $\TT_\Aa$}} has the finite sets
  as its objects. The
  morphisms between sets $X, Y$ are all relations $R \subseteq X \times \Aa
  \times Y$. Given morphisms $R \colon X \to Y, R' \colon Y \to Z$ their composition is specified by
  \[(x, c, z) \in R' \circ R \qquad \text{ iff } \qquad \exists y \in Y. \exists a, b
    \in A.~(x, a, y)
    \in R,~ (y, b, z) \in R' \text{ and } a \vee b = c\]
  The identity morphisms are $1_X \coloneq \{(x, 0, x) ~|~ x \in X\}$. We often write $x R^a y$ to mean $(x, a, y) \in R$.

  The trace condition of \( \TT_\Aa \) is defined as follows:
  A path $P \colon \omega \to \TT_\Aa$
  satisfies the \emph{trace condition} if there exists a subpath $P' \subseteq
  P$ and an infinite sequence $\sigma$, with $\sigma(i) \in P'(i)$ for each $i
  \in \omega$, along it such that $\sigma_i P'(i < i + 1)^\alpha \sigma_{i + 1}$
  for all $i \in \omega$.
\end{definition}

\begin{example}
  The \emph{booleans} $\BB \coloneq \{0, 1\}$ form an activation algebra with the usual
  join-operation and $\alpha \coloneq 1$. They correspond to the to the most
  common style of global trace conditions in the literature: traces have
  progress points (represented by triples $(x, 1, y)$ in maps of $\TT_\BB$) and a
  path satisfies the trace condition if it has infinitely many progress points.
  It is easily verified that the trace condition of $\TT_\BB$ is precisely this.

  Another example of an activation algebra is the three value \emph{failure
    algebra} $\FF \coloneq \{0, 1, 2\}$ with $\max$ as its join operation and
  $\alpha \coloneq 1$. The trace condition of $\TT_\FF$ corresponds to global
  trace conditions under which traces satisfy the trace condition if they have
  infinitely many progress points (triples $(x, 1, y)$ in maps of $\TT_\FF$) and
  \emph{no} failure points (triples $(x, 2, y)$ in maps of $\TT_\FF$). Again,
  the trace condition of $\TT_\FF$ ensures precisely this condition. The failure
  algebra appears in the literature as one of the common trace conditions for
  the modal $\mu$-calculus (see \Cref{def:ff-mu-tc}).
\end{example}

The trace categories $\TT_\Aa$ are a natural medium for the study of cyclic
proof theory. They are abstract enough to capture many trace conditions from the
literature but also concrete enough to allow various theorems of cyclic proof
theory to be derived for them, such as the decidability result below. In this
article, we show how to construct reset proof systems for precisely the cyclic
proof systems $\iota(\RR)$ for trace interpretations $\iota \colon \RR \to \TT_\Aa$.

\begin{proposition}\label{lem:tc-dec}
  Fix a derivation system $\RR$ and an $\RR$-preproof $\Pi$. Given a trace
  interpretation $\iota \colon \RR \to \TT_\Aa$, it is decidable whether $\Pi$ is a
  proof in $\iota(\RR)$.
\end{proposition}
\begin{proof}
  There are various ways of proving this. For example by appealing to infinite
  word automata~\cite[Theorem 4.4]{wehrAbstractFrameworkAnalysis2021} or to
  Ramsey's theorem~\cite[Theorem 4.14]{wehrAbstractFrameworkAnalysis2021}.
\end{proof}

\subsection{Automata Theory}
\label{sec:automata}

The theory of infinite word and tree automata has always served the role of an
important tool in cyclic proof theory. In this regard, this article is no
exception: The notion of Safra boards (\Cref{sec:safra-boards}) central to our
construction of reset systems is based on the Safra tree
construction~\cite{kozenSafraConstruction2006} developed for the efficient
determinisation of certain infinite word automata. The completeness proofs we
give in this article also crucially rely on a theorem about the inhabitation of
languages described by infinite tree automata (\Cref{lem:rabin-det}).

We begin by recalling some notions of infinite word automata.
  A \emph{Büchi automaton} is a tuple $\SB = (Q, \Sigma, \Delta, S, F)$
  where $Q$ is a finite set of \emph{states}, $\Sigma$ is a finite \emph{alphabet}, $S
  \subseteq Q$ is the set of \emph{starting states}, the relation $\Delta \subseteq Q \times \Sigma
  \times Q$ is the \emph{transition relation} and $F \subseteq Q$ is
  the \emph{acceptance condition}.
  Given a word $\sigma \in \Sigma^\omega$, the sequence $\rho \in Q^\omega$
  is called a \emph{\emph{run of $\SB$} on $\sigma$} if $\rho_0 \in S$ and for each $i \in \omega$ one
  has $(\rho_i, \sigma_i, \rho_{i + 1}) \in \Delta$. A run $\rho$ is
  \emph{accepting} if there is some $q \in F$ such that $\rho_i = q$ infinitely
  often. A word $\sigma$ \emph{is accepted by $\SB$} if
  there exists an accepting run of $\SB$ on $\sigma$.
  The set $L(\SB) \coloneq
  \{\sigma \in \Sigma^\omega ~|~ \sigma \text{ is accepted by } \SB\}$ is the \emph{language of
    $\SB$}.

An important result connecting the theories of cyclic proof theory and infinite
word automata is that the branches satisfying the trace conditions of many
cyclic proof systems from the literature can be recognized by certain infinite
word automata.
  Fix any trace category $\TT$. Its trace condition is
  \emph{Büchi-recognizable} if, for any finite set $M$ of morphisms of
  $\TT$, there exists a Büchi-automaton
  $\SB$ such that $L(\SB)$ is the set $T(M)$ below.
  \[T(M) \coloneq \{\tau \in M^\omega ~|~ P(i < i + 1) \coloneq \tau_i
    \text{ is a valid path and satisfies the trace condition of } \TT\}\]
A general construction for such recognizing automata can be given in the setting
of $\TT_\Aa$. This construction can be used to prove \Cref{lem:tc-dec}. The
construction below is a variant of that given in \cite[Proposition
5.11]{wehrAbstractFrameworkAnalysis2021} which also given a proof of
\Cref{lem:bam}.

\begin{definition}
  Let $\Aa$ be an activation algebra, $M$ be a finite set of morphisms of
  $\TT_\Aa$ and fix $O \coloneq \bigcup_{\tau \in M}\{\dom(\tau), \cod(\tau)\}$. The Büchi-automaton
  $\SB(\Aa, M) = (M, Q, \Delta, O, F)$ is defined below,
  fixing some arbitrary $0^* \not\in \Aa$.
  \begin{align*}
    Q &\coloneq O \cup \{(X, x, a) ~|~ X \in O, x \in X, a \in \Aa \cup \{0^*\} \} \\
    \Delta &\coloneq \{(X, R \colon X \to Y, Y) ~|~ R \colon X \to Y \in M\} \\
    &\;\;\cup \{(X, R \colon X \to Y, (Y, y, 0)) ~|~ Y \in O, x \in Y\} \\
    &\;\;\cup \{((X, x, a), R \colon X \to Y, (Y, y, a \vee b)) ~|~ x R^{b} y, a \vee b \neq \alpha, R \in M \} \\
    &\;\;\cup \{((X, x, a), R \colon X \to Y, (Y, y, 0^*)) ~|~ x R^{b} y, a \vee b = \alpha, R \in M \} \\
    &\;\;\cup \{((X, x, 0^*), R \colon X \to Y, (Y, y, a)) ~|~ x R^{a} y, R \in M, a \neq \alpha \} \\
    &\;\;\cup \{((X, x, 0^*), R \colon X \to Y, (Y, y, 0^*)) ~|~ x R^{a} y, R \in M, a = \alpha \} \\
    F &\coloneq \{(X, x, 0^*) ~|~ X \in O, x \in X\}
  \end{align*}
\end{definition}

\begin{proposition}\label{lem:bam}
  For any $\Aa$, the trace condition of $\TT_\Aa$ is Büchi-recognizable.
  Moreover, for any set $M$ of morphisms of $\TT_\Aa$ one has $L(\SB(\Aa, M)) = T(M)$.
\end{proposition}

For an alphabet $\Sigma$, a \emph{$\Sigma$-labeled tree} is a pair $(T,
\lambda \colon T \to \Sigma)$ for a, possibly infinite, tree $T$. A $\Sigma$-labeled
tree $(T, \lambda)$ is a \emph{subtree} of $\Sigma$-labeled $(T', \lambda')$ if
is a `suffix' of $T'$, i.e.
there exists some $t \in T'$ such that $T = \{ts \in T' ~|~ s \in T'\}$ and
$\lambda(s) = \lambda'(ts)$.
A \emph{Rabin tree automaton} is a tuple $\SA = (\Sigma, Q, \Delta, s, R)$ consisting
of a finite alphabet $\Sigma$, a set of \emph{states} $Q$, a set of \emph{transitions}
$\Delta \subseteq Q \times \Sigma \times Q^*$, a \emph{starting state} $s \in Q$ and an \emph{acceptance condition}
  $R = \{(G_0, B_0), \ldots (G_n, B_n)\}$ where $G_i \cap B_i = \emptyset$
  and $G_i \cup B_i \subseteq Q$.
Let $(T, \lambda)$ be a $\Sigma$-labeled tree. A \emph{run} of $\SA$ on $(T,
\lambda)$ is a $Q$-labeling $\rho \colon T \to Q$ of $T$ such that $\rho(\varepsilon)
= s$ and for each $t \in T$ with $\Chld(t) = \{t0, \ldots, tn\}$ the transition
$(\rho(t), \lambda(t), \rho(t0), \ldots, \rho(tn)) \in \Delta$. A run is
\emph{accepting} if for every infinite branch $b \in T^\omega$ of $T$ there
exists $(G, B) \in R$ such that $\rho(b_i) \in G$ for infinitely many $i \in
\omega$ and $\rho(b_i) \in B$ for only finitely many $i \in \omega$. A
$\Sigma$-labeled tree $(T, \lambda)$ is \emph{accepted} by $\SA$ if there is an
accepting run of $\SA$ on it. The set $L(\SA) \coloneq \{(T, \lambda \colon T \to
\Sigma) ~|~ (T, \lambda) \text{ is accepted by } \SA\}$ is the \emph{language of
$\SA$}.

The following a corollary of the memoryless determinacy of Rabin games; see,
e.g.~\cite{pitermanFasterSolutionsRabin2006}.
\begin{proposition}\label{lem:rabin-det}
  If $\SA$ is a Rabin tree automaton with non-empty language then
  \( \SA \) accepts a regular tree via a regular run.
\end{proposition}

\section{Safra Boards}
\label{sec:safra-boards}

This section introduces Safra boards, a variant of the tree construction
introduced by Safra~\cite{safraComplexityOmegaAutomata1988} to determinise Büchi
automata. Our presentation of Safra boards has been adapted specifically to the
automata $\SB(\Aa, M)$ or, equivalently, the trace condition of $\TT_\Aa$.
Inspired by Kozen's account of Safra automata~\cite{kozenSafraConstruction2006}, we
present the construction in terms of boards with stacks of chips on them rather
than trees. Safra boards can recognize whether a
sequence $\tau$ of morphisms is a path satisfying the trace condition, similar
to the automata $\SB(\Aa, M)$. They serve as building blocks of the abstract
cyclic reset proofs presented in \Cref{sec:acd_pwn}.

For the following definitions fix some
countable set $\CC$ with $\omega \subseteq \CC$, which we call the set of \emph{chips}.

\begin{definition}\label{def:safra-board}
  A \emph{Safra board} on an activation algebra $\Aa$ and a set $X \in \Ob(\TT_\Aa)$ is a tuple
  $(\Theta, \sigma)$ consisting of a \emph{control} $\Theta$, a finite linear order
  $(\Theta, \leq)$ on a set $\Theta \subset \CC$, and a map $\sigma \colon X \times \Aa \to
  \pow(\pow(\Theta))$. Furthermore, it is required that for
  every $\gamma \in \Theta$ there are $a \in \Aa$ and $x \in X$ such that
  $\gamma \in S \in \sigma(x, a)$. 
  Elements of \( \Theta \) are called \emph{chips}.
The sets $S \in \sigma(x, a)$ represent
  \emph{stacks} of chips with their $\leq$-least element the bottom
  and their $\leq$-greatest on top.

  A chip $\gamma \in \Theta$ is
  \emph{covered} if for all $x \in X$ and $a \in A$, \( \gamma \) is not on top of any $S \in
  \sigma(x, a)$.

  The stacks of chips in any given control $\Theta$ are \emph{linearly ordered}
  by the relation $S <_\Theta S'$ which holds iff $S$
  contains the $\leq$-least element of the symmetric difference $S \Delta S' = ( S \setminus S' ) \cup ( S' \setminus S )$.

  We write \emph{$\Sb(\Aa, X)$} for the set of Safra boards on $\Aa$ and $X \in \Ob(\TT_\Aa)$.
\end{definition}

Similar to the automaton $\SB(\Aa, M)$, Safra boards give rise to a state
transition system with a notion of `accepting run' which recognizes sequences
$\tau \in M^\omega$ describing paths through $\TT_\Aa$ which satisfy the trace
condition. The transitions the shape $(\Theta,
\sigma) \step{X} (\Theta', \sigma')$: from Safra board to Safra board. Here, the
letter $X$ denotes the type of transition, of which there are five:
$\tau$-successors, weakenings, thinnings, $\gamma$-resets and populations. We
proceed by defining each kind of transition.

\begin{definition}\label{def:t-succ}
  Let $(\Theta, \sigma)$ be a Safra board on $X \in \Ob(\TT_\Aa)$ and let $\tau
  \colon X \to Y$ be a morphism of $\TT_\Aa$. The \emph{$\tau$-successor} of
  $(\Theta, \sigma)$ is a Safra board $(\Theta', \sigma')$ on $Y$, which is
  obtained in two steps:
  \begin{description}
  \item[Move] Move all of the stacks around the board according to $\tau$ to
    obtain the intermediate board $(\Theta^*, \sigma^*)$ on $Y$ as follows:
    \[\sigma^*(y, a) \coloneq \{S \in \sigma(x, b) ~|~ \exists c \in \Aa.~(x, c,
      y) \in \tau \text{ and } a = b \vee c\}\]
    where $\Theta^* \coloneq \{\gamma \in \Theta ~|~ \exists y \in Y, a \in \Aa,
    S \in \sigma^*(y, a) ~|~\gamma \in S\}$ is a suborder of $\Theta$.
  \item[Cover] Cover all stacks that have landed on $\alpha$. First, fix some
    linearly ordered set $\Theta^\circ \subset \CC \setminus \Theta$ and bijection $\iota \colon \{y \in Y ~|~
    \sigma^*(y, \alpha) \neq \emptyset\} \simeq \Theta^\circ$. Then set
    \[\sigma'(y, a) \coloneq
      \begin{cases}
        \emptyset & a = \alpha \\
        \sigma^*(y, a) \cup \{S \cup \{\iota(y)\} ~|~ S \in \sigma^*(y, 
        \alpha)\} & a = 0 \\
        \sigma^*(y, a) & \text{otherwise}
      \end{cases}
    \]
    Now fix $\Theta' \coloneq \Theta^* \oplus \Theta^\circ$ where $\oplus$ denotes
    the concatenation of linear orders.
  \end{description}
  We write $(\Theta, \sigma)
  \step{\tau} (\Theta', \sigma')$ to signal that $(\Theta', \sigma')$ is a
  $\tau$-successor of $(\Theta, \sigma)$.
\end{definition}

\begin{example}\label{ex:t-succ}
  For an example, denote by $\FF$ the three-value \textit{failure algebra}
   $(\{0, 1, 2\}, \leq, \vee, 0, 1)$ and the set
  $\{w, x, y, z\} \in \Ob(\TT_\FF)$. A Safra board in $\Sb(\FF, X)$ may be
  thought of as a square game board, akin to a chess board, as pictured in
  \Cref{fig:t-succ}. Indeed, \Cref{fig:t-succ} gives an example of a
  $\tau$-successor transition for $\tau \coloneq \{(x, 1, x), (x, 1, y), (y, 0, y), (y, 1, y) (z, 2,
  z)\}$ and $\Theta \coloneq \{a, b, c, d, e\}$ (ordered alphabetically) and
  \[\sigma(w, 0) \coloneq \{\{a\}\} \quad
    \sigma(x, 0) \coloneq \{\{b\}\} \quad
    \sigma(y, 0) \coloneq \{\{c, d\}\} \quad
    \sigma(z, 0) \coloneq \{\{e\}\}
  \]
  where $\sigma(u, v) \coloneq \emptyset$ for $u \in X$ and $v \in \{1, 2\}$, as
  pictured in \Cref{fig:t-succ-a}. To obtain one of its $\tau$-successors, one
  first needs to carry out the \textsc{Move}-step, moving the stacks on the
  board according to $\tau$. The board $(\Theta^*, \sigma^*)$ resulting from
  this step is pictured in \Cref{fig:t-succ-b}. Note that the stack from $(y,
  0)$ was both moved to $(y, 1)$ and stayed on $(y, 0)$ and furthermore that the stack on $(w, 0)$
  was removed, as there is no trace triplet for $w$ in $\tau$. Furthermore,
  observe that the \textsc{Move}-step is fully deterministic for a fixed board
  $(\Theta, \sigma)$ and morphism $\tau$. To obtain $(\Theta', \sigma')$, one
  needs to carry out the \textsc{Cover}-step: The stack on each $(u, \alpha)$ in
  $(\Theta^*, \sigma^*)$ need to be moved back to $(u, 0)$. To mark that
  $\alpha$ has been attained, a new chip (which was not present in $\Theta$) is
  placed on each such stack that was moved back. If multiple stacks are moved
  from $(v, \alpha)$ to $(v, 0)$, the same chip is placed on top of each.
  In this case, as $\alpha = 1$, one moves the stacks on $(x, 1)$ and $(y, 1)$,
  introducing the new chips $g$ and $h$.
  \begin{figure}[h]
    \centering
    \begin{subfigure}[b]{0.3\textwidth}
      \centering
      \begin{tabular}{|c|c|c|c|}
        \hline
        $X \setminus \Aa$& 0 & 1 & 2 \\
        \hline
        w & a & & \\
        \hline
        x & b & & \\
        \hline
        y & cd & & \\
        \hline
        z & e & & \\
        \hline
      \end{tabular}
      \caption{Board $(\Theta, \sigma)$}
      \label{fig:t-succ-a}
    \end{subfigure}
    \begin{subfigure}[b]{0.3\textwidth}
      \centering
      \begin{tabular}{|c|c|c|c|}
        \hline
        $X \setminus \Aa$& 0 & 1 & 2 \\
        \hline
        w & & & \\
        \hline
        x & & b & \\
        \hline
        y & c & b, cd & \\
        \hline
        z & & & e \\
        \hline
      \end{tabular}
      \caption{Board $(\Theta^*, \sigma^*)$}
      \label{fig:t-succ-b}
    \end{subfigure}
    \begin{subfigure}[b]{0.3\textwidth}
      \centering
      \begin{tabular}{|c|c|c|c|}
        \hline
        $X \setminus \Aa$& 0 & 1 & 2 \\
        \hline
        w & & & \\
        \hline
        x & bg & & \\
        \hline
        y & c, bh, cdh & & \\
        \hline
        z & & & e \\
        \hline
      \end{tabular}
      \caption{Board $(\Theta', \sigma')$}
      \label{fig:t-succ-c}
    \end{subfigure}
    
    \caption{Example of $(\Theta, \sigma) \step{\tau} (\Theta', \sigma')$}
    \label{fig:t-succ}
  \end{figure}
\end{example}

\begin{definition}\label{def:weakening}
  Let $(\Theta, \sigma)$ be a Safra board on $X \in \Ob(\TT_\Aa)$. Another
  Safra board $(\Theta', \sigma')$ on $X$ is a \emph{weakening} of $(\Theta, \sigma)$
  if $\sigma'(x, a) \subseteq \sigma(x, a)$ for every $x \in X$ and $a \in \Aa$.
  Furthermore, it is required that $\Theta' \subseteq \Theta$ is such that every $\gamma
  \in \Theta'$ occurs in some $S \in \sigma'(x, a)$ for some $x \in X, a \in \Aa$.
  We write $(\Theta, \sigma)
  \step{W} (\Theta', \sigma')$ to express that the latter board is a weakening
  of the former.

  The \emph{thinning} of $(\Theta, \sigma)$ is the special weakening $(\Theta',
  \sigma')$ induced by
  \(\sigma'(x, a) \coloneq \{\min_{<_\Theta}\,\sigma(x, a)\}\) if $\sigma(x, a)
  \neq \emptyset$ and $\sigma'(x, a) = \emptyset$ otherwise.
  We write $(\Theta, \sigma) \step{T} (\Theta',
  \sigma')$ to express that the latter board is the thinning of the former.
\end{definition}

\begin{example}\label{ex:thin}
  \Cref{fig:thin} pictures the result of a thinning transition starting from the
  result of \Cref{ex:t-succ}.
  When multiple stacks are present on a space on the board, a thinning removes
  all but the $<_\Theta$-least.
  On the board in \Cref{fig:thin-a}, the thinning thus modifies only $\sigma(y, 0)$.
  Observe that $c <_\Theta cdh$ as $\{c\} \Delta \{c, d, h\} = \{d, h\}$, meaning $cdh$
  contains the $\Theta$-least element of the symmetric difference. Indeed,
  whenever $S \subset S'$ for two stacks on $\Theta$, one has $S' <_\Theta S$.
  Secondly, $bh <_\Theta cdh$ as $bh$ contains $b$ and $cdh$ does not.
  \begin{figure}[h]
    \centering
    
    \begin{subfigure}[b]{0.45\textwidth}
      \centering
      \begin{tabular}{|c|c|c|c|}
        \hline
        $X \setminus \Aa$& 0 & 1 & 2 \\
        \hline
        w & & & \\
        \hline
        x & bg & & \\
        \hline
        y & c, bh, cdh & & \\
        \hline
        z & & & e \\
        \hline
      \end{tabular}
      \caption{Board $(\Theta, \sigma)$}
      \label{fig:thin-a}
    \end{subfigure}
    \begin{subfigure}[b]{0.45\textwidth}
      \centering
      \begin{tabular}{|c|c|c|c|}
        \hline
        $X \setminus \Aa$& 0 & 1 & 2 \\
        \hline
        w & & & \\
        \hline
        x & bg & & \\
        \hline
        y & bh & & \\
        \hline
        z & & & e \\
        \hline
      \end{tabular}
      \caption{Board $(\Theta', \sigma')$}
      \label{fig:thin-b}
    \end{subfigure}
    \caption{Example of $(\Theta, \sigma) \step{T} (\Theta', \sigma')$}
    \label{fig:thin}
  \end{figure}
\end{example}

\begin{definition}\label{def:reset}
  Let $(\Theta, \sigma)$ be a Safra board on $X \in \Ob(\TT_\Aa)$ and let
  $\gamma \in \Theta$ be covered. The \emph{$\gamma$-reset} \emph{$\sreset{S}{\gamma}$} of a stack
  $S$ is defined as 
  \[
    \sreset{S}{\gamma} \coloneq
    \begin{cases}
      \{z \in S ~|~ z \leq \gamma\} & \text{ if } \gamma \in S \\
      S & \text{ otherwise}
    \end{cases}
  \]
  Then the \emph{$\gamma$-reset} of $(\Theta, \sigma)$ is $(\Theta', \sigma')$
  where
  \[\sigma'(x, a) \coloneq \{\sreset{S}{\gamma} ~|~ S \in \sigma(x, a)\} \qquad
    \Theta' \coloneq \{\gamma \in \Theta ~|~ \exists x \in X, a \in \Aa, S \in
    \sigma'(x, a).~\gamma \in S\}\]
  and we write $(\Theta, \sigma) \step{R_\gamma} (\Theta', \sigma')$ to express this.
\end{definition}
\begin{example}\label{ex:reset}
  Pictured in \Cref{fig:reset} is the result of a $b$-reset applied to the
  resulting board from \Cref{ex:thin}. Note that a reset on $e, g$ or $h$ would
  not be possible on that board as none of them are covered.

  \begin{figure}[h]
    \centering
    
    \begin{subfigure}[b]{0.45\textwidth}
      \centering
      \begin{tabular}{|c|c|c|c|}
        \hline
        $X \setminus \Aa$& 0 & 1 & 2 \\
        \hline
        w & & & \\
        \hline
        x & bg & & \\
        \hline
        y & bh & & \\
        \hline
        z & & & e \\
        \hline
      \end{tabular}
      \caption{Board $(\Theta, \sigma)$}
      \label{fig:reset-a}
    \end{subfigure}
    \begin{subfigure}[b]{0.45\textwidth}
      \centering
      \begin{tabular}{|c|c|c|c|}
        \hline
        $X \setminus \Aa$& 0 & 1 & 2 \\
        \hline
        w & & & \\
        \hline
        x & b & & \\
        \hline
        y & b & & \\
        \hline
        z & & & e \\
        \hline
      \end{tabular}
      \caption{Board $(\Theta', \sigma')$}
      \label{fig:reset-b}
    \end{subfigure}
    \caption{Example of $(\Theta, \sigma) \step{R_b} (\Theta', \sigma')$}
    \label{fig:reset}
  \end{figure}
\end{example}

\begin{definition}\label{def:populate}
  Let $(\Theta, \sigma)$ be a Safra board on $X \in \Ob(\TT_\Aa)$.
  The board $(\Theta, \sigma')$ is a \emph{population} of
  $(\Theta, \sigma)$, denoted by $(\Theta, \sigma)
  \step{P} (\Theta, \sigma')$, if for each $x \in X$ one has
  $\sigma(x, 0) \subseteq \sigma'(x, 0) \subseteq \sigma(x, 0) \cup
  \{\emptyset \}$ and $\sigma'(x, a) = \sigma(x, a)$ for all $a \in \Aa
  \setminus \{0\}$.
\end{definition}

\begin{example}\label{ex:populate}
  Pictured in \Cref{fig:populate} is the result of a population transition on
  the board resulting from \Cref{ex:reset}. Here, the new chip $f$ has been
  added to $(w, 0)$ and $(x, 0)$. It would also have been legal to additionally
  add it to $(y, 0)$ and $(z, 0)$.
  \begin{figure}[h]
    \centering
    
    \begin{subfigure}[b]{0.45\textwidth}
      \centering
      \begin{tabular}{|c|c|c|c|}
        \hline
        $X \setminus \Aa$& 0 & 1 & 2 \\
        \hline
        w & & & \\
        \hline
        x & b & & \\
        \hline
        y & b & & \\
        \hline
        z & & & e \\
        \hline
      \end{tabular}
      \caption{Board $(\Theta, \sigma)$}
      \label{fig:populate-a}
    \end{subfigure}
    \begin{subfigure}[b]{0.45\textwidth}
      \centering
      \begin{tabular}{|c|c|c|c|}
        \hline
        $X \setminus \Aa$& 0 & 1 & 2 \\
        \hline
        w & f & & \\
        \hline
        x & b, f & & \\
        \hline
        y & b & & \\
        \hline
        z & & & e \\
        \hline
      \end{tabular}
      \caption{Board $(\Theta', \sigma')$}
      \label{fig:populate-b}
    \end{subfigure}
    \caption{Example of $(\Theta, \sigma) \step{P} (\Theta', \sigma')$}
    \label{fig:populate}
  \end{figure}
\end{example}

Safra board runs are sequences of the different kinds of transitions we have
defined. Importantly, for such a sequence to be considered a run on some $\tau
\in M^\omega$ it is crucial that it `consumes' all `letters' of $\tau$.

\begin{definition}
  Fix a set $M$ of morphisms of $\TT_\Aa$ and some $\tau \in M^\omega$.
  A sequence $(\Theta_i, \sigma_i)_{i \in \omega}$ of Safra boards is called a
  \emph{run of $\tau$} if there exists a strictly monotone function $\iota :
  \omega \to \omega$ and for every $i \in \omega$ either
  \begin{itemize}
  \item $i = \iota(n)$ for some $n \in \omega$ and $(\Theta_i, \sigma_i)
    \step{\tau_n} (\Theta_{i + im}, \sigma_{i + 1})$ or
  \item $i \neq \iota(n)$ and $(\Theta_i, \sigma_i) \step{W} (\Theta_{i + 1}, \sigma_{i + 1})$ or
  $(\Theta_i, \sigma_i) \step{P} (\Theta_{i + 1}, \sigma_{i + 1})$ or
  $(\Theta_i, \sigma_i) \step{R_\gamma} (\Theta_{i + 1}, \sigma_{i + 1})$
    for some $\gamma \in \Theta_i$
  \end{itemize}
  A run $(\Theta_i, \sigma_i)_{i \in \omega}$ is \emph{accepting} if there exists some $N$ and some $\gamma \in
  \bigcap_{N \leq n} \Theta_n$ such that infinitely many $\gamma$-resets take
  place along $(\Theta_i, \sigma_i)_{i \in \omega}$.
\end{definition}

There is a lot of leeway when constructing a Safra board run because of the many
different kinds of transitions that may be taken at any point it time (for
example, it is always possible to take a weakening transition which leaves
$\sigma$ unchanged). For some proofs in this article, it will prove useful
to be stricter about the ordering of transitions a long a run. This is
accomplished by the concept of greedy runs, runs whose ordering of transitions
is deterministic for any given $\tau$. Such runs are called
\textit{greedy} because it can be shown that whenever there exists an accepting
Safra board run of $\tau$, the greedy run of $\tau$ is accepting as well. In
many cases, it thus suffices to restrict ones attention to greedy runs. Dually,
when constructing runs, one may always follow the greedy construction strategy.
The concept of greedy runs is also closely linked to the runs on determinised
$\SB(\Aa, M)$ (see \Cref{def:safra-aut}).

\begin{definition}\label{def:greedy-step}
  Fix $X, Y \in \Ob(\TT_\Aa)$, some morphism $\tau \colon X \to Y$ and a Safra board
  $(\Theta, \sigma)$ on $X$. Then $(\Theta', \sigma')$ is the result of
  a \emph{greedy $\tau$-transition} from $(\Theta, \sigma)$, denoted by
  $(\Theta, \sigma) \gstep{\tau} (\Theta', \sigma')$, if
  \[(\Theta, \sigma) = (\Theta_0, \sigma_0)
    \step{R_{\gamma_k}} \ldots
    \step{R_{\gamma_1}} (\Theta_k, \sigma_k)
    \step{P} (\Theta_{k + 1}, \sigma_{k + 1})
    \step{\tau} (\Theta_{k + 2}, \sigma_{k + 2})
    \step{T} (\Theta_{k + 3}, \sigma_{k + 3})
    = (\Theta', \sigma')\]
  is the transition sequence produced according to the following
  instructions, starting at step 1.
  \begin{enumerate}
  \item If there exist covered chips $\gamma_1 < \ldots < \gamma_k$ in
    $(\Theta_0, \sigma_0)$ then perform $\gamma_i$-resets in descending order,
    that is:
    \[(\Theta_0, \sigma_0) \step{R_{\gamma_k}} (\Theta_1, \sigma_1)
      \step{R_{\gamma_{k - 1}}} \ldots
      \step{R_{\gamma_1}} (\Theta_k, \sigma_k)\]
    then continue with step 2.
  \item Continue with a population $(\Theta_k,
    \sigma_k) \step{P} (\Theta_{k + 1}, \sigma_{k + 1})$ in such a way that
    every $\sigma_i(x, 0) = \emptyset$ is populated to $\abs{\sigma_{i + 1}(x,
      0)} = 1$ and all other $\sigma_{i + 1}(x, a)$ remain unchanged.
    Continue with step 3.
  \item Carry out the $\tau$-transition $(\Theta_{k + 1}, \sigma_{k + 1}) \step{\tau}
    (\Theta_{k + 1}, \sigma_{k + 1})$ then continue with step 4.
  \item Carry out a thinning $(\Theta_{k + 2},
    \sigma_{k + 2}) \step{T} (\Theta_{k + 3}, \sigma_{k + 3})$.
  \end{enumerate}
  We write $(\Theta, \sigma)
  \gstep{\tau} (\Theta', \sigma')$ to denote the full transition sequence
  described above.
\end{definition}

\begin{definition}\label{def:greedy}
  A run $(\Theta_i, \sigma_i)_{i \in \omega}$ of some $\tau$ is \emph{greedy} if
  $\Theta_0 = \emptyset$ and $\sigma_0(x, a) = \emptyset$ and furthermore the
  run is a sequence of greedy $\tau_i$-transitions, i.e.
  \[(\Theta_0, \sigma_0) \gstep{\tau_0} (\Theta_{\iota(0) + 2}, \sigma_{\iota(0)
    + 2}) \gstep{\tau_1} (\Theta_{\iota(1) + 2}, \sigma_{\iota(1) + 2})
  \gstep{\tau_2} \ldots\]
\end{definition}
\begin{fact}\label{lem:greedy-unique}
  If $\tau \in M^\omega$ describes a path through $\TT_\Aa$, there exists
  a greedy run of $\tau$ which is unique up-to the choice of chips for the
  $\Theta_i$. If $\tau$ does not describe such a path, no greedy run of $\tau$
  exists.
\end{fact}
\begin{proof}
  For the existence of the greedy run, observe that the transitions as prescribed by clauses 1., 2.~and 4.~of
  \Cref{def:greedy-step} can always be taken. The only reason why constructing
  such a run might thus fail is if some prescribed $\tau_i$-successor transition could not
  be taken. The only reason for this would be that the current Safra board is on
  a set different from $\dom(\tau_i)$. But if $(\Theta_0, \sigma_0) \in \Sb(\Aa,
  \dom(\tau_0))$, it is easily observed that this problem will not arise as $\tau$
  is assumed to describe a path through $\TT_\Aa$. Hence a greedy run can be
  constructed and it indeed is a run because all `letters' of $\tau$ are read
  eventually. As the clauses of \Cref{def:greedy-step} always prescribe a unique
  transition to be taken next, the order of transitions along the greedy runs of
  $\tau$ is always fixed, meaning they can only differ by the choice of chips as
  claimed.
  
  For the second claim, observe that if $\tau$ does not describe a path then
  there must exist $\tau_i$ and $\tau_{i + 1}$ such that $\cod(\tau_{i})
  \neq \dom(\tau_{i + 1})$. In such a cases, the $\tau_{i + 1}$-successor step
  cannot be taken as elaborated above, meaning no run (and thus no greedy run)
  on $\tau$ can exist.
\end{proof}


We continue by proving that the definitions we have given above
are correct in the following sense: Any $\tau \in M^\omega$ describes a path
satisfying the trace condition if and only if there exists an accepting Safra
board run on $\tau$. Our arguments rely on the correspondence between Safra
board runs and runs on $\SB(\Aa, M)$. Thus, the results only hold in $\TT_\Aa$
with $\Aa$ finite.

\begin{lemma}\label{lem:tc-to-greedy}
  Fix a finite set $M$ of morphisms of $\TT_\Aa$ for some finite $\Aa$. If $\tau
  \in M^\omega$ describes a path which satisfies the trace condition then the
  greedy run on $\tau$ exists and is accepting.
\end{lemma}
\begin{proof}
  The greedy run $(\Theta_i, \sigma_i)_{i \in \omega}$ exists by
  \Cref{lem:greedy-unique}. Recall that there exists a function $\iota \colon \omega
  \to \omega$ indicating the index at which $\tau_i$ is read, i.e.
  $(\Theta_{\iota(i)}, \sigma_{\iota(i)}) \step{\tau_i} (\Theta_{\iota(i) + 1},
  \sigma_{\iota(i) + 1})$. Furthermore, observe that, as $\tau$ satisfies the
  trace condition, there
  exists an accepting run $\rho \in Q^\omega$ of $\tau$ on $\SB(\Aa, M)$. 
  As the run $\rho$ is accepting, it must, from some point $R$ onwards, `track' a trace along
  $\tau$ through the states $\Sigma X \in O.~ X \times (\Aa \cup 0^*) \subseteq
  Q$. Such a state $(X, x, a)$ corresponds to the spot $(x, a)$ on a Safra
  board on $X$ and this connection is vital to this proof. For any $R \leq i$ we thus
  write $\sigma_{\iota(i)}(\rho_i)$ to mean $\sigma_{\iota(i)}(x, a)$ where
  $\rho_i = (X, x, a)$, treating $0^*$ as $0 \in \Aa$ (it is easily observed
  that the object $X_{\iota(i)} \in \Ob(\TT_\Aa)$ on which $\Theta_{\iota(i)}$ is defined
  must always be identical with $X$).

  We make a few observations about the Safra boards along greedy runs just
  before the next letter of $\tau$ is read, i.e. the boards $(\Theta_{\iota(i)},
  \sigma_{\iota(i)})$. For this, fix $\abs{X} \coloneq
  \sup_{i \in \omega} \abs{X_i}$ which is finite as $M$ is a finite set
  of morphisms and there thus exist only finitely many distinct $X_i$.
  \begin{enumerate}
  \item 
    $\abs{\sigma_{\iota(i)}(x, a)} \leq 1$ at every $\iota(i)$: This is ensured
    by the thinning after the $\tau_{i - 1}$-successor (or the fact that the
    greedy run starts on the empty board).
    We thus
    treat the $\sigma_{\iota(i)}$ as functions $X_i \times \Aa \to
    \pow(\Theta_{\iota(i)})$ with $\sigma_{\iota(i)}(x, a) =
    \emptyset$ iff $\sigma_{\iota(i)}(x, a) = \emptyset$ under the original
    interpretation.
  \item 
    $\abs{\Theta_{\iota(i)}} \leq \abs{X} \cdot \abs{\Aa}$: If there were
    more than $\abs{X} \cdot \abs{\Aa}$ chips, one would have to be covered on
    $(\Theta_{\iota(i) - 2}, \sigma_{\iota(i) - 2})$ (the board resulting from
    the last reset which is part of $(\Theta_{\iota(i - 1) + 2}, \sigma_{\iota(i
      - 1) + 2}) \gstep{\tau_i} (\Theta_{\iota(i) + 2}, \sigma_{\iota(i) + 2})$) as there is only one
    top-most chip on each $(x, a)$, contradicting the fact that the last reset
    was applied at $(\Theta_{\iota(i) - 3}, \sigma_{\iota(i) - 3})$.
  \item $\sigma_{\iota(i)}(\rho_i) \neq \emptyset$ for any $R \leq i$: We argue
    per induction on $i$. First, suppose $R = i$: Then $\rho_i = (X_i,
    x, 0)$ for
    some $x \in X_i$. Because any spot $(x, 0)$ on the Safra board without a
    stack is populated in a greedy run before the next morphism is read, it
    follows
    that $\sigma_{\iota(i)}(\rho_i) = \sigma_{\iota(i)}(x, 0)$ is not empty.
    Now, because $\sigma_{\iota(i)}(\rho_i)$ is not empty, it is easily
    observed that the stack on $\sigma_{\iota(i)}(\rho_i)$ will be moved onto
    $\sigma_{\iota(i) + 1}(\rho_{i + 1})$ when computing the $\tau_i$-successor.
    As any steps that could occur between $(\Theta_{\iota(i) + 1},
    \sigma_{\iota(i) + 1})$ and $(\Theta_{\iota(i + 1)}, \sigma_{\iota(i + 1)})$
    never clear away all stacks on any space on the board which has at least one
    stack on it, it thus follows that $\sigma_{\iota(i + 1)}(\rho_{i + 1})
    \neq \emptyset$ (although the unique stack on it may not be the one
    moved over from $\sigma_{\iota(i)}(\rho_i)$ because of a thinning step).
  \item There must be a maximal height $1 \leq k \leq \abs{X} \cdot \abs{\Aa}$ such
    that from some $R < K$ onwards, $\abs{\sigma_{\iota(i)}(\rho_i)}
    \geq k$ of the height of the stack on $\rho_i$ 
    for every $\iota(i) > K$: This follows from the fact that $\abs{\sigma_{\iota(i)}(\rho_i)} \leq
    \abs{X} \cdot \abs{\Aa}$
    (as a consequence of 2.) and 
    $\abs{\sigma_{\iota(i)}(\rho_i)} \geq 1$ (as a consequence of 3.).
  \item 
    From some $K < N$ onwards, the $k$th chip of all
    $\sigma_{{\iota(i)}}(\rho_{i})$ with $N \leq i$ needs to remain the
    same: As
    $\abs{\sigma_{\iota(i)}(\rho_i)}$ is never less than $k$ again, meaning the $k$th chip is
    never cleared as part of a reset, the only way that the color
    of the $k$th chip could change would be if the stack on $\sigma_{\iota(i)}(\rho_i)$ was `switched' for some
    $<_{\Theta}$-smaller stack with a different $k$th chip as part of a
    thinning. Such a stack will also always be smaller according to the
    lexicographic ordering on the first $k$ elements. But this lexicographic
    ordering is well-founded on arbitrary finite linear orders, as it is always
    embeddable into the well-founded $\omega^k$. Thus, such replacements can
    only take place finitely often.
  \end{enumerate}

  Thus the $k$th value of $\sigma_{\iota(i)}(\rho_i)$, call it
  $\gamma$, stays constant for
  any $N \leq i$, meaning also $\gamma \in \Theta_i$ for all $i$ with $N
  \leq i$. It suffices to
  to prove that infinitely many $\gamma$-resets take place to conclude the run
  $(\Theta_i, \sigma_i)_{i \in \omega}$
  accepting. As $\rho$ is an accepting run, it passes through states $(x,
  0^*) \in F$ infinitely
  many times. Observe that whenever the run enters $(x, 0^*)$, the trace it
  follows has attained $\alpha$, meaning a new chip is placed on top of
  the stack on $\sigma_{\iota(i)}(\rho_{i})$ which is the
  stack containing $\gamma$ from $N$ onwards. As $k$ was chosen as the
  greatest infinitely recurring stack height, it also follows that
  $\abs{\sigma_{\iota(i)}(\rho_i)} = k$, and thus
  $\max\,\sigma_{\iota(i)}(\rho_i) = \gamma$,
  infinitely often. After $N$, this can only happen if the new chips added by
  the trace tracked by $\rho$ attaining $\alpha$ are removed from above $\gamma$ via
  an $\gamma$-reset. Thus, infinitely many $\gamma$-resets have to take place
  along $(\Theta_i, \sigma_i)_{i \in \omega}$, making it an accepting run on $\tau$.
\end{proof}

\begin{lemma}\label{lem:run-to-tc}
  Fix a finite set $M$ of morphisms of $\TT_\Aa$ for some finite $\Aa$. Now
  suppose some $\tau \in M^\omega$ had an accepting run $(\Theta_i, \sigma_i)_{i
  \in \omega}$. Then $\tau$ describes a path through $\TT_\Aa$ which satisfies the
  trace condition of $\TT_\Aa$.
\end{lemma}
\begin{proof}
  We prove this by showing that there must exist an accepting run of $\tau$ on
  $\SB(\Aa, M)$. As the run $(\Theta_i, \sigma_i)_{i \in \omega}$ is accepting,
  there exist $N \in \omega$ and $\gamma \in \bigcap_{N \leq i} \Theta_i$ such
  that infinitely many $\gamma$-resets take place along the run. Now denote by
  $\ol{\tau}[i, j] \in M^\omega$ the letters of $\tau$ read between the indexes
  $i$ and $j$, i.e. if $\iota(k - 1) < i \leq \iota(k) < \iota(k + n) < j \leq
  \iota(k + n + 1)$ then $\ol{\tau}[i, j] = \tau_k\tau_{k + 1} \ldots \tau_{k +
    n}$. We begin by proving a crucial fact: For $N \leq i \leq j$ if $\gamma \in \bigcup \sigma_j(x,
  a)$ then there must exist $x' \in X_i$ and $b \in \Aa$ such that $\gamma \in
  \bigcup \sigma_i(x', b)$ and $(X_i, x', b) \xrightarrow{\ol{\tau}[i, j]} (X_j,
  x, a)$ on
  $\SB(\Aa, M)$ (for this, we identify $0$ and $0^*$). We prove this per
  induction on $j$. Clearly, if $i = j$ then one
  may choose $x' \coloneq x$ and $b \coloneq a$ as $(X_i, x, a)
  \xrightarrow{\varepsilon} (X_i, x, a)$ in $\SB(\Aa, M)$. For the inductive step, we
  proceed per case distinction on the transition step between $(\Theta_j,
  \sigma_j)$ and $(\Theta_{j + 1}, \sigma_{j + 1})$:
  \begin{itemize}
  \item $(\Theta_j, \sigma_j) \step{\tau_k} (\Theta_{j + 1}, \sigma_{j + 1})$:
    Suppose $\gamma \in \bigcup \sigma_{j + i}(x, a)$. As $\gamma \in \Theta_j$,
    it is easily observed that a stack containing $\gamma$ can only have arrived
    on $\gamma$ if it was `moved' there by the previous transition. More
    formally, that means there have to be a $x' \in X_j$ and $b \in \Aa$ with
    $\gamma \in \bigcup \sigma_j(x', b)$ and $(X_j, x', b) \xrightarrow{\tau_k}
    (X_{j + 1}, x, a)$. Per inductive hypothesis, there furthermore have to be
    $x'' \in X_i$ and $c \in \Aa$ such that $\gamma \in \bigcup\sigma_i(x'', c)$
    and $(X_i, x'', c) \xrightarrow{\ol{\tau}[i, j]} (X_j, x', b)$. As
    $\ol{\tau}[i, j + 1] = \ol{\tau}[i, j]\tau_k$, this yields $(X_i, x'', c)
    \xrightarrow{\ol{\tau}[i, j + 1]} (X_{j + 1}, x, a)$ as desired.
  \item $(\Theta_j, \sigma_j) \step{W} (\Theta_{j + 1}, \sigma_{j + 1})$: If
    $\gamma \in \bigcup \sigma_{j + 1}(x, a)$ then also $\gamma \in \bigcup
    \sigma_{j}(x, a)$ because weakening may only remove stacks. Then the claim
    readily follows from the inductive hypothesis because $\ol{\tau}[i, j + 1] =
    \ol{\tau}[i, j]$.
  \item $(\Theta_j, \sigma_j) \step{R_{\gamma'}} (\Theta_{j + 1}, \sigma_{j + 1})$:
    As such a reset only removes chips from some stacks, $\gamma \in \bigcup
    \sigma_{j + 1}(x, a)$ means that also $\gamma \in \bigcup \sigma_{j}(x, a)$.
    Thus simply proceed per inductive hypothesis.
  \item $(\Theta_j, \sigma_j) \step{P} (\Theta_{j + 1}, \sigma_{j + 1})$:
    Again, $\gamma \in \bigcup \sigma_{j + 1}(x, a)$ entails $\gamma \in \bigcup
    \sigma_{j}(x, a)$ because $\sigma_{j + 1}$ differs from $\sigma_j$ only by
    the addition of some empty stacks (which is thus cannot contain $\gamma$).
    Proceed per inductive hypothesis.
  \end{itemize}

  Now let $(r_n)_{n \in \omega}$ be a sequence of indexes of $\gamma$-resets
  after $N$, i.e. a monotone increasing sequence with $N <
  r_0$ and $(\Theta_{r_n}, \sigma_{r_n}) \step{R_\gamma} (\Theta_{r_n + 1},
  \sigma_{r_n + 1})$. Define the sets $S_n \coloneq \{(x, a) ~|~ \gamma \in
  \bigcup \sigma_{r_n}(x, a)\}$. The previous result means that for any $(x, a)
  \in S_{n + 1}$ there exist $(x', b) \in S_n$ such that $(X_{r_n}, x', b)
  \xrightarrow{\ol{\tau}[r_n, r_{n + 1}]} (X_{r_{n + 1}}, x, a)$ in $\SB(\Aa,
  M)$ (in which we identify $0 \in \Aa$ with $0^*$ in the automata states). An
  application of König's Lemma yields a sequence $((x_n, a_n) \in S_n)_{n \in
    \omega}$ such that $(X_{r_n}, x_n, a_n) \xrightarrow{\ol{\tau}[r_n, r_{n +
      1}]} (X_{r_{n + 1}}, x_{n + 1}, a_{n + 1})$ for every $n \in \omega$.
  Notably, each run segment $(X_{r_n}, x_n, a_n) \xrightarrow{\ol{\tau}[r_n, r_{n +
      1}]} (X_{r_{n + 1}}, x_{n + 1}, a_{n + 1})$ crosses the set $F$ of
  accepting states of $\SB(\Aa, M)$ at least once: In $(\Theta_{r_n + 1},
  \sigma_{n + 1})$, each instance of $\gamma$ is the top-most chip on its
  respective stack. In $(\Theta_{r_{n + 1}}, \sigma_{r_{n + 1}})$, on the other
  hand, every instance of $\gamma$ is covered. This means that each stack
  $\gamma \in S \in \sigma(x, a)$ with $(x, a) \in S_{n + 1}$ must have
  `attained $\alpha$' at least once between $r_n$ and $r_{n + 1}$. In
  $\SB(\Aa, M)$, this corresponds to crossing $F$. The run segments 
  $(X_{r_n}, x_n, a_n) \xrightarrow{\ol{\tau}[r_n, r_{n +
      1}]} (X_{r_{n + 1}}, x_{n + 1}, a_{n + 1})$ thus already provide the
  suffix of an accepting run on $\tau$ as $F$ is crossed infinitely often. All that remains is to show that there
  is a run segment $X_0 \xrightarrow{\ol{\tau}[0, r_0]} (X_{r_0}, x_0, a_0)$ to
  assemble an accepting run of $\tau$ on $\SB(\Aa, M)$. 
  It follows from another application of the previous result that there has to be an $x \in X_N$ and an $a \in \Aa$
  such that $(X_N, x, a) \xrightarrow{\ol{\tau}[N, r_0]} (X_{r_0}, x_0, a_0)$.
  Now examine the step $(\Theta_{N - 1}, \sigma_{N - 1}) \step{} (\Theta_{N},
  \sigma_N)$ which one may assume, without loss of generality, introduces the
  chip $\gamma$ to $\Theta_N$, i.e. $\gamma \not\in \Theta_{N - 1}$. New chips can
  only be introduced by the covering phase of a
  $\tau_k$-step. Thus, new chips can only appear on $(x, 0)$, meaning
  the run segment above is actually $(X_n, x, 0) \xrightarrow{\ol{\tau}[N, r_0]}
  (X_{r_0}, x_0, a_0)$. Lastly, observe that the existence of the run
  $(\Theta_i, \sigma_i)_{i \in \omega}$ already guarantees that $\cod(\tau_k) =
  \dom(\tau_{k + 1})$ as $(\Theta_{\iota(k + 1)}, \sigma_{\iota(k + 1)})
  \step{\tau_{k + 1}} (\Theta_{\iota(k + 1) + 1}, \sigma_{\iota(k + 1) + 1})$
  for each $k \in \omega$.
  That means that $X_0 \xrightarrow{\ol{\tau}[0, N]} (X_n, x, 0)$ is a run
  segment on $\SB(\Aa, M)$. Thus, one may assemble the accepting run of
  $\tau$ on
  $\SB(\Aa, M)$ pictured below and an conclude that $\tau$ indeed describes a
  path through $\TT_\Aa$ which satisfies the trace condition.
  \[X_0 \xrightarrow{\ol{\tau}[0, N]} (X_n, x, 0) \xrightarrow{\ol{\tau}[N,
      r_0]} (X_{r_0}, x_0, a_0) \xrightarrow{\ol{\tau}[r_0, r_1]} (X_{r_1}, x_1,
    a_1) \xrightarrow{\ol{\tau}[r_1, r_2]} \ldots\]
\end{proof}

\begin{theorem}
  Fix a finite set $M$ of morphisms of $\TT_\Aa$ for some finite $\Aa$. Then
  there exists an accepting Safra board run on $\tau \in M^\omega$ if and only
  if $\tau$ describes a path through $\TT_\Aa$ which satisfies the trace condition.
\end{theorem}

To close the section, we illustrate the connection between Safra boards and the
determinisation of Büchi automata (more concretely of $\SB(\Aa, M)$) via Safra's
construction~\cite{safraComplexityOmegaAutomata1988}. We do this by defining
a determinised variant of $\SB(\Aa, M)$ in terms of Safra boards.

To ensure that the automaton we construct has a finite state space, we first
prove that one can `make do' with a finite supply of chips when carrying out
greedy transition steps. The last condition asserted in \Cref{lem:greedy-ceil}
is crucial for the acceptance condition of the constructed automaton.

\begin{definition}
  Fix a finite $\Aa$, $X \in \TT_\Aa$ and a number $K \geq \abs{X}$. 
  Fixing the supply of chips $\ol{K} = \{n \in \omega ~|~ n < K \cdot (\abs{\Aa} +
  1)\}$, write 
  $\Sb(\Aa, X, K) \subseteq \Sb(\Aa, X)$ for the
  set of \emph{$K$-sparse} Safra boards. A board $(\Theta, \sigma)$ is $K$-sparse if
  \begin{itemize}
  \item $\Theta \subseteq \ol{K}$ and $\abs{\Theta} \leq K \cdot
    \abs{\Aa}$
  \item There is at most one stack on each board
    position in $(\Theta, \sigma)$
  \item There are no stacks on any position $(x, \alpha)$ on $(\Theta, \sigma)$
  \end{itemize} 
\end{definition}

\begin{lemma}\label{lem:greedy-ceil}
  For $X, Y \in \Ob(\TT_\Aa)$ and a $K \geq \abs{X}, \abs{Y}$
  let $(\Theta_0, \sigma_0) \in \Sb(\Aa, X, K)$.
  Then there exists, for any $\tau \colon X \to Y$, a Safra board $(\Theta_n,
  \sigma_n) \in \Sb(\Aa, Y, K)$ and a transition sequence $(\Theta_0, \sigma_0)
  \step{?} \ldots \step{?} (\Theta_n, \sigma_n)$ such that $(\Theta_0, \sigma_0)
  \gstep{\tau} (\Theta_n, \sigma_n)$. Furthermore, $\Theta_n \cap
  \Theta_0 = \bigcap_{i \leq n} \Theta_i $.
\end{lemma}
\begin{proof}
  For a board $(\Theta, \sigma)$ define $\abs{\sigma} \coloneq
  \abs{\{(x, a) \in \dom(\sigma) ~|~ \sigma(x, a) \neq \emptyset\}}$, i.e. the
  number of board positions with stacks on them.
  We shall argue that the steps 1. through 4. from \Cref{def:greedy-step} can be taken
  from $(\Theta_0, \sigma_0)$ in such a way that the resulting sequence $(\Theta_i,
  \sigma_i)_{i \leq n}$ is such that $\Theta_i \subseteq \ol{K}$. To ensure
  that $\Theta_n \cap \Theta_0 = \bigcap_{i \leq n} \Theta_i$, one requires the
  transitions will only introduce chips from $\Theta' \coloneq \ol{K}
  \setminus \Theta_0$. Observe that from the first sparseness condition on $(\Theta_0,
  \sigma_0)$, it follows that $\abs{\Theta_0} \leq K \cdot \abs{\Aa}$ and
  thus $K \leq \abs{\Theta'}$.
  \begin{enumerate}
  \item[1.] Suppose $\gamma_1 < \ldots < \gamma_n \in \Theta_0$ were covered in
    $(\Theta_0, \sigma_0)$. We begin by showing that each transition of the sequence
    \[(\Theta_0, \sigma_0) \step{R_{\gamma_k}} (\Theta_1, \sigma_1)
      \step{R_{\gamma_{k - 1}}} \ldots
      \step{R_{\gamma_1}} (\Theta_k, \sigma_k)\]
    can be taken. The only thing which could prevent a transition $(\Theta_i,
    \sigma_i) \step{R_{\gamma_{k - i}}} (\Theta_{ i + 1}, \sigma_{i + 1})$ along this
    sequence from being legal would be that $\gamma_{k - i}$ was not covered in
    $(\Theta_i, \sigma_i)$. This can only happen if some earlier reset along
    this sequence `uncovered' $\gamma_{k - i}$. But this cannot happen:
    Suppose $\gamma > \gamma'$ and consider some stack with $S \ni \gamma'$ and
    $\gamma'$ covered, i.e. $\max S \neq \gamma'$. Then there are two
    possibilities for $\sreset{S}{\gamma}$: If $\gamma \not\in S$ then
    $\sreset{S}{\gamma} = S$ and $\gamma'$ thus remains covered. If $\gamma \in
    S$ then $\gamma \in \sreset{S}{\gamma}$, meaning a chip $\gamma > \gamma'$
    remains in $S$ and $\gamma'$ remains covered. Thus, a reset on some
    $\gamma_j > \gamma_{k - i}$ cannot uncover $\gamma_{k - i}$, meaning the
    $\gamma_{k - i}$-reset may be carried out.

    As resets only ever remove chips, it is easily observed that $\Theta_k \subseteq
    \Theta_0 \subseteq \ol{K}$. Because resets never add any new stacks, it follows
    that $\abs{\sigma_k} = \abs{\sigma_0}$.
    Now observe the following: If a chip is covered in $(\Theta_k, \sigma_k)$ it
    is also covered in $(\Theta_0, \sigma_0)$, as resets only ever remove chips from
    the tops of stacks. Thus, every chip $\gamma \in \Theta_k$ must be at the
    top of at least one stack on $(\Theta_k, \sigma_k)$: Suppose, towards
    contradiction, that there was a covered chip in $(\Theta_k, \sigma_k)$. But
    then it would also have been covered in $(\Theta_0, \sigma_0)$, meaning it would
    have been among the $\gamma_i$ and would have been reset by the sequence of
    resets. But then it cannot be covered in $(\Theta_k, \sigma_k)$ as any reset
    chip is uncovered by the reset. If each chip of $\Theta_k$ must occur on top
    of at least one stack, it is easily observed that $\abs{\Theta_k} \leq
    \abs{\sigma_k}$.
  \item[2.] Carry out the population $(\Theta_k, \sigma_k)
    \step{P} (\Theta_{k + 1}, \sigma_{k + 1})$. That means if $\sigma_k(x, 0) =
    \emptyset$ then $\sigma_{k + 1}(x, 0) = \{\emptyset\}$. As $\Theta_{k + 1} = \Theta_k$,
    it follows $\Theta_{k + 1} \subseteq \ol{K}$.
  \item[3.] Now carry out $(\Theta_{k + 1}, \sigma_{k + 1}) \step{\tau} (\Theta_{k + 2},
    \sigma_{k + 2})$ with $\tau \colon X \to Y$. During the transformation, the stacks are
    moved from positions $(x, a) \in X \times \Aa$ to positions $(y, a') \in Y
    \times \Aa$ according to $\tau$. If a stack has landed on some $(y,
    \alpha)$, it is then moved to $(y, 0)$ and has a new chip added to its top.
    If multiple stacks landed on $(y, \alpha)$, the same chip is added to all of
    them. During this step, only at most $\abs{\{(y, \alpha) ~|~ y \in
      Y\}} = \abs{Y}$ new chips are introduced to $\Theta_{k + 2}$, meaning the
    supply of chips $\Theta'$ is sufficient. Thus,
    $\Theta_{l + 1} \subseteq \ol{K}$.
  \item[4.] In this step $(\Theta_{k + 2}, \sigma_{k + 2}) \step{T} (\Theta_{k +
    3}, \sigma_{k + 3})$,
    some of the stacks, and possibly some of the chips that used to be in them,
    are removed and no new chips are added. Thus $\Theta_{k + 3} \subseteq \ol{K}$.
  \end{enumerate}
  It is easily observed that after step 4.~there is at most one stack on each
  position of $(\Theta_{k + 3}, \sigma_{k + 3})$. Furthermore, after step 3.~there are no
  stacks on positions $(y, \alpha)$, a fact which remains unchanged by step 4.
  It remains to show that $\abs{\Theta_{k + 3}} \subseteq K \cdot \abs{\Aa}$.
  We have shown that after step 1.~ $\abs{\Theta_k} = \abs{\sigma_0}$. As
  there are no stacks on $(x, \alpha)$ in $(\Theta_0, \sigma_0)$, that
  means that $\abs{\Theta_k} \leq \abs{X} \cdot (\abs{\Aa} - 1) \leq K \cdot
  (\abs{\Aa} - 1)$. It is easily
  observed that after step 2.~$\abs{\Theta_{k + 1}} = \abs{\Theta_k}$. As
  already argued, step 3.~adds at most $\abs{Y} \leq K$ new chips, meaning
  $\abs{\Theta_{k + 2}} \leq \abs{\Theta_{k + 1}} + K \leq K \cdot \abs{\Aa}$. As
  step 4.~only removes chips, this means $\abs{\Theta_{k + 3}} \leq \abs{\Theta_{k +
      2}} \leq K \cdot \abs{\Aa}$ as desired.
\end{proof}

We can thus construct a deterministic Rabin automaton which recognizes the trace
condition of $\TT_\Aa$, similarly to the non-deterministic Büchi-automaton $\SB(\Aa, M)$.
For the construction of the automaton, as well as some of the arguments in
\Cref{sec:complete}, it would be helpful if for each $(\Theta, \sigma) \in
\Sb(\Aa, X, K)$ and $\tau \colon X \to Y$ there was some unique $(\Theta', \sigma')
\in \Sb(\Aa, Y, K)$ such that $(\Theta, \sigma) \gstep{\tau} (\Theta',
\sigma')$. We thus simply assume that for each such
$(\Theta, \sigma)$ and $\tau$ where this is applicable, such a choice has been
made, for example via an application of the axiom of choice or some other means,
and treat $\gstep{\tau}$ as an injective function on $K$-sparse $(\Theta,
\sigma)$ for suitable $K$ such that $(\Theta, \sigma) \gstep{\tau} (\Theta',
\sigma')$ was always derived according to \Cref{lem:greedy-ceil}.

\begin{definition}\label{def:safra-aut}
  Let $\Aa$ be finite and fix a finite set of objects $O \subset \Ob(\TT_\Aa)$ and
  a set of morphisms $M \subseteq \bigcup_{X, Y \in O} \Hom(X, Y)$. Furthermore,
  set $K \coloneq \max_{X \in O} \abs{X}$. The \emph{Safra
    automaton $\SSB(\Aa, S, M)$} for a starting object $S \in O$ is the Rabin automaton $(M, Q,
  \delta, (S, (\Theta_0, \sigma_0)), R)$ where
  \begin{align*}
    Q & \coloneq \Sigma X \in O.~\Sb(\Aa, X, K) \\
    \delta & \coloneq ((X, (\Theta, \sigma)), \tau \colon X \to Y ) \mapsto \text{some } (Y, (\Theta', \sigma')) \stackrel{\tau}{\leftsquigarrow_g} (X, (\Theta, \sigma)) \\
    R & \coloneq \{(\{(X, (\Theta, \sigma)) \in Q ~|~ \gamma \in \Theta, \gamma \text{ covered}\}, \{(X, (\Theta, \sigma)) \in Q ~|~ \gamma \not\in \Theta\}) ~|~ \gamma \in \ol{K}\}
  \end{align*}
  and $(\Theta_0, \sigma_0) = (\emptyset, (s, a) \mapsto \emptyset)$.
\end{definition}

\begin{lemma}\label{lem:safra-automaton-correct}
  For a finite $\Aa$, $O \subset \Ob(\TT_\Aa)$ and $M \subseteq \bigcup_{X, Y
    \in O} \Hom(X, Y)$, the Safra automaton $\SSB(\Aa, S, M)$ accepts a sequence
  $\tau \in M^\omega$ if and only if $\dom(\tau_0) = S$ and $\tau$ describes a
  path through $\TT_\Aa$ which satisfies the trace condition.
\end{lemma}
\begin{proof}
  First, consider any run $(S, (\Theta_0, \sigma_0)) \xrightarrow{\tau_0} (X_1,
  (\Theta_1, \sigma_1)) \xrightarrow{\tau_1} \ldots$ of $\tau \in M^\omega$ on
  $\SSB(\Aa, S, M)$. Each state
  transition $(X_i, (\Theta_i, \sigma_i)) \xrightarrow{\tau_i} (X_{i + 1},
  (\Theta_{i + 1}, \sigma_{i + 1}))$ corresponds to a greedy transition
  $(\Theta_i, \sigma_i) \gstep{\tau_i} (\Theta_{i + 1}, \sigma_{i + 1})$
  obtained via \Cref{lem:greedy-ceil}. Thus, one may `expand' the run into a
  Safra board run of the following shape
  \[
    (\Theta_0, \sigma_0) \step{?} (\Theta_0^1, \sigma_0^1) \step{?} \ldots
    \step{?} (\Theta_0^{n_0}, \sigma_0^{n_0}) \step{?}
    (\Theta_1, \sigma_1) \step{?} (\Theta_1^1, \sigma_1^1) \step{?} \ldots
    \step{?} (\Theta_1^{n_1}, \sigma_1^{n_1}) \step{?} \ldots
  \]
  By comparing \Cref{def:greedy,def:greedy-step}, it is easy to see that the run
  above must be a greedy run. Analogously to \Cref{lem:greedy-unique}, this greedy
  Safra board run, and by extension the run on $\SSB(\Aa, S, M)$, exists if and
  only if $\tau \in M^\omega$ describes a path through $\TT_\Aa$ with
  $\dom(\tau_0) = S$. The latter condition is caused by the fact that
  $(\Theta_0, \sigma_0)$ always is a Safra board on $S \in \Ob(\TT_\Aa)$. From
  \Cref{lem:tc-to-greedy,lem:run-to-tc} it follows that the Safra board run above
  is accepting if and only if $\tau$ satisfies the trace condition. It thus
  suffices to argue that the Rabin condition $R$ holds on a run on $\SSB(\Aa, S,
  M)$ iff the expanded Safra board run is accepting. For this, observe that $R$
  holds on a run if there exists some chip $\gamma \in \ol{K}$ and some 
  $N \in \omega$ such that $\gamma \in \Theta_i$ for ever $N \leq i$ and
  $\gamma$ is covered in infinitely many $(\Theta_i, \sigma_i)$. By scrutinizing
  \Cref{lem:greedy-ceil}, one can see that the former condition means that
  $\gamma$ is present on every Safra board along the greedy Safra board run from
  some point onwards. \Cref{def:greedy-step} dictates that whenever $\gamma$ is
  covered in $(\Theta_i, \sigma_i)$, a $\gamma$-reset takes place between
  $(\Theta_i, \sigma_i)$ and $(\Theta_{i + 1}, \sigma_{i + 1})$ in the greedy
  run. The Rabin condition is thus completely analogous to the acceptance
  condition on Safra board runs: It holds if and only if there is some chip
  $\gamma$ which is eventually never removed again and reset infinitely often.
\end{proof}

We close this section by deriving some bounds on the sizes of various components
of $\SSB(\Aa, S, M)$. Note that these are not optimal bounds for Safra
constructions. The reader may consult
\cite{pitermanNondeterministicBuchiStreett}, for example, for a more space-efficient
construction.

\begin{lemma}\label{lem:complexity}
  Fix finite $\Aa$, $O \subset \Ob(\TT_\Aa)$ and $M \subseteq \bigcup_{X, Y
    \in O} \Hom(X, Y)$. Denote the Safra automaton $\SSB(\Aa, S, M) = (M, Q,
  \delta, s, R)$ and let $K \coloneq \max_{X \in O} \abs{X}$.
  \begin{enumerate}
  \item For any $X \in O$ one has $\abs{\Sb(\Aa, X, K)} \leq \sum_{C = 1}^{K \cdot \abs{A}} {
    K \cdot (\abs{\Aa} + 1) \choose C} \cdot C! \cdot 2^{C \cdot \abs{X} \cdot
    (\abs{\Aa} - 1)} = O(K!)$.
\item $\abs{Q} \leq \sum_{X \in O} \abs{\Sb(\Aa, X, K)} = O(\abs{O} \cdot K!)$.
\item $\abs{R} = K \cdot (\abs{\Aa} + 1) = O(K)$.
  \end{enumerate}
\end{lemma}
\begin{proof}
  2.~readily follows from 1. Furthermore, 3.~holds as $\abs{R} = \abs{\ol{K}} =
  K \cdot (\abs{\Aa} + 1)$. To understand the bound in 1., observe that
  a board $(\Theta, \sigma) \in \Sb(\Aa, X, K)$ with $\abs{\Theta} = C$ consists
  of three components: A choice $\Theta \subseteq \ol{K}$ (of which there are
  ${\ol{K} \choose C} = {K \cdot (\abs{\Aa} + 1) \choose C}$ different ones
  if $\abs{\Theta} = C$), a linear order imposed on $\Theta$ (of which there are
  $\abs{\Theta}! = C!$ many) and one stack $S
  \subseteq \PP(\Theta)$ on each position $(x, a)$ with $a \neq \alpha$, i.e. a
  function $\sigma \colon X \times (\Aa \setminus \{\alpha\}) \to \PP(\Theta)$ (of
  which there are ${\bigl(2^{C}\bigr)}{}^{\abs{X} \cdot (\abs{\Aa} - 1)} = 2^{C \cdot \abs{X}
    \cdot (\abs{\Aa} - 1)}$ many).
    Furthermore taking into account that
  $1 \leq \abs{\Theta} = C \leq K \cdot \abs{\Aa}$, one observes the bound
  stated in 1.
\end{proof}

\section{Reset Proof Systems}
\label{sec:acd_pwn}

Fix an activation algebra \( \Aa \) and an activation trace category \( \TT_\Aa \).
In this section, we show that every cyclic proof system induced by a trace
interpretation into $\TT_\Aa$ gives rise to a cyclic proof system whose soundness
condition is based on Safra boards. It serves as a starting point
 for deriving
concrete \textsc{Reset}-based proof systems based on concrete cyclic proof
systems, as we do in \Cref{sec:concrete}.

Given a cyclic proof system $\RR$ induced by a trace interpretation $\iota \colon \RR
\to \TT_\Aa$, the reset proof system $\Rrr(\RR)$ is obtained by annotating the
sequents of $\RR$ with $\Aa$-Safra boards. More specifically, an $\RR$-sequent
$\Gamma$ is annotated with a Safra board $(\Theta, \sigma) \in \Sb(\Aa,
\iota(\Gamma))$. Each derivation rule $R \in \RR$ is `lifted' to a corresponding derivation
rule in $\Rrr(\RR)$, the Safra boards annotating the $i$th premise being the result
of the transition of the trace interpretation map $r_i$ on the conclusion's
Safra board. Furthermore, the system $\Rrr(\RR)$ also contains structural rules
corresponding to the three types `bookkeeping transitions' on Safra boards:
Weakening, reset and population. The soundness condition of $\Rrr(\RR)$ requires
each simple cycle $\pi(t)$, i.e. path between a bud $t$ and its companion $\beta(t)$, to
contain an application of the $\RReset$-rule which leaves an invariant
$\theta$, a prefix of the control which remains unchanged along the whole simple
cycle.

\begin{definition}\label{def:acr}
  Fix a cyclic proof system $(\Seq, \RR, \rho, \SC)$ induced by a trace
  interpretation $\iota$ into $\TT_\Aa$. The \emph{reset proof system for
    $\RR$} is the cyclic proof system $\Rrr(\RR) \coloneq (\Rrr(\Seq), \Rrr(\RR), \Rrr(\rho), \Rrr(\SC))$ specified as follows.
    Sequents in \( \Rrr(\RR) \) are expressions \( \Gamma ; (\Theta, \sigma) \)
    where \( \Gamma \in \Seq \) is an \( \RR \)-sequent and $ (\Theta, \sigma)
    \in \Sb(\Aa, \iota(\Gamma))$ is a Safra board on $\Gamma$'s trace object $\iota(\Gamma)$.
  The order \( \Theta \) is called the \emph{control} of $\Gamma ; (\Theta, \sigma)$. The
  derivation rules of $\Rrr(\RR)$ consist annotated versions of rules from \( \RR \) and additional \emph{structural rules}. The structural rules are given by the following three rule schemas.
  \begin{mathpar}
    \inference[\textsc{Weak}]{\Gamma ; (\Theta', \sigma') \text{ where } (\Theta, \sigma) \step{W}
      (\Theta', \sigma')}{\Gamma ; (\Theta, \sigma)}

    \inference[$\textsc{Reset}_\gamma$]{\Gamma ; (\Theta', \sigma') \text{ where }
      (\Theta, \sigma) \step{R_\gamma}
      (\Theta', \sigma')}{\Gamma ; (\Theta, \sigma)}

    \inference[\RPop]{\Gamma ; (\Theta', \sigma') \text{ where } (\Theta, \sigma) \step{P}
      (\Theta', \sigma')}{\Gamma ; (\Theta, \sigma)}
  \end{mathpar}
  For each rule $R \in \RR$ with $\rho(R) = (\Gamma, \Delta_1, \ldots, \Delta_n)$
  and maps $r_i \colon \iota(\Gamma) \to \iota(\Delta_i)$ given by the trace interpretation, the
  following schema gives rules for each $(\Theta, \sigma) \in \Sb(\Aa, \iota(\Gamma))$:
  \[
    \inference[$R$]{
      \Delta_1 ; (\Theta_1, \sigma_1) \text{ where } (\Theta, \sigma) \step{r_1} (\Theta_1, \sigma_1)  \quad
      \ldots \quad
      \Delta_n ; (\Theta_n, \sigma_n) \text{ where } (\Theta, \sigma) \step{r_n} (\Theta_n, \sigma_n)
    }
    {\Gamma ; (\Theta, \sigma)}
  \]
  Let $D = (C, \lambda, \delta)$ be a preproof of $\Rrr(\RR)$.
  Pick some $t \in \dom(\beta)$ and let $\pi(t) = (\Gamma_i ; (\Theta_i, \sigma_i))_{i < n}$
  be sequents along the path from $\beta(t)$ to $t$. Let $\Theta$ be the longest common prefix
  of all of the $\Theta_i$. An \emph{invariant} of $\pi(t)$ is any prefix $\theta$
  of $\Theta$ such that an application of a $\max(\theta)$-reset occurs between
  $\beta(t)$ and $t$. Sometimes we
  speak of \emph{the} invariant of $\pi(t)$, in which case we refer to the longest
  such. 
  An $\Rrr(\RR)$-preproof satisfies the soundness condition $\Rrr(\SC)$ iff for every $t \in
  \dom(\beta)$ the path $\pi(t)$ between $\beta(t)$ and $t$ has an invariant.
\end{definition}

A \emph{reset proof} for \( \RR \) is a cyclic proof in $\Rrr(\RR)$ . This is essentially a cyclic proof
in $\RR$ with additional structure in the form of annotations. Any application of a
rule corresponding to $R \in \RR$ directly impacts the traces running through a
preproof while the structural rules perform `bookkeeping' for the
control $(\Theta, \sigma)$. This intuition can be made more formal: There exists
a proof morphism from $\Rrr(\RR)$ into $\RR$ arising from stripping away the annotations
$(\Theta, \sigma)$.

Fix a cyclic proof system $\RR$ induced by a trace interpretation on
$\TT_\Aa$. The function $\strip \colon \Rrr(\Seq) \to \Seq$ is defined by
$\strip(\Gamma ; (\Theta, \sigma)) \coloneq \Gamma$ on sequents.

\begin{lemma}\label{lem:strip-simple-morph}
  For a cyclic proof system $\RR$ induced by a trace interpretation on $\TT_\Aa$, the
  function $\strip \colon \Rrr(\Seq) \to \Seq$ can be extended to a preproof morphism
  $\strip \colon \Rrr(\RR) \to \RR$.
\end{lemma}
\begin{proof}
  We need to assign to every rule $\hat{R} \in \Rrr(\RR)$ a
  corresponding preproof $\strip(\hat{R})$ in $\RR$. There are only two cases to
  consider:
  \begin{itemize}
  \item \emph{$\hat{R}$ is a structural rule:} Then $\hat{R}$ is of shape
    \[
      \inference[$\hat{R}$]{\Gamma ; (\Theta', \sigma') \text{ where } (\Theta, \sigma) \step{X}
        (\Theta', \sigma')}{\Gamma ; (\Theta, \sigma)}
    \]
    where $X$ is $W$, $P$ or $R_\gamma$ for some $\gamma \in \Theta$. In any
    case, we need to find a preproof with assumption $\strip(\Gamma ; (\Theta,
    \sigma)) = \Gamma$ and premise $\strip(\Gamma ; (\Theta, \sigma)) =
    \Gamma$. Such a preproof is given by the identity preproof of $\Gamma$,
    i.e.~the triple $(\{\varepsilon\}, \varepsilon \mapsto \Gamma, \emptyset)$. 
  \item \emph{$\hat{R}$ corresponds to a rule $R \in \RR$:} That is, there is $R
    \in \RR$ with $\rho(R) = (\Gamma, \Delta_1, \ldots, \Delta_n)$ and and maps
    $r_i \colon \iota(\Gamma) \to \iota(\Delta_i)$ given by the trace interpretation
    and $\hat{R}$ is of the form
    \[
      \inference[$\hat{R}$]{
        \Delta_1 ; (\Theta_1, \sigma_1) \text{ where } (\Theta, \sigma) \step{r_1} (\Theta_1, \sigma_1)  \quad
        \ldots \quad
        \Delta_n ; (\Theta_n, \sigma_n) \text{ where } (\Theta, \sigma) \step{r_n} (\Theta_n, \sigma_n)
      }
      {\Gamma ; (\Theta, \sigma)}
    \]
    Then, analogously to the first case, we need to find a preproof of
    $\Gamma$ with open leaves $\Delta_1, \ldots, \Delta_n$ in $\RR$. The
    preproof consisting of exactly one application of $R$ is as desired.
  \end{itemize}
\end{proof}

\Cref{lem:strip-simple-morph} merely establishes that $\strip$ is a preproof
morphism not a proof morphism.
Showing the latter is more involved. 
That $\strip$ constitutes a proof morphism between $\Rrr(\RR)$
and $\RR$ can be understood as a relative soundness result: Suppose $\RR$ is
sound, i.e., the system proves only true sequents. As $\strip$ is a proof morphism, if a sequent $\Gamma ; (\Theta, \sigma)$ is provable in \( \Rrr(\RR) \),
then there is a cyclic proof of $\strip(\Gamma ; (\Theta, \sigma))$ in $\RR$ obtained
via the morphism, and so \( \Rrr(\RR) \) is sound.

\Cref{sec:sound} below concerns showing that \( \strip \) is a proof morphism.
In \Cref{sec:complete} we prove a completeness theorem for $\Rrr(\RR)$ relative
to $\RR$: If there is a cyclic proof $\Pi$ of $\Gamma$ in $\RR$, there
exists a cyclic proof $\hat{\Pi}$ of $\Gamma ; (\emptyset, (s, a) \mapsto
\emptyset)$. Furthermore, $\strip(\hat{\Pi})$ is a finite unfolding of
$\Pi$.

\subsection{Soundness}
\label{sec:sound}

The soundness proof relies on the concept of connected subgraphs of a cyclic
proof. In cyclic proofs, each connected subgraph can be
identified with a subset $\eta \subseteq \dom(\beta)$ which we call a connected
cycle.
Given a cyclic tree $C$ in cyclic normal form, a \emph{connected cycle} is a
set $\eta \subseteq \dom(\beta)$ of buds of
$C$ such that
\begin{enumerate}[(i)]
\item there exists some \emph{base element} $b(\eta) \in \eta$ such that $\beta(b(\eta)) \leq \beta(t)$ for
  every $t \in \eta$
\item for every $t_0 \in \eta$ there exist $t_1, \ldots, t_n \in \eta$ (where possibly
  $n = 0$) such that for each $i < n$, $\beta(t_i) \leq t_{i + 1}$ and $t_n = b(\eta)$
\end{enumerate}
For a cyclic tree $C = (T, \beta)$, a \emph{subtree} is a set $T' \subseteq T$
such that if $s, t \in T'$ and $s < u < t$ by the prefix ordering then $u \in
T'$ and furthermore if $s \in T' \cap \dom(\beta)$ then $\beta(s) \in T'$.
Any connected cycle $\eta$ of $C$ describes a subtree $C[\eta] = \{s \in T ~|~
\exists t \in \eta.~\beta(t) \leq s \leq t\}$ of $C$.

The connected cycles of a preproof are closely linked to their infinite
branches: For any infinite path through a cyclic tree, the nodes visited
infinitely often by it form a subtree described by a connected cycle. In the
following, we represent \emph{infinite paths} through a cyclic tree $C = (T, \beta)$ by
sequences $\pi \in T^\omega$ such that $\pi_{i + 1} \in \Chld(\pi_i)$ or $\pi_{i + 1}
= \beta(\pi_i)$ for each $i \in \omega$. Furthermore, denote $\Inf(\pi) \coloneq \{s
\in T ~|~ \pi_i = s \text{ infinitely often}\}$ and $\Occ(\pi) \coloneq
\{s \in T ~|~ \pi_i = s \text{ for some } i \in \omega\}$ and write $s <_+ t$ to
mean $t \in \Chld(s)$.

\begin{lemma}\label{lem:inf-cs}
  Let $\pi \in T^\omega$ be an infinite path through a cyclic tree $C = (T, \beta)$ in cyclic normal form.
  Then there exists a connected cycle $\eta$ of $C$ such that $\text{Inf}(\pi) =
  C[\eta]$.
\end{lemma}
\begin{proof}
  Without loss of generality, assume that $\text{Occ}(\pi) = \Inf(\pi)$. Now consider $\eta
  \coloneq \Occ(\pi) \cap \dom(\beta)$. We show that $\eta$ is a connected
  cycle and that indeed $\Inf(\pi) = C[\eta]$ via multiple intermediary steps.
  \begin{enumerate}
  \item \emph{For $t \in \beta$, if $s \leq t$ and $s \not\leq \beta(t)$ then
      $\beta(t) < s$:} Follows as $\leq$ is the prefix relation.
  \item \emph{Let $s\alpha$ be a finite path. Then for every $\alpha_i$ it follows
    that $s \leq \alpha_i$ or there exists some $\alpha_j$ such that $j < i$,
    $\beta(\alpha_j) \leq \alpha_i$ and $\beta(\alpha_j) < s$.} Proof per
    induction on $\abs{\alpha}$. If $\abs{\alpha} = 1$ then $s <_+ \alpha_0$ or
    $s \in \dom(\beta)$ and $\alpha_0 = \beta(s)$, meaning $\alpha_0 < s$ as
    desired. Now let $\beta(\alpha) = n + 1$, there are four cases to consider
    \begin{itemize}
    \item $s \leq \alpha_{n - 1}$ and $\alpha_{n - 1} <_+ \alpha_n$: Then $s \leq 
      \alpha_n$, trivially.
    \item $s \leq \alpha_{n - 1}$ and $\alpha_{n} = \beta(\alpha_{n - 1})$: Then
      suppose $s \not\leq \beta(\alpha_{n - 1})$. But this means that
      $\beta(\alpha_{n - 1}) < s$ necessarily, satisfying the second clause.
    \item $\beta(\alpha_j) \leq \alpha_{n - 1}$ and $\alpha_{n - 1} <_+ \alpha_n$:
      Again, the second clause trivially holds for $\alpha_n$.
    \item $\beta(\alpha_j) \leq \alpha_{n - 1}$ and $\alpha_n = \beta(\alpha_{n - 1})$:
      If $\beta(\alpha_j) \not\leq \alpha_{n}$ then $\alpha_n = \beta(\alpha_{n - 1}) <
      \beta(\alpha_j) < s$.
    \end{itemize}
  \item \emph{$\eta$ has a base element $b(\eta)$:} We prove that if $X
    \subseteq \eta$ such that there is a $b \in X$ such that $\beta(b) \leq \eta
    \setminus X$ (meaning $\beta(b) \leq t$ for all $t \in \eta \setminus X$)
    then $b(\eta) \in X$ per induction on $\abs{X}$. If $\abs{X} =
    1$ then clearly $X = \{b(\eta)\}$. Now let $\abs{X} > 1$, pick some finite
    segment
    $\beta(b) \alpha$ of $\pi$ such that $\Occ(\alpha) = \Inf(\pi)$. By the
    previous result, either $\beta(b) \leq \beta(t)$ for all $t \in \eta$, meaning $b = b(\eta)$,
    or there is some $b' \in \eta$ with $\beta(b') < \beta(b)$. In the latter case, $b' \in X$ as
    $\beta(b) \leq \beta(b')$ otherwise, and $\beta(b') \leq \eta \setminus X$
    by transitivity. Then continue the argument with $X' \coloneq X \setminus
    \{b\}$ and $b' \in X'$, noting that $\abs{X'} = \abs{X} - 1$.
  \item \emph{For every $t \in \eta$ there exist $l_0 \ldots l_n$ with $l_0 =
      t$, $l_n = b(\eta)$ and $\beta(l_i) < l_{i + 1}$ for all $i < n$:} Follows
    directly by observing that for every $l_0 \in \eta$ there exists a finite
    subpath $l_0 \alpha b(\eta)$ of $\pi$, describing such a sequence of leaves.
  \item \emph{$\Inf(\pi) \subseteq C[\eta]$:} For each $s \in \Inf(\pi)$ we must
    find a $t \in \eta$ with $\beta(t) \leq s \leq t$.
    First, if $s \in \Inf(\pi)$, there must
    be some $s \leq t \in \dom(\beta) \cap \Inf(\pi) = \eta$, as $\pi$ could not
    continue on infinitely from $s$ otherwise. Now suppose $s < \beta(t)$ for
    all $t \in \eta$ with $s \leq t$. Once $\pi$ passes $s$, it can never
    `jump back' below $s$: The `lowest' point it can reach is $\beta(t)$ for
    some $s \leq t$. But then $s$ cannot be reached more than once,
    contradicting $s \in \Inf(\pi)$.
  \item \emph{Let $s \alpha$ be a finite path, $s < t$ and $t \not\in \Occ(\alpha)$ then
      $\Occ(\alpha) \cap \Up(t) = \emptyset$:} Proof per induction on
    $\abs{\alpha}$. If $\abs{\alpha} = 1$ then $s <_+ \alpha_0$ as $s \not\in
    \dom(\beta)$ because $s < t$. In such a situation, $\alpha_0 \in \Up(t)$ is
    only possible if $\alpha_0 = t$, which contradicts the assumption.
    If $\abs{\alpha} = n + 1$ suppose $\alpha_{n} \in \Up(t)$. By the same
    argument as for $\abs{\alpha} = 1$, this means $\alpha_{n - 1} \not<_+
    \alpha_n$. Thus $\alpha_n = \beta(\alpha_{n - 1})$ and $t \leq \alpha_n \leq
    \alpha_{n - 1}$, contradicting the induction hypothesis.
  \item \emph{$C[\eta] \subseteq \Inf(\pi)$:} Let $\beta(t) \leq s \leq t$ for
    some $t \in \eta$. There are infinitely many finite `subsegments'
    $\beta(t) \alpha t$ of $\pi$. Then $s \in
    \Occ(\alpha)$, as by the previous result,
    $t$ cannot be reached from $\beta(t)$ otherwise.
  \end{enumerate}
\end{proof}

The idea behind the soundness argument is rather simple: For every connected
cycle $\eta$ of an $\Rrr(\RR)$-proof, one can find a `shared invariant' which is
common to all cycles in $\eta$. The properties of such invariants allow one to
conclude that reading off the controls $(\Theta, \sigma)$ off any infinite path
through the proof which visits precisely $C[\eta]$ infinitely often must be an
accepting Safra board run and the underlying trace thus must satisfy the trace
condition. The most complicated step of the argument is establishing the
existence of such shared invariants.

For the remainder of this section, fix some cyclic proof system $\RR$ induced
by a trace interpretation $\iota \colon \RR \to \TT_\Aa$.

\begin{proposition}\label{lem:induced-io}
  Let $(C, \lambda, \delta)$ be an $\Rrr(\RR)$ proof and let $\eta$ be a connected cycle of $C$.
  Then there exists some $t \in \eta$ such that the invariant $\theta$ of
  $\beta(t) < t$ is a prefix of the invariant of each $\beta(s) < s$ with $s \in
  \eta$.
\end{proposition}
\begin{proof}
  Observe that one can impose a linear order $\sqsubset$ on $\eta$ such that for
  any $s_0 \in \eta$ condition (ii) of the definition of connected cycles
  can be fulfilled by taking $s_1, \ldots, s_n$ such that they are
  $\sqsubset$-less than $s_0$. Clearly, every downset $\Down_{\sqsubseteq}(s)$
  for $s \in \eta$ is a connected cycle. We prove per induction on the $\sqsubset$-order
  that for every $s \in \eta$, the connected cycle $\Down_{\sqsubseteq}(s)$
  contains a cycle $\beta(t) < t$ whose invariant is a prefix of all invariants
  in $\Down_{\sqsubseteq}(s)$. The claim then follows
  as $\eta = \Down_{\sqsubseteq}(\max_{\sqsubseteq}\eta)$.
  The case of the $\sqsubset$-least element is trivial. Thus pick some $s \neq
  b(\eta)$ and consider $\eta' \coloneq \Down_{\sqsubseteq}(s) \setminus
  \{s\}$. Clearly $\eta' = \Down_{\sqsubseteq}(s')$ for some $s' \in \eta'$ and
  thus has an element $t' \in \eta'$ with invariant
  $\theta'$ which is a prefix of all invariants of cycles in $\eta'$. We first
  prove that the path $\beta(s) \in C[\eta']$: As
  $\Down_{\sqsubseteq}(s)$ is a connected cycle, 
  there needs to be a shortest possible sequence $s_1, \ldots, s_n \in \eta'$ with $0 < n$
  such that $\beta(s_i) \leq s_{i + 1}$ and $s_n = b(\eta)$. Then $\beta(s_{i +
    1}) < \beta(s_i)$ always as otherwise the `detour' through $s_{i + 1}$ could be
  avoided, shortening the sequence. This means that $\beta(s_1) < \beta(s) \leq s_1$,
  meaning that $\beta(s)$ occurs on the path $\beta(s_1) \leq s_1$. As $\theta'$
  is a prefix of the invariant of $\beta(s_1) \leq s_1$, it must also be a
  prefix of the control at $\beta(s)$.
  There are thus only two possibilities for the
  invariant $\theta$ of $\beta(s) \leq s$: Either $\theta'$ is a prefix of it or it is
  a prefix of $\theta'$. In the former case, $t'$ remains the
  element in $\Down_{\sqsubseteq}(s)$ whose invariant is a prefix of all other
  invariants, in the latter $s$ is the new such element by the transitivity of
  the prefix relation.
\end{proof}

\begin{theorem}\label{lem:strip-cyclic-morph}
  The function $\strip \colon \Rrr(\Seq) \to \Seq$ is a cyclic proof system
  morphism.
\end{theorem}
\begin{proof}
  Part of the claim has already been proven in \Cref{lem:strip-simple-morph}.
  It only remains to show that if $\Pi$ is a proof in $\Rrr(\RR)$ then
  $\strip(\Pi)$ is a proof in $\RR$.
  
  Let $\Pi = (C, \lambda, \delta)$ be a cyclic proof of $\Gamma ; (\Theta,
  \sigma)$ in $\Rrr(\RR)$. For this, it suffices to show that every path
  $\widehat{\pi'}: \omega \to \TT_\Aa$ induced by a path $\pi'$ through 
  $\strip(\Pi) = (C', \lambda', \delta')$ satisfies the trace condition of $\TT_\Aa$.
  There must exist a path
  $\pi$ through $\Pi$
  `following' $\pi'$. Let $(\Theta_i, \sigma_i)_{i \in \omega}$ be such that
  $\lambda(\pi'_i) = \Gamma_i ; (\Theta_i, \sigma_i)$. Clearly, $(\Theta_i, \sigma_i)_{i \in
    \omega}$ is a Safra board run of $\tau$ with $\tau_i \coloneq \widehat{\pi'}(i < i + 1)$. By
  \Cref{lem:run-to-tc}, it thus suffices to show that $(\Theta_i, \sigma_i)_{i
    \in \omega}$ is accepting to prove that $P$ satisfies the global trace
  condition. By \Cref{lem:inf-cs}, there exists a connected cycle $\eta$ of $C$
  such that $\Inf(\pi) = C[\eta]$. Thus $\pi$ remains within $C[\eta]$ from some
  point onwards, say from index $N$ ondwards. By
  \Cref{lem:induced-io}, there furthermore exists $t \in \eta$ such that the invariant
  $\theta$ of $\beta(t) < t$ is a prefix of all controls in the annotations in
  $C[\eta]$. This means that $\theta \leq \Theta_i$ for all $i \geq N$. Consider
  the chip $\gamma \coloneq \max(\theta)$: From the previous observation follows that
  $\gamma \in \Theta_i$ for all $i \geq N$. Furthermore, $\gamma$ is reset on
  the path $\beta(t) < t$ as $\theta$ is an invariant of that path. Thus,
  infinitely many $\gamma$-resets take place along $(\Theta_i, \sigma_i)_{i \in
    \omega}$, making it an accepting run as desired.
\end{proof}

\begin{corollary}[Soundness]\label{lem:sound}
  If $\Pi$ is a proof of $\Gamma ; (\Theta, \sigma)$ in $\Rrr(\RR)$ then
  $\strip(\Pi)$ is a proof of $\Gamma$ in $\RR$.
\end{corollary}

\subsection{Completeness and proof search}
\label{sec:complete}

In this section prove completeness of $\Rrr(\RR)$ relative to $\RR$, i.e. any
sequent provable in $\RR$ can be provable in $\Rrr(\RR)$. We do this by showing
that proof search can be performed in $\Rrr(\RR)$ if $\RR$ has a finite amount of
derivation rules. Thus let $\RR$ be such that the set $\TF$ of derivation rules
is finite and let its soundness condition be induced by a trace interpretation
in $\TT_\Aa$. Recall that the objects of \( \TT_\Aa \) are finite sets.

We begin by constructing a proof search system $\Sss(\RR)$ for $\RR$. Similarly to $\Rrr(\RR)$ the sequents of
$\Sss(\RR)$ are $\RR$-sequents annotated with Safra boards. However, the
annotations of $\Sss(\RR)$ are restricted to be $K$-sparse for a suitable $K$.
Crucially, the system
$\Sss(\RR)$ has a finite number of derivation rules if $\RR$ does, a difference
from $\Rrr(\RR)$ which eases proof search. More specifically, each rule of $\Sss(\RR)$
is formed by taking a rule $R \in \RR$, annotating its conclusion with a
$K$-sparse Safra board and annotating the $i$th premise with the $K$-sparse Safra
board resulting from the \emph{greedy} transition via the trace interpretation
map $r_i$. The soundness condition is a `global variant' of the acceptance
condition of the Safra automata in \Cref{def:safra-aut}.
\begin{definition}\label{def:search-system}
Fix $K \coloneq \max\{\abs{\iota(\Gamma)} ~|~ \Gamma \in \Seq\}$.
  The \emph{proof search system} of $\RR$ is the system $\mathrm{S}(\RR) =(\Sss(\Seq),
  \Sss(\RR), \Sss(\rho), \Sss(\SC))$ defined below. The sequents of \( \Sss(\RR) \) are
  expressions $\Gamma ; (\Theta, \sigma)$ with $\Gamma \in \Seq$ a $\RR$-sequent
  and $(\Theta, \sigma) \in \Sb(\Aa, \iota(\Gamma), K)$ a $K$-sparse Safra board
  on $\iota(\Gamma)$.
  The rules of $\Sss(\RR)$ comprise, for each $R \in \RR$ with
  $\rho(R) = (\Gamma, \Delta_1, \ldots, \Delta_n)$ and maps $r_i \colon \iota(\Gamma)
  \to \iota(\Delta_i)$, and each $(\Theta, \sigma) \in
  \Sb(\Aa, \iota(\Gamma), K)$ the rule:
  \[
    \inference[$\Rrr(\Theta, \sigma)$]{
      \Delta_1 ; (\Theta_1, \sigma_1) \text{ where } (\Theta, \sigma) \gstep{r_1} (\Theta_1, \sigma_1)  \quad
      \ldots \quad
      \Delta_n ; (\Theta_n, \sigma_n) \text{ where } (\Theta, \sigma) \gstep{r_n} (\Theta_n, \sigma_n)
    }
    {\Gamma ; (\Theta, \sigma)}
  \]

  A $\Sss(\RR)$-preproof $\Pi$ satisfies the soundness condition $\Sss(\SC)$ if along
  every infinite path $(\Gamma_i ; (\Theta_i, \sigma_i))_{i \in \omega}$ through
  $\Pi$ there
  exists some $N \in \omega$ and $\gamma \in \bigcap_{N \leq i} \Theta_i$ such
  that $\gamma$ is covered infinitely often.
\end{definition}
As we fixed $\gstep{\tau}$ to be an injective function on $K$-sparse Safra
boards, the choice of $R$ and $(\Theta, \sigma)$ specifies the rule $\Rrr(\Theta,
\sigma)$ uniquely.

\begin{lemma}\label{lem:elab}
  The function $\elab \colon \Sss(\Seq) \to \Rrr(\Seq)$ with $\elab(\Gamma; (\Theta,
  \sigma)) \coloneq \Gamma; (\Theta, \sigma)$ can be extended to a proof morphism.
\end{lemma}
\begin{proof}
  Towards this claim, first pick some $\Rrr(\Theta, \sigma) \in \Sss(\RR)$ arranged as follows
  \[
    \inference[$\Rrr(\Theta, \sigma)$]{
      \Delta_1 ; (\Theta_1, \sigma_1) \quad
      \ldots \quad
      \Delta_n ; (\Theta_n, \sigma_n)
    }
    {\Gamma ; (\Theta, \sigma)}
  \]
  Then there is $R \in \RR$ with $\rho(R) = (\Gamma, \Delta_1, \ldots,
  \Delta_n)$ and morphisms $r_i \colon \iota(\Gamma) \to \iota(\Delta_i)$ given by
  the trace interpretation.
  We have to find a corresponding preproof $\elab(\Rrr(\Theta, \sigma))$ in $\Rrr(\RR)$.
  Note that for each $i \leq n$ there is $(\Theta, \sigma)
  \gstep{r_i} (\Theta_i, \sigma_i)$ with the expanded sequence
  the expanded sequence
  \[(\Theta, \sigma) \step{R_{\gamma_1}} (\Theta_r^1, \sigma_r^1) \ldots
    \step{R_{\gamma_k}} (\Theta_r^k, \sigma_r^k)
    \step{P} (\Theta_p, \sigma_p) \step{\tau_i^{\Delta_i}} (\Theta^*_i,
    \sigma^*_i) \step{T} (\Theta_i, \sigma_i) \]
  in which the initial $R_\gamma$- and $P$-steps are shared between all $i \leq
  n$ (see \Cref{lem:greedy-ceil}). Then we may derive $\elab(\Rrr(\Theta, \sigma))$
  as follows:
  \begin{comfproof}
    \AXC{$\Delta_1 ; (\Theta_1, \sigma_1)$}
    \LIC{\textsc{Weak}}
    \UIC{$\Delta_1 ; (\Theta_1^*, \sigma_1^*)$}
    \AXC{$\ldots$}
    \AXC{$\Delta_n ; (\Theta_n, \sigma_n)$}
    \RIC{\textsc{Weak}}
    \UIC{$\Delta_n ; (\Theta_n^*, \sigma_n^*)$}
    \LIC{R}
    \TIC{$\Gamma ; (\Theta_p, \sigma_p)$}
    \LIC{$\textsc{Pop}$}
    \UIC{$\Gamma ; (\Theta_r^k, \sigma_r^k)$}
    \DOC{}
    \LIC{$\textsc{Reset}_{\gamma_2}$}
    \UIC{$\Gamma ; (\Theta_r^1, \sigma_r^1)$}
    \LIC{$\textsc{Reset}_{\gamma_1}$}
    \UIC{$\Gamma ; (\Theta_0, \sigma_0)$}
  \end{comfproof}

  To prove that $\elab$ preserves the soundness condition, let $\Pi = (C,
  \lambda, \delta)$ be a
  $\Sss(\RR)$-proof and let $\elab(\Pi) = (C', \lambda', \delta')$ be its
  $\elab$-translation. Now consider some $t' \in \dom(\beta')$ and the
  associated path $(\Gamma'_i; (\Theta'_i, \sigma'_i))_{i \leq n}$ between $\beta'(t')$ and
  $t'$ in $\elab(\Pi)$. There is a corresponding path $(\Gamma_i ; (\Theta_i,
  \sigma_i))_{i \leq m}$ through $\Pi$. By the soundness condition of $\Sss(\RR)$,
  the path through $\Pi$ which starts at the root and then cycles
  infinitely on $(\Gamma_i ; (\Theta_i, \sigma_i))_{i \leq m}$ must have some
  $\gamma \in \bigcap_{i \leq m} \Theta_i$ which is covered somewhere along
  $(\Gamma_i ; (\Theta_i, \sigma_i))_{i \leq m}$. By \Cref{lem:greedy-ceil} that
  means that $\gamma \in \bigcap_{i \leq n} \Theta'_i$ as well and the fact that
  $\gamma$ is covered somewhere along $(\Gamma_i ; (\Theta_i, \sigma_i))_{i \leq
    m}$ means a $\gamma$-reset must take place somewhere along $(\Gamma'_i;
  (\Theta'_i, \sigma'_i))_{i \leq n}$. It remains to show that the cycle has an
  invariant $\theta$ with $\max(\theta) = \gamma$.
  Taking $\theta \coloneq \{\gamma' \in \Theta ~|~ \gamma' \leq
  \gamma\}$, it remains to show that no
  chip within $\theta$ disappears somewhere along the cycle. But this cannot
  happen, as it cannot be `replaced underneath $\gamma$' before the bud is
  reached, as this would require removing $\gamma$ from the control first.
  Thus, every cycle in $\elab(\Pi)$ has an invariant $\theta$ with an
  accompanying $\max(\theta)$-reset, meaning $\elab(\Pi)$ is a $\Rrr(\RR)$-proof.
\end{proof}

We employ notion of proof search systems to prove completeness of $\Rrr(\RR)$
  relative $\RR$: Every sequent provable in $\RR$ can also be proven in $\Rrr(\RR)$.
  Because the proof is based on a proof search procedure, the result we obtain is
  even stronger: The $\Rrr(\RR)$-proof will essentially be an unfolding of the $\RR$-proof.
  Fix a preproof (of any derivation system $\TS$) $\Pi = ((T, \beta), \lambda,
  \delta)$, \emph{unfolding $\Pi$ at bud $t \in \dom(\beta)$} yields the
  preproof $\Pi' \coloneq ((T', \beta'), \lambda', \delta')$ with
  \[
    T' \coloneq T \cup \{tu ~|~ \beta(t)u \in T\}
    \qquad
    \lambda'(s) \coloneq 
    \begin{cases}
      \lambda(s) & s \in T \\
      \lambda(\beta(t)u) & s = tu 
    \end{cases}
  \]
  $\delta'$ defined analogously to $\lambda'$ and $\beta' \coloneq \beta
  \setminus \{(t, \beta(t))\} \cup \{(tu, c)\}$ where $t = \beta(t)u$ and either
  $c = t$ or $c = \beta(t)$. A preproof $\Pi'$ of $\Pi$ is an \emph{unfolding} of $\Pi$
  if $\Pi'$ can be arrived at by repeatedly unfolding $\Pi$.

\begin{theorem}[Completeness]\label{lem:complete}
  Let $(\Seq, \RR, \rho, \SC)$ be a cyclic proof system induced by a trace
  interpretation on $\TT_\Aa$. If there is a cyclic proof $\Pi$ of $\Gamma \in \Seq$
  in $\RR$ then there is a proof $\Pi'$ of $\Gamma ; (\emptyset, (x, a) \mapsto
  \emptyset)$ in $\Rrr(\RR)$. Furthermore, $\strip(\Pi')$ is an unfolding of $\Pi$.
\end{theorem}
\begin{proof}
  Let $\Pi = (C, \lambda, \delta)$ be a proof of $\Gamma$ in $\RR$. Consider the
  subsystem $\RR' \coloneq (\im(\lambda), \im(\delta), \rho \uh \RR', C \cap
  \Pp(\RR'))$ of $\RR$. As $\Pi$ is finite, so is $\RR'$. Now fix
  \[M \coloneq \{r_i \colon \iota(\Gamma) \to \iota(\Delta_i) ~|~ R \in \RR', \rho(R)
    = (\Gamma, \Delta_1, \ldots, \Delta_n), r_i \text{ given by
      the trace interpretation}\}\]
  and construct the Safra automaton $\SSB(\Aa, i(\Gamma), M) = (M, Q, s,
  \delta_\SSB, R_\SSB)$ according to \Cref{def:safra-aut}.
  We construct the Rabin tree automaton $\SA = (\Seq', Q', \Delta, s', R')$ with
  \begin{align*}
    Q' & \coloneq \{(s, (\Theta, \sigma)) ~|~ s \in C \setminus \dom(\beta), \Gamma \coloneq \lambda(s), (\iota(\Gamma), (\Theta, \sigma)) \in Q\} \\
    \Delta' & \coloneq \{((s, (\Theta, \sigma)), (\Delta_1, \ldots, \Delta_n), ((t_1, (\Theta_1, \sigma_n)), \ldots, (t_n, (\Theta_n, \sigma_n)))) ~|~ \\
       & \qquad \text{if } s \in T \setminus \Leaf(T), \Chld(s) = \{t_1, \ldots, t_n\}, \rho(\delta(s)) = (\Gamma, \Delta_1, \ldots, \Delta_n), \\
       & \qquad r_i \colon \iota(\Gamma) \to \iota(\Delta_i) \text{ given by the trace interpretation } \\
       & \qquad \text{and } \delta_\SSB((\iota(\Gamma), (\Theta, \sigma)), r_i) = (\iota(\Delta_i), (\Theta_i, \sigma_i))\} \\
    s' & \coloneq (\varepsilon, (\Theta_0, \sigma_0)) \text{ where } s = (\iota(\lambda(\varepsilon)), (\Theta_0, \sigma_0)) \\
    R' & \coloneq \{(\{(s, (\Theta, \sigma)) ~|~ (\iota(\lambda(s)), (\Theta, \sigma)) \in G\}, \{(s, (\Theta, \sigma)) ~|~ (\iota(\lambda(s)), (\Theta, \sigma)) \in B\}) ~|~ (G, B) \in R_\SSB\}
  \end{align*}
  It is easily observed that $L(\SA)$ contains precisely the `infinite unfolding' of $\Pi$
  which has a successful run of
  $\Sb(\Aa, i(\Gamma), K)$ along the paths $P \colon \omega \to \TT_\Aa$ of each of
  their branches. In other words, the only tree in $L(\SA)$ corresponds to the unfolding of
  $\Pi$.
  By \Cref{lem:rabin-det}, there exists a regular tree in $L(\SA)$ which has
  a regular run on $\SA$. This run may be turned into an
  $\Sss(\RR')$-preproof $\Pi'$ by
  replacing each step $((s, (\Theta, \sigma)), \langle \Delta_1, \ldots,
  \Delta_n \rangle, ((t_1, (\Theta_1, \sigma_n)), \ldots, (t_n,
  (\Theta_n, \sigma_n))))$ corresponding to the rule $R \coloneq \delta(s)$ with the
  corresponding $\Sss(\RR')$-rule $\Rrr(\Theta, \sigma)$:
  \[\inference[$\Rrr(\Theta, \sigma)$]{\Delta_1; (\Theta_1, \sigma_1) \quad
      \ldots \quad \Delta_n ; (\Theta_n, \sigma_n)}{\Gamma ; (\Theta, \sigma)}\]
  As the run satisfies the Rabin condition $R'$, the preproof $\Pi'$ satisfies the
  soundness condition of $\Sss(\RR')$. The conclusion of $\Pi'$ is $\Gamma ;
  (\emptyset, (x, a) \mapsto \emptyset)$ as this corresponds to the initial
  state $s'$ of $\Aa$. Now, by \Cref{lem:elab}, $\elab(\Pi')$ is an
  $\Rrr(\RR')$-proof (and thus an $\Rrr(\RR)$-proof). As the states of $\SA$ are
  labeled by the nodes of $C$, the regular run on $\SA$ must correspond to an
  unfolding of $\Pi$. Thus, $\strip(\Pi')$ must be an unfolding of $\Pi$ as well.
\end{proof}

\section{Deriving Concrete Reset Proof Systems}
\label{sec:concrete}

In this section, we apply the results do derive reset proof systems for 3 cyclic
proof systems from the literature: cyclic arithmetic~\Cref{sec:pa}, cyclic
Gödel's T~\Cref{sec:godel-t} and the modal $\mu$-calculus~\Cref{sec:modal-mu}.
For each system, the abstract reset system $\Rrr(\RR)$ will serve as a starting
point. However, the reset systems we derive all differ from the `na\"ive' system
$\Rrr(\RR)$ by a few `ergonomic adjustments' and a more syntactic annotation mechanism.
In each of the following sections, we begin by recalling the original cyclic
proof system formulated in terms of a global trace condition before defining our
proposed reset proof systems.

\subsection{Peano Arithmetic}
\label{sec:pa}

Cyclic arithmetic was first proposed by Alex Simpson
in~\cite{simpsonCyclicArithmeticEquivalent2017}. It is a cyclic proof
system which proves the same theorems as Peano arithmetic.

\subsubsection{Cyclic Arithmetic}
\label{sec:ca}

The term and formula
languages of $\CA$ are given below. The formula language is non-standard, treating
inequality $s < t$ as a primitive, rather than a defined notion. As will become
clear below, this eases the definition of the global trace condition of $\CA$.
\begin{align*}
  s,t \in \term \langeq~& x ~|~ 0 ~|~ S s ~|~ s + t ~|~ s \cdot t \\
  \varphi, \psi \in \form \langeq~& s = t ~|~ s < t ~|~ \bot ~|~ \varphi \wedge \psi ~|~ \varphi \vee \psi ~|~ \varphi \to \psi ~|~ \forall x.\varphi ~|~ \exists x. \varphi
\end{align*}
Denote by $[t / x]$ the usual \emph{substitution operation}, substituting the term
$t$ into all free occurrences of the variable $x$ in a term or formula. This is
a partial operation, $\varphi[t / x]$ being undefined when the free variables in
$t$ are not distinct from the bound variables in $\varphi$. Henceforth, writing
$\varphi[t / x]$ will double as an assertion of the resulting formula being defined.

\begin{definition}\label{def:ca-rules}
  The \emph{sequents} of $\CA$ are expressions $\Gamma \sdash \Delta$ where
  $\Gamma, \Delta$ are finite sets of formulas. The set of $\CA$ sequents is
  denoted by $\Seq_\CA$. Write $\Gamma, \varphi$ for
  $\Gamma \cup \{\varphi\}$ and $\Gamma, \Gamma'$ for $\Gamma \cup \Gamma'$.
  The \emph{derivation rules} of $\CA$ comprise of the following choice of standard rules for
  first-order logic,
  \begin{center}
    \begin{tabular}{ccc}
	\inference[\RAx]{}{\Gamma, \varphi \sdash \varphi, \Delta{}}
	&
    \inference[$\to$L]{\Gamma{}, \varphi{} \sdash{} \Delta{} \quad \Gamma{} \sdash{}
      \psi{}, \Delta{}}{\Gamma{}, \varphi{} \to \psi{} \sdash{} \Delta{}}
	&
    \inference[$\to$R]{\Gamma{}, \varphi{} \sdash{} \psi{}, \Delta{}}{\Gamma{} \sdash{} \varphi{} \to \psi{}, \Delta{}}
\\[1.5em]
    \inference[$\wedge$L]{\Gamma{}, \varphi{}, \psi{} \sdash{} \Delta{}}{\Gamma{},
      \varphi{} \wedge{} \psi{} \sdash \Delta{}}
	&
    \inference[$\wedge$R]{\Gamma \sdash \varphi, \Delta \quad \Gamma \sdash \psi, \Delta}{\Gamma \sdash \varphi \wedge
      \psi, \Delta}
	&
    \inference[$\vee$L]{\Gamma, \varphi \sdash \Delta \quad \Gamma, \psi \sdash \Delta}{\Gamma, \varphi \vee{}
      \psi \sdash \Delta}
\\[1.5em]
    \inference[$\vee$R]{\Gamma{} \sdash \varphi{}, \psi{}, \Delta{}}{\Gamma{} \sdash
      \varphi{} \vee{} \psi{}, \Delta{}}
	&
    \inference[$\forall$L]{\Gamma, \varphi[t / x] \sdash \Delta}{\Gamma, \forall x.\varphi \sdash \Delta}
	&
    \inference[$\forall$R]{\Gamma \sdash \varphi, \Delta \quad x \not\in
      \FV(\Gamma, \Delta)}{\Gamma
      \sdash \forall x.\varphi, \Delta}
\\[1.5em]
    \inference[$\exists$L]{\Gamma, \varphi \sdash \Delta \quad x \not\in
      \FV(\Gamma, \Delta)}{\Gamma, \exists x. \varphi \sdash \Delta}
	&
    \inference[$\exists$R]{\Gamma \sdash \varphi[t / x], \Delta}{\Gamma \sdash \exists x. \varphi, \Delta}
	&
    \inference[$\bot$L]{}{\Gamma, \bot \sdash \Delta}
\\[1.5em]
    \multicolumn{2}{c}{\inference[$=$L]{\Gamma[t / x, s / y] \sdash \Delta[t / x, s / y] \quad x,
      y \not\in \FV(s, t)}{\Gamma[s / x, t / y], s = t \sdash \Delta[s / x, t
      / y]}}
	&
    \inference[$=$R]{}{\Gamma \sdash t = t, \Delta}
    \end{tabular}
  \end{center}
  with the following structural rules,
  \begin{mathpar}
    \inference[\RWk]{\Gamma \sdash \Delta}{\Gamma, \Gamma' \sdash \Delta, \Delta'}

    \inference[\RCut]{\Gamma, \varphi \sdash \Delta \quad \Gamma \sdash \varphi,
      \Delta}{\Gamma \sdash \Delta}
    
    \inference[\RSub]{\Gamma \sdash \Delta}{\Gamma[s / x] \sdash \Delta[s / x]}
  \end{mathpar}
  %
  %
  %
  the following arithmetic-specific axioms%
  \begin{align*}
    s < t, t < u & \sdash s < u &
    s < t & \sdash Ss < St &
    & \sdash s + St = \Sss(s + t) \\
    s < t, t < s & \sdash &
    & \sdash s < t, s = t, t < s &
    & \sdash t \cdot 0 = 0 \\
    s < t, t < Ss & \sdash &
    & \sdash t < St &
    & \sdash s \cdot St = (s \cdot t) + s \\
    t < 0 & \sdash &
    & \sdash t + 0 = t
  \end{align*}
  and the arithmetic-specific derivation rule%
  \begin{mathpar}
    \inference[S]{\Gamma, t = Sx \sdash \Delta \quad x \text{ fresh}}{\Gamma, 0 < t \sdash \Delta}
  \end{mathpar}
  $ $
\end{definition}

Observe that the assumption-free, non-cyclic preproofs using the rules of $\CA$ and the
\emph{induction scheme} $(\forall x. (\forall y. y < x \to \varphi[y /
x]) \to \varphi) \to \forall x. \varphi$ prove
exactly the theorems of Peano arithmetic. Cyclic arithmetic also proves exactly the
theorems of Peano arithmetic, trading the induction scheme for a global trace
condition (see \cite[Theorem 6]{simpsonCyclicArithmeticEquivalent2017} for a
proof of this).

A term $t$ \emph{occurs} in a sequent $\Gamma \sdash
\Delta$ if it appears, possibly as a subterm of another term, in a formula in
$\Gamma$ or $\Delta$. Write $\term(\Gamma \sdash \Delta)$ for the set of
terms occurring in $\Gamma \sdash \Delta$.
Let $R \in \CA$ be such that $\rho(R) = (\Gamma \sdash \Delta, \Gamma_1 \sdash \Delta_1, \ldots, \Gamma_n
\sdash \Delta_n)$, i.e. with $\Gamma \sdash \Delta$ as its conclusion and
$\Gamma_i \sdash \Delta_i$ as one of its premises.
Fix $t \in \term(\Gamma \sdash \Delta)$ and $t' \in \term(\Gamma_i \sdash
\Delta_i)$. The term $t'$ is called a
\emph{precursor} of $t$, denoted $t' \leftarrow^i_R t$ if one of the following three
conditions holds:
\begin{itemize}
\item $R$ is an instance of $(\textsc{Sub})$ and $\Gamma = \Gamma'[s / x], \Delta =
  \Delta'[s / x]$ and $t = t'[s / x]$;
\item or $R$ is an instance of a rule \emph{other than} $(\RSub)$ and $t = t'$;
\item or \( R \) is an instance of \( (=\hspace{-0.25em}\text{L})\) and $\Gamma = \Gamma''[s / x,
  t / y], \Gamma' = \Gamma''[t / x, s / y]$ and analogously for the $\Delta$
  and there exists a term $t''$ such that $t = t''[s / x, t / y]$ and $t' =
  t''[t / x, s / y]$.
\end{itemize}
Recall that the booleans $\BB = \{0, 1\}$ with the usual join operation and
$\alpha \coloneq 1$ form an activation algebra. This is the most natural
activation algebra for the specification of the global trace condition of $\CA$.

\begin{definition}
  The \emph{trace interpretation} $\iota \colon \CA \to \TT_\BB$ is given by $\iota(\Gamma
  \sdash \Delta) \coloneq \term(\Gamma \sdash \Delta)$ and for any $R \in \CA$
  with $\rho(R) = (\Gamma \sdash \Delta, \Gamma_1 \sdash \Delta_1, \ldots,
  \Gamma_n \sdash \Delta_n)$ the trace map $r_i \colon \term(\Gamma \sdash \Delta) \to
  \term(\Gamma_i \sdash \Delta_i)$ is given by%
  \begin{align*}
    r_i \coloneq & \{(t, 0, t') ~|~ t \in \term(\Gamma \sdash \Delta), t' \in \term(\Gamma_i
                   \sdash \Delta_i) \text{ and } t' \leftarrow^i_R t\} \,\cup \\
                 & \{(t, 1, s) ~|~ t \in \term(\Gamma \sdash \Delta), t', s \in \term(\Gamma_i \sdash \Delta_i) \text{ and } t \leftarrow^i_R t' \text{ and } s < t' \in \Gamma_i\}
  \end{align*}
  This trace interpretation induces the soundness condition of $\CA$ as described in
  \Cref{def:tc-induced}.
\end{definition}

\subsubsection{Reset Arithmetic}
\label{sec:ra}

We present a cyclic proof system for Peano arithmetic called \emph{reset
  arithmetic} $\RA$. It is based on the reset system $\Rrr(\CA)$ induced by $\CA$
with some slight modifications.%

  Sequents of $\RA$ are expressions
  $\Theta ; \sigma \rsep \Gamma \sdash \Delta$
  where \( \Gamma \sdash \Delta \) is a \( \CA \)-sequent, $\Theta$ is a
  sequence of distinct characters called the \emph{control} and $\sigma$ is a finite set of \emph{assignments} $t \mapsto u$ where $t$ is a
  term in $\Gamma, \Delta$ and $u$ is a subsequence $u \sqsubseteq \Theta$.
  The set of $\RA$-sequents is denoted $\Seq_\RA$.
  For an assignment $t \mapsto u$,
  write $(t \mapsto u)[s / x]$ for $t[s / x] \mapsto u$. This notation extends
  to sets of assignments $\sigma[s / x]$.
  Sequents \( \varepsilon ; \emptyset \rsep \Gamma \sdash \Delta \) with empty control are identified with \( \CA \)-sequents \( \Gamma \sdash \Delta \).

\begin{definition}\label{def:ra-rules}
  The \emph{derivation rules} of reset arithmetic are listed below.
  In each rule, $\Theta' ;
  \sigma'$ denotes the result of first removing all assignments to terms not occurring
  in the premise from $\sigma$ and then removing all letters of $\Theta$ which
  are not assigned to at least one term.
  The rules of $\RA$ contain the rules of $\CA$ adjusted to `properly treat'
  the control $\Theta; \sigma$. Observe that the $\RWk$ also allows for the
  `weakening' of the assignments $\sigma$.
  \begin{mathpar}
    \inference[\RAx]{}{\Theta; \sigma \rsep \Gamma, \varphi \sdash \varphi, \Delta{}}

    \inference[$\to$L]{\Theta'; \sigma' \rsep \Gamma{}, \varphi{} \sdash{}
      \Delta{} \quad \Theta{}'; \sigma{}' \rsep{} \Gamma{} \sdash{}
      \psi{}, \Delta{}}{\Theta{}; \sigma{} \rsep{} \Gamma{}, \varphi{} \to \psi{} \sdash{} \Delta{}}

    \inference[$\to$R]{\Theta{}; \sigma{} \rsep{} \Gamma{}, \varphi{} \sdash{}
      \psi{}, \Delta{}}{\Theta{} ; \sigma{} \rsep{} \Gamma{} \sdash{} \varphi{} \to \psi{}, \Delta{}}

    \inference[$\wedge$L]{\Theta{} ; \sigma{} \rsep{} \Gamma{}, \varphi{},
      \psi{} \sdash{} \Delta{}}{\Theta{} ; \sigma{} \rsep{} \Gamma{},
      \varphi{} \wedge{} \psi{} \sdash \Delta{}}

    \inference[$\wedge$R]{\Theta{}' ; \sigma{}' \rsep{} \Gamma \sdash \varphi,
      \Delta \quad \Theta{}' ; \sigma{}' \rsep{} \Gamma \sdash \psi,
      \Delta}{\Theta{} ; \sigma{} \rsep{} \Gamma \sdash \varphi \wedge
      \psi, \Delta}

    \inference[$\vee$L]{\Theta{}'; \sigma{}' \rsep{} \Gamma, \varphi \sdash
      \Delta \quad \Theta{}' ; \sigma{}' \rsep{} \Gamma, \psi \sdash
      \Delta}{\Theta{} ; \sigma{} \rsep{} \Gamma, \varphi \vee{}
      \psi \sdash \Delta}

    \inference[$\vee$R]{\Theta{} ; \sigma{} \rsep{} \Gamma{} \sdash \varphi{},
      \psi{}, \Delta{}}{\Theta{} ; \sigma{} \rsep{} \Gamma{} \sdash
      \varphi{} \vee{} \psi{}, \Delta{}}

    \inference[$\forall$L]{\Theta ; \sigma \rsep \Gamma, \varphi[t / x] \sdash
      \Delta}{\Theta ; \sigma \rsep \Gamma, \forall x.\varphi \sdash \Delta}

    \inference[$\forall$R]{\Theta ; \sigma \rsep \Gamma \sdash \varphi, \Delta \quad x \not\in
      \FV(\Gamma, \Delta)}{\Theta ; \sigma \rsep \Gamma
      \sdash \forall x.\varphi, \Delta}

    \inference[$\exists$L]{\Theta ; \sigma \rsep \Gamma, \varphi \sdash \Delta \quad x \not\in
      \FV(\Gamma, \Delta)}{\Theta ; \sigma \rsep \Gamma, \exists x. \varphi \sdash \Delta}

    \inference[$\exists$R]{\Theta' ; \sigma' \rsep \Gamma \sdash \varphi[t /
      x], \Delta}{\Theta ; \sigma \rsep \Gamma \sdash \exists x. \varphi, \Delta}

    \inference[$\bot$L]{}{\Theta ; \sigma \rsep \Gamma, \bot \sdash \Delta}

    \inference[$=$L]{\Theta ; \sigma[t / x, s / y] \rsep \Gamma[t / x, s / y] \sdash
      \Delta[t / x, s / y] \quad x, y \not\in \FV(s, t)}{\Theta ; \sigma[s / x, t / y] \rsep \Gamma[s / x, t / y], s = t \sdash \Delta[s / x, t
      / y]}

    \inference[$=$R]{}{\Theta ; \sigma \rsep \Gamma \sdash t = t, \Delta}

    \inference[\RWk]{\Theta' ; \sigma \rsep \Gamma \sdash \Delta}{\Theta ; \sigma, \sigma^* \rsep \Gamma, \Gamma^* \sdash \Delta, \Delta^*}

    \inference[\RCut]{\Theta ; \sigma \rsep \Gamma, \varphi \sdash \Delta \quad \Theta ; \sigma \rsep \Gamma \sdash \varphi,
      \Delta}{\Theta ; \sigma \rsep \Gamma \sdash \Delta}

    \inference[\RSub]{\Theta ; \sigma[s / x] \rsep \Gamma \sdash \Delta}{\Theta ; \sigma \rsep \Gamma[s / x] \sdash \Delta[s / x]}

    \inference[$S$]{\Theta ; \sigma \rsep \Gamma, t = Sx \sdash \Delta \quad x
      \text{ fresh}}{\Theta ; \sigma \rsep \Gamma, 0 < t \sdash \Delta}
  \end{mathpar}
  The axioms of $\RA$ are the arithmetical \emph{axioms} of $\CA$ listed below.
  This means $\Theta; \sigma \rsep \Gamma
  \sdash \Delta$ for any $\CA$-sequent $\Gamma \sdash \Delta$ below and any control
  $\Theta; \sigma$ is an axiom of $\RA$.
  \begin{align*}
    s < t, t < u & \sdash s < u &
                                  s < t & \sdash Ss < St &
    & \sdash s + St = \Sss(s + t) \\
    s < t, t < s & \sdash &
                                        & \sdash s < t, s = t, t < s &
    & \sdash t \cdot 0 = 0 \\
    s < t, t < Ss & \sdash &
                                        & \sdash t < St &
    & \sdash s \cdot St = (s \cdot t) + s \\
    t < 0 & \sdash &
                                        & \sdash t + 0 = t
  \end{align*}
  Lastly, $\RA$ features three derivation rules which have no corresponding rule
  in $\CA$.
  \begin{mathpar}
    \inference[\RFocus]{\Theta ; \sigma, (t \mapsto \varepsilon) \rsep \Gamma
      \sdash \Delta \quad t \in \term(\Gamma, \Delta)}{\Theta ; \sigma \rsep
      \Gamma \sdash \Delta}

    \inference[$\RReset_a$]{\Theta' ; \sigma, (t_1 \mapsto ua), \ldots, (t_n
      \mapsto ua) \rsep \Gamma \sdash \Delta \quad a \text{ does not occur in }
    \sigma}{\Theta ; \sigma, (t_1 \mapsto uau_1), \ldots, (t_n
      \mapsto uau_n) \rsep \Gamma \sdash \Delta}

    \inference[$<$L]{\Theta a ; \sigma, (s \mapsto ua) \rsep \Gamma \sdash
      \Delta \quad a \text{ fresh}}{\Theta ; \sigma, (t \mapsto u) \rsep \Gamma, s < t \sdash \Delta}
  \end{mathpar}

  An $\RA$-preproof satisfies the soundness condition of $\RA$ if
  every pair of bud $t
  \in \dom(\beta)$ and companion $\beta(t)$ has an \emph{invariant}\emph{:} there
  exists a letter $a$ such that $a$ occurs in all of the controls
  between $t$ and $\beta(t)$, the prefix of $a$ is
  constant across these controls, and the $\textsc{Reset}_a$ rule is applied between
  $t$ and $\beta(t)$.
  An $\RA$-proof is a \emph{proof of $\CA$-sequent $\Gamma \sdash \Delta$} if its root is labeled
  $\varepsilon ; \emptyset \rsep \Gamma \sdash \Delta$.
  If there exists an $\RA$ proof of $\Gamma \sdash \Delta$ write \emph{$\RA \vdash
  \Gamma \sdash \Delta$}.
\end{definition}

The proof system $\RA$ features one `ergonomic adjustment' differentiating it
from the `na\"ive' reset system $\Rrr(\CA)$. In $\Rrr(\CA)$, the rules corresponding
to $\CA$-rules add new chips to the control $(\Theta, \sigma)$ whenever progress
takes place, i.e. whenever there are inequalities $s < t$ in $\Gamma$ of the
assumption. This can quickly get out hand, making the handling of the control
quite unwieldy. In $\RA$, the $\CA$-correspondents never add chips to the
control, only remove them if they are no longer used. Instead, the $<$L rule
allows the prover to add chips corresponding to the progress embodied by an
inequality $s < t$ in $\Gamma$.

We prove soundness of $\RA$ relative to $\PA$ by constructing a proof morphism
$\embed \colon \RA \to \Rrr(\CA)$. Recall that $\Rrr(\CA)$ is the reset proof system

\begin{lemma}
  There is a function $\embed \colon \Seq_\RA \to \Rrr(\Seq_\CA)$ which is defined by
  \[\embed(\Theta ; \sigma \rsep \Gamma \sdash \Delta) \coloneq \Gamma \sdash
    \Delta ; (\ol{\Theta}, \ol{\sigma}
    )\]
  where $\ol{\Theta}$ is the set $\{a ~|~ a \in \Theta \}$ ordered
  according to the letters' positions in $\Theta$ and $\ol{\sigma}(t) \coloneq
  \{\{x ~|~ x \in u \} ~|~ t \mapsto u \in \sigma\}$. It can be extended into a
  proof morphism $\embed \colon \RA \to \Rrr(\CA)$.
\end{lemma}
\begin{proof}
  To translate rules corresponding to $\CA$-rules, we need to account for
  the
  difference in how $\RA$ and $\Rrr(\CA)$ treat inequalities left of $\sdash$
  explained above. Pick a rule $R \in \RA$ with a corresponding $\CA$-rule
  (i.e. \( R \) is not an instance of \( < \)L, \textsc{Wk}, \textsc{Focus} or
  \textsc{Reset}). It is translated as the $\Rrr(\CA)$-preproof as below. 
  Here, we denote the chips that were `erroneously' added by $R$ by $u_i$ and by
  $\sigma'_i$ the `erroneous' stacks. For this, we employ the notation
  $\sigma \cup \sigma'$ to denote the function $(x, a) \mapsto
  \sigma(x, a) \cup \sigma'(x, a)$. Observe that for every $s < t \in \Gamma_i$,
  the trace map $r_i$ dictating the Safra board transition $(\ol{\Theta},
  \ol{\sigma}) \step{r_i} (\ol{\Theta_i} \oplus u_i, \ol{\sigma_i} \cup \sigma'_i)$
  contains two transitions concerning the predecessor $t' \leftarrow^i_R t$ of $t$:
  $(t, 0, t')$ and $(t, 1, s)$. The latter is the cause of a chip being added
  `erroneously'. The former ensures that the stacks on $t$ are not removed from
  $\ol{\sigma_i}$ but simply `reassigned' to $t'$, just as is done in $\RA$.
  This guarantees that $\ol{\sigma_i}$ is part of the control in each of the
  premises.
  \begin{lrbox}{\mypt}
    \begin{varwidth}{\linewidth}
      \begin{comfproof}
        \AXC{$\Gamma_1 \sdash \Delta_1 ; (\ol{\Theta_1}, \ol{\sigma_1})$}
        \LSC{\textsc{Weak}}
        \UIC{$\Gamma_1 \sdash \Delta_1 ; (\ol{\Theta_1} \oplus u_1, \ol{\sigma_1} \cup \sigma'_1)$}
        \AXC{$\dotsm$}
        \AXC{$\Gamma_n \sdash \Delta_n ; (\ol{\Theta_n}, \ol{\sigma_n})$}
        \LSC{\textsc{Weak}}
        \UIC{$\Gamma_n \sdash \Delta_n ; (\ol{\Theta_n} \oplus u_n, \ol{\sigma_n} \cup \sigma'_n)$}
        \LSC{R}
        \TIC{$\Gamma \sdash \Delta ; (\ol{\Theta}, \ol{\sigma})$}
      \end{comfproof}
    \end{varwidth}
  \end{lrbox}
  \begin{multline*}
    \inference[R]{\Theta_1 ; \sigma_1 \rsep \Gamma_1 \sdash \Delta_1 \quad
    \dotsm \quad \Theta_1 ; \sigma_1 \rsep \Gamma_1 \sdash \Delta_1}{\Theta ;
    \sigma \rsep \Gamma \sdash \Delta} \quad \stackrel{\embed}{\leadsto} \\ \\
    \usebox{\mypt}
  \end{multline*}

  The \textsc{Reset} and \textsc{Focus} rules correspond to applications of the
  \textsc{Reset} and \textsc{Pop} rules in $\Rrr(\CA)$. The \textsc{Wk} is
  translated to a combination of the \textsc{Wk}-rule of $\CA$ to `weaken' in
  $\Gamma \sdash \Delta$ (taking care of
  `erroneous' chips as above) and the
  \textsc{Weak}-rule of $\Rrr(\CA)$ to `weaken' in $\sigma$.
  The only rule for which the translation via $\embed$ remain open is $<$L.
  This translation is achieved by a (possibly vacuous)
  application of the \textsc{Wk}-rule from $\CA$, `simulating' the removal of
  the inequality $s < t$, followed by the \textsc{Weak}-rule of $\Rrr(\CA)$ to
  remove all `erroneously' added chips and stacks, i.e. all except the one
  induced by the inequality $s < t$.
  \begin{lrbox}{\mypt}
    \begin{varwidth}{\linewidth}
      \begin{comfproof}
        \AXC{$\Gamma \sdash \Delta ; (\ol{\Theta a}, \ol{\sigma, (s, 0) \mapsto ua})$}
        \LSC{\textsc{Weak}}
        \UIC{$\Gamma \sdash \Delta ; (\ol{\Theta a v} , \ol{\sigma, (t \mapsto u)} \cup \ol{(s, 0) \mapsto ua} \cup \sigma')$}
        \LSC{\textsc{Wk}}
        \UIC{$\Gamma, s < t \sdash \Delta ; (\ol{\Theta} , \ol{\sigma, (t \mapsto u)})$}
      \end{comfproof}
    \end{varwidth}
  \end{lrbox}
  \begin{multline*}
    \inference[$<$L]{\Theta a ; \sigma, (s \mapsto ua) \rsep \Gamma \sdash
      \Delta \quad a \text{ fresh}}{\Theta ; \sigma, (t \mapsto u) \rsep \Gamma, s < t \sdash \Delta}
    \quad \stackrel{\embed}{\leadsto} \\
    \usebox{\mypt}
  \end{multline*}

  By comparing the soundness conditions of $\RA$ and $\Rrr(\CA)$, it is easily
  observed that $\embed$ maintains the soundness condition of $\RA$.
\end{proof}

\begin{corollary}[Soundness]\label{lem:ra-sound}
  If $\RA \vdash \Gamma \sdash \Delta$ then $PA \vdash \Gamma \sdash \Delta$.
\end{corollary}
\begin{proof}
  If $\Pi$ is a proof of $\varepsilon ; \emptyset \rsep \Gamma \sdash \Delta$
  in $\RA$ then $\embed(\Pi)$ is a proof of $\Gamma \sdash \Delta ; (\emptyset ,
  ((x, a) \mapsto \emptyset))$ in $\Rrr(\CA)$ and thus $\strip(\embed(\Pi))$ a
  proof of $\Gamma \sdash \Delta$ in $\CA$. As $\CA$ proves the same sequents as
  $\PA$ (see \cite[Theorem 6]{simpsonCyclicArithmeticEquivalent2017}) there must
  also be a proof of $\Gamma \sdash \Delta$ in $PA$.
\end{proof}

Let $\TF$ be a finite fragment of $\CA$. To conclude completeness of $\RA$
relative to $\PA$, we
construct a proof morphism $\search \colon \Sss(\TF) \to \RA$. For this, recall that
$\Sss(\TF)$ is the `proof search system' induced by the derivation rules in $\TF$
and the trace interpretation of $\CA$.

\begin{lemma}
  There is a function $\search \colon \Sss(\Seq_\TF) \to \Seq_\RA$
  with 
  \[\search(\Gamma \sdash \Delta ; (\Theta, \sigma) ) \coloneq \widehat{\Theta}
    ; \widehat{\sigma} \vdash \Gamma \sdash \Delta\]
  where $\widehat{S} \in \Theta^*$ for $S \subseteq \Theta$ is the duplicate-free sequence of length
  $\abs{\Theta}$ which is strictly sorted according to $\Theta$ and, if
  $(\Theta, \sigma)$ is a $K$-sparse Safra board on $X$, then
  \[\widehat{\sigma} \coloneq \{s \mapsto \widehat{S} ~|~ x \in X, s \in
    \term(\Gamma, \Delta) \text{ and } \sigma(s, 0)
    = \{S\}\} \]
  The function can be extended into a proof morphism $\search \colon \Sss(\TF) \to \RA$.
\end{lemma}
\begin{proof}
  Towards this claim, first pick some $\Rrr(\Theta, \sigma) \in \Sss(\TF)$ arranged as follows
  \[
    \inference[$\Rrr(\Theta, \sigma)$]{
      \Gamma_1 \sdash \Delta_1 ; (\Theta_1, \sigma_1) \quad
      \ldots \quad
      \Gamma_n \sdash \Delta_n ; (\Theta_n, \sigma_n)
    }
    {\Gamma \sdash \Delta ; (\Theta, \sigma)}
  \]
  Then there is $R \in \TF$ with $\rho(R) = (\Gamma \sdash \Delta, \Gamma_1
  \sdash \Delta_1, \ldots,
  \Gamma_n \sdash \Delta_n)$ and morphisms $r_i \colon \iota(\Gamma \sdash \Delta)
  \to \iota(\Gamma_i \sdash \Delta_i)$ given by
  the trace interpretation. Then for each $i \leq n$ there is $(\Theta, \sigma)
  \gstep{r_i} (\Theta_i, \sigma_i)$ with the expanded sequence
  the expanded sequence
  \[(\Theta, \sigma) \step{R_{\gamma_1}} (\Theta_r^1, \sigma_r^1) \ldots
    \step{R_{\gamma_k}} (\Theta_r^k, \sigma_r^k)
    \step{P} (\Theta_p, \sigma_p) \step{r_i} (\Theta^*_i,
    \sigma^*_i) \step{T} (\Theta_i, \sigma_i) \]
  in which the initial $R_\gamma$- and $P$-steps are shared between all $i \leq
  n$ (see \Cref{lem:greedy-ceil}). Similarly to \Cref{lem:elab}, we may derive
  the following in $\RA$:
  \begin{comfproof}
    \AXC{$\widehat{\Theta_1} ; \widehat{\sigma_1} \rsep \Gamma_1 \sdash \Delta_1$}
    \LSC{\textsc{Wk}}
    \UIC{$\widehat{\Theta_1'} ; \widehat{\sigma_1'} \rsep \Gamma_1 \sdash \Delta_1$}
    \DOC{}
    \LSC{$<$L}
    \UIC{$\widehat{\Theta_p} ; \widehat{\sigma_p} \rsep \Gamma_1 \sdash \Delta_1$}
    \AXC{$\ldots$}
    \AXC{$\widehat{\Theta_n} ; \widehat{\sigma_n} \rsep \Gamma_n \sdash \Delta_n$}
    \RSC{\textsc{Wk}}
    \UIC{$\widehat{\Theta_n'} ; \widehat{\sigma_n'} \rsep \Gamma_n \sdash \Delta_n$}
    \DOC{}
    \RSC{$<$L}
    \UIC{$\widehat{\Theta_p} ; \widehat{\sigma_p} \rsep \Gamma_n \sdash \Delta_n$}
    \LSC{R}
    \TIC{$\widehat{\Theta_p} ; \widehat{\sigma_p} \rsep \Gamma \sdash \Delta$}
    \DOC{}
    \LSC{$\textsc{Focus}$}
    \UIC{$\widehat{\Theta_r^k} ; \widehat{\sigma_r^k} \rsep \Gamma \sdash \Delta$}
    \DOC{}
    \LSC{$\textsc{Reset}_{a}$}
    \UIC{$\widehat{\Theta_0} ; \widehat{\sigma_0} \rsep \Gamma \sdash \Delta$}
  \end{comfproof}
  That is, first apply all possible $\textsc{Reset}_a$-rules, starting at the
  $\Theta_0$-greatest $a$. Then \textsc{Focus} on each $t \in \term(\Gamma \sdash
  \Delta)$ with no $t \mapsto u \in \widehat{\sigma_r^k}$. After applying the
  $\RA$-rule corresponding to $R \in \TF$, apply various instances of $<$L
  carefully, as described in the next paragraph. Close each branch of the
  preproof with an application of \textsc{Wk} which removes all superfluous
  annotations from the $\widehat{\sigma_i}$.

  The application of the $<$L-instances requires a little more consideration:
  If $s < t, s < t' \in \Gamma_i$ with $t \neq t'$ then $\sigma_i^*(s, 0)$ will
  contain the annotations $\sigma_p(t, 0), \sigma_p(t', 0)$ extended by the same
  $\gamma \in \Theta_i^*$. In $\RA$, on the other hand, the annotations of $t$
  and $t'$ can only be extended with separate applications of the $<$L-rule,
  meaning the annotations will be extended with two different chips $\gamma,
  \gamma'$. Observe, however, that after the thinning step $(\Theta^*_i,
  \sigma^*_i) \step{T} (\Theta_i, \sigma_i)$ only one annotation remains
  in $\sigma_1(s, 0)$. Thus, the preproof pictured above only applies the
  $<$L-instance yielding this `surviving' annotation of $s$ with the
  `correct' chip $\gamma$. As visible in the preproof above, this results in
  a sequent $\widehat{\Theta_1'} ; \widehat{\sigma_1'} \rsep \Gamma_1 \sdash
  \Delta_1$ rather than the `na\"ive' sequent $\widehat{\Theta_1^*} ; \widehat{\sigma_1^*} \rsep \Gamma_1 \sdash
  \Delta_1$. However, the application of \textsc{Wk} then yields the desired
  sequent at the leaf.

  An argument analogous to that given for $\elab$ in \Cref{lem:elab} shows that
  $\search$ maintains the soundness condition.
\end{proof}

\begin{corollary}[Completeness]\label{lem:ra-complete}
  If $\PA \vdash \Gamma \sdash \Delta$ then $\RA \vdash \Gamma \sdash \Delta$.
\end{corollary}
\begin{proof}
  Suppose there is a proof of $\Gamma \sdash \Delta$ in $\PA$. By \cite[Theorem
  6]{simpsonCyclicArithmeticEquivalent2017} there is a proof of the same sequent
  in $\CA$. Indeed,
  as cyclic proofs are finite trees, this proof is made in some finite fragment
  $\TF$ of $\CA$. By \Cref{lem:complete}, there is a proof
  $\Pi$ of $\Gamma \sdash \Delta ; (\emptyset, (x, a) \mapsto \emptyset)$ in
  $\Sss(\TF)$ and thus $\search(\Pi)$ is a proof of $\varepsilon ; \emptyset \rsep
  \Gamma \sdash \Delta$ in $\RA$.
\end{proof}

\subsection{Gödel's T}
\label{sec:godel-t}

Gödel's T~\cite{godelUberBisherNoch1958} is an extension of the simply
typed $\lambda$-calculus with a type $N$ of natural numbers and arbitrarily
typed primitive recursion.
Cyclic Gödel's T is cyclic variant of Gödel's T put first forward by
Das~\cite{dasCircularVersionGodel2021}.
This means Gödel's T as a derivation system does not
derive `proofs' but rather intrinsically typed terms (see \cite[Chapter
15]{reynoldsTheoriesProgrammingLanguages1998}).
Correspondingly, the preproofs of cyclic Gödel's T ($\CGT$) are intrinsically
typed coterms and the proofs are such coterms satisfying a certain
well-definedness condition. Nonetheless, we continue to refer to them as
preproofs and proofs of $\CGT$, respectively, to keep in line with the
terminology of the rest of the article.

The example of (reset) Gödel's T serves to illustrate the `happy path' of the
method proposed in this article: Only a minor adjustment is made to $\Rrr(\CGT)$ to
obtain the reset system $\RGT$. Thus, the soundness and completeness results
relating $\RGT$ with $\CGT$ can be proven with little effort.

\subsubsection{Cyclic Gödel's T}
\label{sec:cyclic-godel}

The \emph{types} of Gödel's T are given by the grammar
\[A \in \type \langeq N ~|~ A \to A\] The \emph{sequents} of $\CGT$ are
expressions $\Gamma \sdash A$ where $\Gamma$ is a finite sequences of types. The
set of sequents in Gödel's T is denoted \( \Seq_\CGT \).

\begin{definition}
  The \emph{derivation rules} of Cyclic Gödel's T $\CGT$ are:
  \begin{mathpar}
    \inference[\RId]{}{A \sdash A}

    \inference[$0$]{}{\sdash N}

    \inference[S]{}{N \sdash N}

    \inference[L]{\Gamma \sdash \rho \quad \Gamma, A \sdash B}{\Gamma, \rho \to A \sdash B}

    \inference[R]{\Gamma, A \sdash B}{\Gamma \sdash A \to B}

    \inference[\RCond]{\Gamma \sdash A \quad \Gamma, N \sdash A}{\Gamma, N \sdash A}

    \inference[\REx]{\Gamma_0, A, B, \Gamma_1 \sdash C}{\Gamma_0, B, A, \Gamma_1 \sdash C}

    \inference[\RWk]{\Gamma \sdash B}{\Gamma, A \sdash B}

    \inference[\RCtr]{\Gamma, A, A \sdash B}{\Gamma, A \sdash B}

    \inference[\RCut]{\Gamma \sdash B \quad \Gamma, B \sdash A}{\Gamma \sdash A}
  \end{mathpar}
\end{definition}

The presentation of $\CGT$ as a sequent calculus breaks with the tradition of
presenting typing derivations (or, equivalently, intrinsically typed terms) in a
natural deduction style. Observe that the assumption-free non-cyclic preproofs
of $\CGT$ which may also employ the following rule for primitive recursion
\[
  \inference[\RRec]{\Gamma \sdash A \quad \Gamma, A, \sdash A}{\Gamma, N \sdash A}
\]
correspond are precisely the intrinsically typed terms of Gödel's T (presented
in the sequent style).
One can prove that the proofs of $\CGT$ define precisely the
same functionals as ordinary Gödel's T (see \cite[Theorem
94]{dasCircularVersionGodel2021}).

It remains to give the soundness condition of $\CGT$.
Given a sequence $\Gamma$ of types, define $\abs{\Gamma}_N$ to be the number
of occurrences of the ground type \( N \) in \( \Gamma \):
\[\abs{\Gamma}_N \coloneq
  \begin{cases}
    0 & \Gamma = \varepsilon \\
    1 + \abs{\Gamma'} & \Gamma = N, \Gamma' \\
    \abs{\Gamma'} & \Gamma = A, \Gamma' \text{ with } \tau \neq N
  \end{cases}
\]

\begin{definition}
  The trace interpretation $\iota \colon \CGT \to \TT_\BB$ is given by $\iota(\Gamma
  \sdash \tau) \coloneq \{1, \ldots, \abs{\Gamma}_N\}$ and for any $R \in \RR$
  with $\rho(R) = (\Gamma \sdash \tau, \Gamma_1 \sdash \tau_1, \ldots,
  \Gamma_n \sdash \tau_n)$ the trace map $r_i \colon \iota(\Gamma \sdash \tau) \to
  \iota(\Gamma_i \sdash \tau_i)$ is defined as follows:
  \begin{itemize}
  \item
    if $R$ is an instance of $\REx$ exchanging two instances of $N$, meaning $\Gamma = \Gamma_0, N,
    N, \Gamma_1$, then
    \[
      r_i \coloneq \{(j, 0, j) ~|~ j \leq k \text{ or } k + 2 < j\} \cup \{(k + 1
      , 0, k + 2), (k + 2, 0, k + 1)\}
    \]
    where $k \coloneq \abs{\Gamma_0}_N$.
  \item if $\Gamma = \Gamma', N$ and $R$ is an instance of $\RWk$ or $R$ is an
    instance of $\RCond$ with $i = 1$ then
    \[
      r_i \coloneq \{(j, 0, j) ~|~ j < \abs{\Gamma}_N\}.
    \]
  \item if $R$ is an instance of $\RCond$ with $i = 2$ then
    \[
      r_i \coloneq \{(j, 0, j) ~|~ j < k\} \cup \{(k, 1, k)\}
    \]
    with $k \coloneq \abs{\Gamma}_N$.
  \item if $R$ is an instance of $\RCtr$ then
    \[
      r_i \coloneq \{(j, 0, j) ~|~ j < k\} \cup \{(k - 1, 0, k)\}
    \]
    with $k \coloneq \abs{\Gamma}_N$.
  \item otherwise $ r_i \coloneq \{(j, 0, j) ~|~ j \leq \abs{\Gamma}_N\} $.
  \end{itemize}
  This induces the soundness condition $C$ of Cyclic Gödel's T as described in
  \Cref{def:tc-induced}.
  If $\Gamma \sdash \tau$ is provable in $\CGT$, write $\CGT \vdash \Gamma
  \sdash \tau$.
\end{definition}

\subsubsection{Reset Gödel's T}
\label{sec:reset-godel}

We introduce a reset-based proof system corresponding to $\CGT$ called
\emph{reset Gödel's T} ($\RGT$). It is based on the reset system $\Rrr(\CGT)$
induce by $\CGT$ with some slight modifications.

The \emph{sequents} of $\RGT$ are expressions $\Theta \rsep \Gamma \sdash A$,
where the \emph{control} $\Theta$ is a sequence of distinct characters, $A$ is
an unannotated type and the \emph{context} $\Gamma$ is a list of types $B^u$
annotated with a subsequence $u \sqsubseteq \Theta$ which may only be non-empty
if $B = N$. The set of $\RGT$-sequents is denoted $\Seq_\RGT$.

\begin{definition}
  The derivation rules of reset Gödel's T follow. Denote by $\Theta'$ the
  control from which all letters not occurring in any annotation in the context
  are removed.
  \begin{mathpar}
    \inference[\RId]{}{\Theta \rsep A^u \sdash A}

    \inference[$0$]{}{u \rsep \sdash N}

    \inference[S]{}{\Theta \rsep N^u \sdash N}

    \inference[L]{\Theta' \rsep \Gamma \sdash A \quad \Theta \rsep \Gamma, B^u \sdash C}{\Theta \rsep \Gamma, (A \to B)^u \sdash C}

    \inference[R]{\Theta \rsep \Gamma, A^\varepsilon \sdash B}{\Theta \rsep \Gamma \sdash A \to B}

    \inference[\RCond]{a \text{ fresh in } \Theta \\ \Theta' \rsep \Gamma \sdash A \quad \Theta a \rsep \Gamma, N^{ua} \sdash A}{\Theta \rsep \Gamma, N^u \sdash A}

    \inference[\REx]{\Theta \rsep \Gamma, A^v, B^u, \Theta \rsep \Gamma' \sdash C}{\Theta \rsep \Gamma, B^u, A^v, \Theta \rsep \Gamma' \sdash C}

    \inference[\RWk]{\Theta' \rsep \Gamma \sdash B}{\Theta \rsep \Gamma, A^u \sdash B}

    \inference[\RCtr]{\Theta \rsep \Gamma, A^u, A^u \sdash B}{\Theta \rsep \Gamma, A^u \sdash B}

    \inference[\RCut]{\Theta \rsep \Gamma \sdash B \quad \Theta \rsep \Gamma, B^\varepsilon \sdash A}{\Theta \rsep \Gamma \sdash A}

    \inference[$\RReset_a$]{\Theta' \rsep \Gamma, N^{ua}, \ldots, N^{ua} \sdash
      B \quad a \text{ does not occur in } \Gamma}{\Theta \rsep \Gamma, N^{uau_1}, \ldots, N^{uau_n} \sdash B}
  \end{mathpar}

  An \( \RGT \)-preproof is a \emph{proof}
  every pair of bud $t
  \in \dom(\beta)$ and companion $\beta(t)$ has an \emph{invariant}\emph{:}
  there
  exists a letter $a$ such that $a$ occurs in all of the controls
  $\Theta$ between $t$ and $\beta(t)$, the prefix of $a$ in the controls
  $\Theta$ remains constant and the $\textsc{Reset}_a$ rule
  is applied between $t$ and $\beta(t)$.
  A proof $\Pi$ is a \emph{proof of}
  $\Gamma \sdash A$ if its root
  is labeled $\varepsilon \rsep \Gamma^\varepsilon \sdash A$.
  Write $\RGT \vdash \Gamma \sdash A$ if there is a proof of
  $\varepsilon ~|~ \Gamma^\varepsilon \sdash A$ in $\RGT$.
\end{definition}

The system $\RGT$ diverges form $\Rrr(\CGT)$ in one aspect: Every type in the
context is always annotated. In $\Rrr(\CGT)$, annotations need to manually be added
via the $\RPop$ rule. Furthermore, to make the system slightly easier to define,
$\RGT$ uses annotations on every type in $\Gamma$, rather than just instances of
$N$. In both systems, however, only instances of the type $N$ can
ever have non-empty annotations.

We prove soundness of $\RGT$ relative to $\CGT$ by constructing a proof morphism
$\embed \colon \RGT \to \Rrr(\CGT)$.

\begin{lemma}
  There is a function $\embed \colon \Seq_\RGT \to \Rrr(\Seq_\CGT)$ defined by
  \[\embed(\Theta \rsep \Gamma \sdash A) \coloneq \ol{\Gamma} \sdash A ; (\ol{\Theta},
    \sigma(i) \mapsto \{\{x \text{ occurs in } u\} ~|~ \Gamma @ i = N^u\}) \]
  where $\ol{\Theta}$ is the set $\{u \text{ occurs in } \Theta\}$ ordered
  according to the letters' positions in $\Theta$, $\ol{\Gamma}$ is $\Gamma$
  with all annotations removed from the types and $\Gamma @ i$ is the
  partial operation recursively defined by
  \begin{mathpar}
    \Gamma, N^u @ 0 \coloneq N^u

    \Gamma, N^u @ i + 1 \coloneq \Gamma @ i

    \Gamma, A^u @ i \coloneq \Gamma @ i \text{ (where $A \neq N$)}
  \end{mathpar}
  The function can be extended into a proof morphism $\embed \colon \RGT \to
  \Rrr(\CGT)$.
\end{lemma}
\begin{proof}
  Most rules of reset Gödel's T are simply translated to their correspondent
  in the induced reset system for Cyclic Gödel's T. The only exception are the
  rules R and \RCut{} with $B = N$, which generates `unannotated
  instance' in $\Rrr(\CGT)$. These cases can be dealt with by an
  additional application of \RPop{}, as illustrated for the case of $R$ below:
  \begin{lrbox}{\mypt}
    \begin{varwidth}{\linewidth}
      \begin{comfproof}
        \AXC{$\Gamma, N \sdash A; (\ol{\Theta}, \ol{\sigma} \cup \{(k + 1,
          \{\emptyset\})\})$}
        \LSC{\RPop}
        \UIC{$\Gamma, N \sdash A; (\ol{\Theta}, \ol{\sigma} \cup \{(k + 1,
          \emptyset)\})$}
        \LSC{R}
        \UIC{$\Gamma \sdash N \to A; (\ol{\Theta}, \ol{\sigma})$}
      \end{comfproof}
    \end{varwidth}
  \end{lrbox}
  \[
    \inference[R]{\Theta \rsep \Gamma, N^\varepsilon \sdash A}{\Theta \rsep \Gamma \sdash N \to A}
    \quad
    \stackrel{\embed}{\leadsto}
    \quad
    \usebox{\mypt}
  \]

  By comparing the soundness conditions of $\RGT$ and $\Rrr(\CGT)$, it is easily
  observed that $\embed$ maintains the soundness condition of $\RGT$.
\end{proof}

Note that in the soundness theorem below is important that the $\CGT$ proof
corresponding to the $\RGT$ proof not only derives the same sequent $\Gamma
\sdash A$ but furthermore also has the same computational content. It is easily
observed that neither $\strip$ nor $\embed$ change the computational content of
the proofs involved.

\begin{corollary}[Soundness]\label{lem:rct-sound}
  If $\RGT \vdash \Gamma \sdash A$ via $\Pi$ then $\CGT \vdash \Gamma \sdash
  A$ via $\strip(\embed(\Pi))$.
\end{corollary}

Let $\TF$ be a finite fragment of $\CGT$. To conclude completeness of $\RGT$
relative to $\CGT$, we
construct a proof morphism $\search \colon \Sss(\TF) \to \CGT$. For this, recall that
$\Sss(\TF)$ is the `proof search system' induced by the derivation rules in $\TF$
and the trace interpretation of $\CGT$.

\begin{lemma}
  The function $\search \colon \Sss(\Seq_\TF) \to \Seq_\RGT$ is defined as
  \[\search(\Gamma \sdash A; (\Theta, \sigma)) \coloneq \widehat{\Theta}
    \rsep \Gamma^\sigma \sdash A\]
  where for any $S \subseteq \Theta$, $\widehat{S} \in \Theta^*$ is the
  duplicate-free sequence of length
  $\abs{S}$, consisting of the elements of $S$ which is strictly sorted
  according to $\Theta$. Recalling that $(\Theta, \sigma)$ is $K$-sparse, the notation
  $\Gamma^\sigma \coloneq \Gamma^\sigma_1$ is recursively defined by
  \[
    \Gamma^\sigma_i \coloneq
    \begin{cases}
      A^\varepsilon, \Gamma'^\sigma_i & \text{ if } \Gamma = A, \Gamma'
      \text{ with } A \neq N \\
      N^{\hat{S}}, \Gamma'^\sigma_{i + 1} & \text{ if } \Gamma = N, \Gamma'
      \text{ with } \sigma(i) = \{S\} \\
      N^\varepsilon, \Gamma'^\sigma_i & \text{ if } \Gamma = N, \Gamma'
      \text{ with } \sigma(i) = \emptyset \\
      \varepsilon & \text{ if } \Gamma = \varepsilon
    \end{cases}
  \]
  The function can be extended into a proof morphism $\search \colon \Sss(\TF) \to \RGT$.
\end{lemma}
\begin{proof}
  Towards this claim, first pick some $\Rrr(\Theta, \sigma) \in \Sss(\TF)$ arranged as follows
  \[
    \inference[$\Rrr(\Theta, \sigma)$]{
      \Gamma_1 \sdash B_1 ; (\Theta_1, \sigma_1) \quad
      \ldots \quad
      \Gamma_n \sdash B_n ; (\Theta_n, \sigma_n)
    }
    {\Gamma \sdash A ; (\Theta, \sigma)}
  \]
  Then there is $R \in \TF$ with $\rho(R) = (\Gamma \sdash A, \Gamma_1
  \sdash B_1, \ldots,
  \Gamma_n \sdash B_n)$ and morphisms $r_i \colon \iota(\Gamma \sdash A)
  \to \iota(\Gamma_i \sdash B_i)$ given by
  the trace interpretation. Then for each $i \leq n$ there is $(\Theta, \sigma)
  \gstep{r_i} (\Theta_i, \sigma_i)$ with the expanded sequence
  the expanded sequence
  \[(\Theta, \sigma) \step{R_{\gamma_1}} (\Theta_r^1, \sigma_r^1) \ldots
    \step{R_{\gamma_k}} (\Theta_r^k, \sigma_r^k)
    \step{P} (\Theta_p, \sigma_p) \step{r_i} (\Theta_i,
    \sigma_i) \step{T} (\Theta_i, \sigma_i) \]
  in which the initial $R_\gamma$- and $P$-steps are shared between all $i \leq
  n$ (see \Cref{lem:greedy-ceil}). Observe that because of the structure of the
  trace interpretation for Cyclic Gödel's T, there will never be two stacks on
  the same field of a Safra board in the expanded sequence, meaning the thinning
  does not change $\Theta_i$ and $\sigma_i$. Similarly to \Cref{lem:elab}, we may derive
  the following in reset Gödel's T:
  \begin{comfproof}
    \AXC{$\widehat{\Theta_1} \rsep \Gamma^{\sigma_1}_1 \sdash B_1$}
    \AXC{$\ldots$}
    \AXC{$\widehat{\Theta_n} \rsep \Gamma^{\sigma_n}_n \sdash B_n$}
    \LSC{R}
    \TIC{$\widehat{\Theta_r^k}  \rsep \Gamma^{\sigma_r^k} \sdash A$}
    \DOC{}
    \LSC{$\textsc{Reset}_{a}$}
    \UIC{$\widehat{\Theta_0}  \rsep \Gamma^{\sigma_0} \sdash A$}
  \end{comfproof}
  That is, first apply all possible $\textsc{Reset}_a$-rules, starting at the
  $\Theta_0$-greatest $a$. Because $\Gamma^{\sigma}$ annotates instances of $N$
  to which $\sigma$ `assigns' no stack with $N^\varepsilon$, the population
  step does not need to be replicated in the preproof as $\Gamma^{\sigma_r^k}
  = \Gamma^{\sigma_p}$. Complete the preproof by applying the rule corresponding
  to $R \in \TF$. Observe that the rules of reset Gödel's T again annotate
  instances of $N$ to which $\sigma_i$ does not assign a stack with
  $N^\varepsilon$, meaning the premises are indeed $\widehat{\Theta_i} \vdash
  \Gamma_i^{\sigma_i} \sdash B_i$. As the thinning transition does not change
  $\sigma_i$, it does not need to be replicated in the preproof. 

  An argument analogous to that given for $\elab$ in \Cref{lem:elab} shows that
$\search$ maintains the soundness condition.
\end{proof}

Similarly to soundness, for completeness it is again important that the
computational content of the proof remain unchanged.

\begin{theorem}[Completeness]
  If $\CGT \vdash \Gamma \sdash A$ via $\Pi$ then $\RGT \vdash
  \Gamma \sdash A$ via a proof $\Pi'$ such that
  $\strip(\embed(\Pi'))$ is an unfolding of $\Pi$.
\end{theorem}


\subsection{Modal $\mu$-Calculus}
\label{sec:modal-mu}

The modal $\mu$-calculus ($\mML$) extends the classical modal logic $K$ with a
fixed-point quantifiers $\nu x. \varphi$ and $\mu x. \varphi$, denoting the
greatest and least fixed-point, respectively. The $\mu$-calculus has been
central to the field of cyclic proof theory: 
The first cyclic proof system was given for
$\mML$~\cite{niwinskiGamesMcalculus1996} and $\mML$ (and its variants such as
$\mu$MALL and higher-order $\mu$-calculi) have been studied the most in the
field of cyclic proof theory.

In this section we construct two reset proof systems for the modal
$\mu$-calculus, called $\FRmML$ (\Cref{sec:failure-mu}) and $\BRmML$
(\Cref{sec:boolean-mu}), which correspond to two different
natural formulations of the trace condition of $\mML$: one in terms of the
booleans $\BB$ and one in terms of the failure algebra $\FF$. This also
demonstrates that the same derivation system can induce multiple quite distinct
reset proof systems if there are multiple sensible trace interpretations for it.

The first reset proof system was given for satisfiability of $\mML$ by Jungteerapanich
in~\cite{jungteerapanichTableauSystemsModal2010} and later reformulated by
Stirling~\cite{stirlingProofSystemNames2013} into a reset proof system for
validity. The latter system is often called the Jungteerapanich-Stirling ($\JS$)
system in the literature. In \Cref{sec:js} we recall the system $\JS$ and
compare it to the systems $\FRmML$ and $\BRmML$. This comparison highlights
the likely biggest shortcoming of our approach: The reset proof systems in our
article are constructed solely based on the trace condition without deeper
insight into the semantics of the logic in question.

\subsubsection{Cyclic Modal $\mu$-Calculus}
\label{sec:cyclic-mu}

Our presentation of $\mML$ is based on the presentation given
in~\cite{afshariCutfreeCompletenessModal2017}.
For a set $\prop$ of propositional letters a countable set $\var$ of variables,
the $\mu$-formulas are given by the following grammar:
\[\varphi \in \form \langeq p ~|~ \neg p ~|~ x ~|~ \varphi \wedge \varphi
  ~|~ \varphi \vee \varphi ~|~ \Box \varphi ~|~ \Diamond \varphi ~|~ \mu
  x.\varphi ~|~ \nu x. \varphi
  \quad p \in \prop, x \in \var \]
If $x, y \in \var$ occur in $\varphi$, \emph{$x$ subsumes $y$},
writing $x <_\varphi y$, if $\sigma y.\psi$ occurs as a
subformula of $\varphi$ for some $\sigma \in \{\mu, \nu\}$ and $\psi$, and
furthermore $x$ is free in $\sigma y.\psi$.
If the relation $<_\varphi$ is a strict preorder, one calls $\varphi$
\emph{well-named}. Henceforth, we only consider sequents $\Gamma$ where $\Gamma$ is a set of
well-named formulas.

\begin{definition}
  The \emph{sequents} of $\mML$ are finite sets $\Gamma$ of well-named $\mu$-formulas.
  Write $\Gamma, \varphi$ to mean $\Gamma \cup \{\varphi\}$ and $\Gamma,
  \Gamma'$ to mean $\Gamma \cup \Gamma'$.
  The set of $\mML$-sequents is denoted $\Seq_{\mML}$.
  The \emph{derivation rules} of $\mML$ are the following:
  \begin{mathpar}
    \inference[\textsc{Ax}]{}{p, \neg p}

    \inference[\textsc{Wk}]{\Gamma}{\Gamma, \varphi}

    \inference[$\vee$]{\Gamma, \varphi, \psi}{\Gamma, \varphi \vee \psi}

    \inference[$\wedge$]{\Gamma, \varphi \qquad \Gamma, \psi}{\Gamma, \varphi \wedge \psi}

    \inference[\textsc{Mod}]{\Gamma, \varphi}{\Diamond \Gamma, \Box \varphi}

    \inference[$\mu$]{\Gamma, \varphi[\mu x. \varphi / x]}{\Gamma, \mu x. \varphi}

    \inference[$\nu$]{\Gamma, \varphi[\nu x. \varphi / x]}{\Gamma, \nu x. \varphi}
  \end{mathpar}
\end{definition}
In the rules above, $\varphi[\psi / x]$ denotes the formula resulting from
replacing all instances of the variable $x$ in $\varphi$ by the formula $\psi$.
This is a partial operation which is only defined if $x$ does is not bound in
$\varphi$ by some fixed-point quantifier $\nu x$ or $\mu x$. Writing
$\varphi[\psi / x]$ is to be understood as a tacit claim that this the operation
is defined on these arguments.

As noted previously, the modal $\mu$-calculus $\mML$ can be given at least two
natural trace interpretations, one in terms of the booleans $\BB$ and one in
terms of the failure algebra $\FF$. We distinguish the two trace interpretations
by denoting them $\iota_\Aa$ where $\Aa$ is the activation algebra over which
the interpretation in question is defined. The interpretation $\iota_\FF$
corresponds to the `usual' formulation of the global trace condition of the
$\mu$-calculus as a parity condition in terms of the subsumption hierarchy
$<_\varphi$. In fact, this global trace condition was the original motivation
for the concept of activation algebras~\cite{afshariAbstractCyclicProofs2022}.

For a derivation rule $R \in \mML$ with $\rho(R) = (\Gamma, \Delta_1, \ldots,
\Delta_n)$ call $\varphi' \in \Delta_i$ a
\emph{precursor of $\varphi \in \Gamma$}, writing $\varphi' \leftarrow_R^i
\varphi$, if either
$\varphi$ is principal in $R_i$, i.e. $\varphi$ is `altered by
$R_i$', and $\varphi'$ is one of the formula
occurrences resulting from $\varphi$ via $R$
or if $\varphi$ is not principal in $R_i$ and $\varphi = \varphi'$.

\begin{definition}\label{def:ff-mu-tc}
  Writing $V_\nu(\varphi) \coloneq \{x ~|~ x \text{ is bound by } \nu \text{ in
  } \varphi\}$, the trace interpretation $\iota_\FF \colon \mML \to \TT_\FF$ is given by
  \[
    \iota_\FF(\Gamma) \coloneq \{(\varphi, x) ~|~ \varphi \in \Gamma \text{ and
    } x \in V_\nu(\varphi)\}
  \]
  and for any $R \in \mML$ with $\rho(R) = (\Gamma, \Delta_1, \ldots, \Delta_n)$
  the trace maps $r_i \colon \iota_\FF(\Gamma) \to \iota_\FF(\Delta)$ are defined by
  \(
  r_i \coloneq \{((\varphi, x), a^*, (\varphi', x)) ~|~ \varphi' \leftarrow_R^i \varphi \}
  \) where $a^*$ is defined by
  \[
    a^* \coloneq
    \begin{cases}
      2, & \text{ if }R_i \text{ instance of } \mu, \varphi = \mu y.\theta, \varphi' = \theta[\mu y.\theta / y] \text{ and } y <_{\varphi} x,  \\
      1, & \text{ if } R_i \text{ instance of } \nu, \varphi = \nu x.\theta, \varphi' = \theta[\nu x. \theta / x], \\
      0, & \text{ otherwise.}
    \end{cases}
  \]
\end{definition}

The trace interpretation $\iota_\BB$ for $\mML$ is defined by `tracking'
individual fixed-point quantifier instances in the sequent and finding a
greatest fixed-point quantifier which is unfolded infinitely often and never
`erased' by the unfolding of a higher quantifier.
Call a sequence $a \in \BB^*$ a \emph{subformula address} and define a partial addressing
function $\varphi \addr a$ as follows:
\[
  \varphi \addr \varepsilon \coloneq \varphi
  \qquad \quad
  \varphi_0 \bullet \varphi_1 \addr ia \coloneq \varphi_i \addr a
  \qquad \quad
  \ocircle \varphi \addr 0a \coloneq \varphi \addr a
\]
where $\bullet \in \{\wedge, \vee\}$ and $\ocircle \in \{\neg, \Box,
\Diamond\} \cup \{\mu x, \nu x ~|~ x \in \var\}$. Define the set of
$\nu$-addresses of a formula as $N(\varphi) \coloneq \{a \in \BB^* ~|~ \varphi
\addr a = \nu x. \psi\}$. Given $x \in \var$, define the set of open $x$
addresses in $\varphi$ by $O_x(\varphi) \coloneq \{a \in \BB^* ~|~ \varphi
\addr a = x \text{ and } \forall a' < a, \psi.~\varphi \addr a' \neq \nu x.
\psi \wedge \varphi \addr a' \neq \mu x. \psi \}$.

\begin{definition}\label{def:bb-mu-tc}
  The trace interpretation $\iota_\BB \colon \mML \to \TT_\BB$ is defined by
  \(\iota_\BB(\Gamma) \coloneq \{(\varphi, a) ~|~ \varphi \in \Gamma \wedge a
    \in N(\varphi)\}\). For each rule $R \in \mML$ with $\rho(R) = (\Gamma,
    \Delta_1, \ldots, \Delta_n)$ which is not a fixed-point rule, the trace maps
    $r_i \colon \iota_\BB(\Gamma) \to \iota_\BB(\Delta_i)$ simply `track' the
    fixed-point instances. For instance, if $R$ is the following instance of the
    $\vee$-rule
    \[
      \inference[$\vee$]{\Gamma, \varphi_0, \varphi_1}{\Gamma, \varphi_0 \vee \varphi_1}
    \]
    then $r_1 \coloneq \{((\psi, a), 0, (\psi, a)) ~|~ (\psi, a) \in
    \iota_\BB(\Gamma)\} \cup \{((\varphi_0 \vee \varphi_1, ia), 0, (\varphi_i,
    a)) ~|~ ia \in N(\varphi_0 \vee \varphi_1)\}$. Now suppose $R$ was an
    instance of a fixed-point rule
    \[
      \inference[$\sigma$]{\Gamma, \varphi[\sigma x. \varphi / x]}{\Gamma, \sigma x. \varphi}
    \]
    then
    \begin{align*}
      r_1 \coloneq\,& \{(v, 0, v) ~|~ v \in \iota_\BB(\Gamma)\} ~\cup \\
                   & \{((\sigma x. \varphi, \varepsilon), b, (\varphi[\sigma x. \varphi / x], a) ~|~ a \in O_x(\varphi)\} ~\cup \\ 
                   & \{((\sigma x. \varphi, 0a), 0, (\varphi[\sigma x. \varphi / x], a) ~|~ a \in N(\varphi)\} ~\cup \\
                   & \{((\sigma x. \varphi, 0a), 0, (\varphi[\sigma x. \varphi / x], a'0a) ~|~ a \in N(\varphi), a' \in O_X(\varphi)\}
    \end{align*}
    where $b = 0$ iff $\sigma = \mu$. Spelling out all the details of this
    definition would quickly become unwieldy, so we rely on the reader's
    intuition. A detailed account of a very similar trace condition can be found
    in~\cite{koriCyclicProofSystem2021}.
\end{definition}

Henceforth, we write $\mML$ to denote the cyclic proof systems $\iota_\FF(\mML)$
and $\iota_\BB(\mML)$. There is no need to distinguish between the two as both
trace interpretations induce the same cyclic proof system: Precisely the same
preproofs satisfy the global trace condition specified via $\iota_\BB$ as that
specified via $\iota_\FF$. Nonetheless, they induces two different reset proof
systems.

\subsubsection{$\FF$-Reset Modal $\mu$-Calculus}
\label{sec:failure-mu}

We first present a reset proof system $\FRmML$ corresponding to
$\iota_\FF(\mML)$ (pronounced ``$\FF$-reset modal
$\mu$-calculus'').

The \emph{sequents} of $\FRmML$ are expressions of the form $\Theta
\rsep \Gamma$ where the control $\Theta$ is a sequence of distinct characters,
called the \emph{control},
and $\Gamma$ is a set of pairs $(\varphi, \sigma)$ of $\mu$-formulas $\varphi$
and an \emph{annotation}, a set of \emph{assignments} $x_1 \mapsto u_1, \ldots, x_n \mapsto u_n$ where $\{x_1,
\ldots, x_n\} = V_\nu(\varphi)$, $u_1, \ldots, u_n$ are subsequences of
$u_i \sqsubseteq \Theta$ and each variable $x \in V_\nu(\varphi)$ is part of
precisely one assignment. For simplicity of notation, we often treat such
sequences $\sigma$ simply as a functions mapping $V_\nu(\varphi)$ to
subsequences of $\Theta$. The set of sequents $\FRmML$ is denoted
$\Seq_{\FRmML}$.
Analogously to Safra boards, a letter $a \in \Theta$ is \emph{covered} in $\Gamma$ if for every
$(\varphi, \sigma) \in \Gamma$ and every $x \in V_\nu(\varphi)$, if $a \in
\sigma(x)$ then it is not at the last position of that sequence.

\begin{definition}
  The derivation rules of the $\FF$-reset modal $\mu$-calculus ($\FRmML$)
  follow. Denote by $\Theta'$ the control from which all letters not
  occurring in any annotation in the corresponding $\Gamma$ are removed. Similarly, $\sigma'$
  denotes the annotation from which all assignments $x \mapsto u$ with $x
  \not\in V_\nu(\varphi)$ in the corresponding $\mu$-formula $\varphi$ have been
  removed.
  \begin{mathpar}
    \inference[\textsc{Ax}]{}{\Theta \rsep (p, \varepsilon), (\neg p, \varepsilon)}

    \inference[\textsc{Wk}]{\Theta' \rsep \Gamma}{\Theta \rsep \Gamma, (\varphi, \varepsilon)}

    \inference[$\vee$]{\Theta \rsep \Gamma, (\varphi, \sigma'), (\psi, \sigma')}{\Theta \rsep \Gamma, (\varphi
      \vee \psi, \sigma)}

    \inference[$\wedge$]{\Theta' \rsep \Gamma, (\varphi, \sigma') \qquad
      \Theta' \rsep \Gamma, (\psi, \sigma')}{\Theta \rsep \Gamma, (\varphi
      \wedge \psi, \sigma)}

    \inference[\textsc{Mod}]{\Theta \rsep \Gamma, (\varphi, \sigma)}{\Theta \rsep \Diamond
      \Gamma, (\Box \varphi, \sigma)}

    \inference[$\mu$]{\Theta' \rsep \Gamma, (\varphi[\mu x. \varphi / x], \sigma
      \setminus x)}{\Theta \rsep \Gamma, (\mu x.
      \varphi, \sigma)}

    \inference[$\nu$]{\Theta a \rsep \Gamma, (\varphi[\nu x. \varphi / x],
      (\sigma, x \mapsto ua)) \quad a \not\in \Theta }{\Gamma, (\nu x.
      \varphi, (\sigma, x \mapsto u))}

    \inference[$\RReset_a$]
    {\Theta' \rsep \sreset{\Gamma}{a} \quad a \text{ covered in } \Gamma }
    {\Theta \rsep \Gamma}

    \inference[\(\RMerge\)]{\Theta' \rsep \Gamma, (\varphi,
      \merge_\Theta(\sigma, \sigma'))}{\Theta \rsep \Gamma, (\varphi, \sigma), (\varphi, \sigma')}
  \end{mathpar}
  Where $\Diamond \Gamma \coloneq \{(\Diamond \varphi, \sigma) ~|~ (\varphi,
  \sigma) \in \Gamma\}$ and $\sreset{\Gamma}{x} \coloneq \{(\varphi,
  \sreset{\sigma}{a}) ~|~ (\varphi, \sigma) \in \Gamma\}$ and the various
  annotations used above are defined below
  \begin{mathpar}
    (\sigma \setminus x)(y) \coloneq
    \begin{cases}
      \varepsilon & \text{ if } x <_{\nu x. \varphi} y \\
      \sigma(y) & \text{ otherwise}
    \end{cases}

    (\sreset{\sigma}{a})(x) \coloneq
    \begin{cases}
      u & \text{ if } \sigma(x) = uav \\
      \sigma(x) & \text{ otherwise}
    \end{cases}

    \merge_{\Theta}(\sigma, \sigma')(x) \coloneq {\min}_\Theta(\sigma(x), \sigma'(x))
  \end{mathpar}

  Write $\varepsilon$ for the annotation $x_1 \mapsto \varepsilon, \ldots,
  x_n \mapsto \varepsilon$.
  A $\FRmML$-preproof is a \emph{proof}
  every pair of bud $t
  \in \dom(\beta)$ and companion $\beta(t)$ has an \emph{invariant}, i.e. there
  exists a letter $a$ such that $a$ occurs in all of the controls
  $\Theta$ between $t$ and $\beta(t)$, the prefix of $a$ in the controls
  $\Theta$ remains constant and the $\textsc{Reset}_a$ rule
  is applied between $t$ and $\beta(t)$.
  A proof is a \emph{proof of a $\mu$-sequent} $\Gamma$, writing
  $\FRmML \vdash \Gamma$, if its root
  is labeled $\varepsilon \rsep \{(\varphi, \varepsilon) ~|~ \varphi \in \Gamma\}$.
\end{definition}

The $\RMerge{}$ rule is somewhat inelegant and in some cases subsumed by $\RWk$.
It is required to simulate the thinning step of greedy runs in the completeness
proof. We conjecture that, every instance of $\RMerge$ that could be
needed to construct $\FRmML$-proofs corresponding to $\mML$-proofs is subsumed
by $\RWk$. Indeed, the Jungteerapanich-Stirling
system~\cite{stirlingProofSystemNames2013} for the modal $\mu$-calculus features
an analogous $\RThin$-rule which always chooses one of the two
$\varphi$-instances and which is sufficient to prove completeness. However,
proving that $\RMerge$ is superfluous would likely be as involved as a direct
completeness proof for $\FRmML$ with regards to the semantics of the modal
$\mu$-calculus. We thus forgo this `ergonomic optimization' as the goal of our
article is to derive concrete reset proof systems without having to spend much
effort.

Soundness of $\FRmML$ with regards to $\mML$ is proven by constructing a proof
morphism $\embed \colon \FRmML \to \mML$.

\begin{lemma}\label{lem:frmml-embed}
  There exists a function $\embed \colon \Seq_{\FRmML} \to \Rrr(\Seq_{\mML})$ which is defined by
  \[\embed(\Theta \rsep \Gamma) \coloneq \ol{\Gamma} ; (\ol{\Theta}, \sigma_\Gamma) \]
  where $\ol{\Theta}$ is the set $\{u \in \Theta\}$ ordered
  according to the letters' positions in $\Theta$, $\ol{\Gamma} \coloneq
  \{\varphi ~|~ (\varphi, \sigma) \in \Gamma\}$ and $\sigma_\Gamma$ is defined as
  \[
    \sigma_\Gamma((\varphi, x), 0) \coloneq \{\{a \in \sigma(x)\} ~|~ (\varphi, \sigma) \in \Gamma \}
    \qquad \qquad
    \sigma_\Gamma((\varphi, x), e) \coloneq \emptyset \quad \text{if } e \in \{1, 2\}.
  \]
  The function can be extended into a proof morphism $\embed \colon \FRmML \to \Rrr(\mML)$.
\end{lemma}
\begin{proof}
  The main complication in completing this morphism is that for $\Theta \rsep
  \Gamma$, there may be $(\varphi, \sigma), (\varphi, \sigma') \in \Gamma$ with
  $\sigma \neq \sigma'$, i.e. two instances of the same formula with
  different annotation. The $\embed$-function `collapses' the two annotated
  formulas into one instance of a formula in which each $\nu$-variable is
  annotated by two stacks.

  A simple example in which this causes complications is the $\RWk$-rule.
  Suppose the $\RWk$-instance removed $(\varphi, \sigma)$ from $\Gamma,
  (\varphi, \sigma'), (\varphi, \sigma)$
  with $\sigma \neq \sigma'$. Then $\ol{\Gamma,
  (\varphi, \sigma'), (\varphi, \sigma)} = \ol{\Gamma}, \varphi$. Thus, to
  simulate $\RWk$, the corresponding $\RWk$-rule of
  $\Rrr(\mML)$ need not be applied; only some annotations need to be removed from
  $\varphi$. Thus, the simulation is as given below.
  \begin{lrbox}{\mypt}
    \begin{varwidth}{\linewidth}
      \begin{comfproof}
        \AXC{$\Gamma, \psi ; (\ol{\Theta'}, \sigma_{\Gamma,
            (\varphi, \sigma')}) $}
        \LSC{\RWeak}
        \UIC{$\ol{\Gamma}, \ol{\psi} ; (\ol{\Theta}, \sigma_{\Gamma,
            (\varphi, \sigma'), (\varphi, \sigma)}) $}
      \end{comfproof}
    \end{varwidth}
  \end{lrbox}
  \[
    \inference[\textsc{Wk}]{\Theta' \rsep \Gamma, (\varphi, \sigma')}{\Theta \rsep \Gamma,
      (\varphi, \sigma'), (\varphi, \sigma)}
    \quad
    \stackrel{\embed}{\leadsto}
    \quad
    \usebox{\mypt}
  \]

  Similar issues with `collapsing' can arise in the rules $\vee$, $\wedge$,
  $\nu$ and $\mu$.
  For an example, consider a sequent $\Theta \rsep \Gamma_0$ with $\Gamma_0$ as below and $\sigma \neq
  \sigma'$. For simplicity, suppose there were no more `copies' of $\varphi \vee
  \psi$ in $\Gamma$. To embed an application of the $\vee$-rule $\Theta \rsep \Gamma_2$,
  one must first apply the $\Rrr(\mML)$-correspondent of $\vee$ with $\varphi \vee
  \psi$ both principal and part of the `context' $\ol{\Gamma}, \varphi \vee \psi$ to
  $\ol{\Gamma_0} ; (\ol{\Theta}, \sigma_{\Gamma_0})$.
  However, because of the `collapsing' of identical formulas, this only yields
  $\ol{\Gamma_1} ; (\ol{\Theta}, \sigma_{\Gamma_1})$ because the trace
  interpretation $r \colon \iota_\FF(\Gamma, \varphi \vee \psi, \varphi \vee \psi)
  \to \iota_\FF(\Gamma, \varphi \vee \psi, \varphi, \psi)$ contains
  $(\varphi \vee \psi, 0, \varphi \vee \psi)$ as well as $(\varphi \vee \psi, 0,
  \varphi)$ and $(\varphi \vee \psi, 0, \psi)$. Noting $\ol{\Gamma_1} =
  \ol{\Gamma_2}$, the desired premise $\ol{\Gamma_2} ; (\ol{\Theta},
  \sigma_{\Gamma_2})$ can be reached with an application of the $\RWeak$-rule of $\Rrr(\mML)$.
  \begin{align*}
    \Gamma_0 \coloneq\,& \Gamma, (\varphi \vee \psi, \sigma), (\varphi \vee \psi, \sigma') \\
    \Gamma_1 \coloneq\,& \Gamma, (\varphi \vee \psi, \sigma), (\varphi \vee \psi, \sigma'), (\varphi, \sigma), (\psi, \sigma), (\varphi, \sigma'), (\psi, \sigma') \\
    \Gamma_2 \coloneq\,& \Gamma, (\varphi \vee \psi, \sigma), (\varphi, \sigma'), (\psi, \sigma')
  \end{align*}
  In terms of preproofs, this yields:
  \begin{lrbox}{\mypt}
    \begin{varwidth}{\linewidth}
      \begin{comfproof}
        \AXC{$\ol{\Gamma}, \varphi \vee \psi, \varphi, \psi; (\ol{\Theta}, \sigma_{\Gamma_2})$}
        \LSC{\RWeak}
        \UIC{$\ol{\Gamma}, \varphi \vee \psi, \varphi, \psi; (\ol{\Theta}, \sigma_{\Gamma_1})$}
        \LSC{$\vee$}
        \UIC{$\ol{\Gamma}, \varphi \vee \psi; (\ol{\Theta}, \sigma_{\Gamma_0})$}
      \end{comfproof}
    \end{varwidth}
  \end{lrbox}
  \[
    \inference[$\vee$]{\Theta \rsep \Gamma, (\varphi \vee \psi, \sigma),
      (\varphi, \sigma'), (\psi, \sigma')}{\Theta \rsep \Gamma, (\varphi \vee \psi,
      \sigma), (\varphi \vee \psi, \sigma')}
    \quad
    \stackrel{\embed}{\leadsto}
    \quad
    \usebox{\mypt}
  \]
  If there were more `copies' $(\varphi \vee \psi, \sigma'') \in \Gamma$ with
  $\sigma \neq \sigma'' \neq \sigma'$ then the stacks corresponding to
  $(\varphi, \sigma'')$ and $(\psi, \sigma'')$ also need to be removed by the
  $\RWeak$-application. If there is only one `copy' of $\varphi \vee \psi$ in
  $\Gamma$, then $\varphi \vee \psi$ need not be part of the `context'
  $\ol{\Gamma}$ in the $\vee$-application and no application of $\RWeak$ is
  needed. The `collapsing'-related issues that can arise when embedding instances of the rules
  $\wedge$, $\nu$ and $\mu$ are analogous and can be dealt with in an analogous manner.

  If `collapsing'-related issues are dealt with as described in the previous
  paragraph, the $\nu$-rule of $\FRmML$ can simply be translated as the
  corresponding instance of the $\nu$-rule in $\Rrr(\mML)$.
  While the operation $\sigma\,\uh\,a$ involved in the
  $\nu$-rule of $\FRmML$ is paralleled by the corresponding $\nu$-rule in
  $\Rrr(\mML)$, this is not the case for the operation $\sigma \setminus x$ of the $\mu$-rule.
  Consider the sequent
  $b \rsep \Gamma_0$ with 
  $\Gamma_0 \coloneq (\mu x. \nu y. x, y \mapsto b)$. The corresponding
  $\Rrr(\mML)$-sequent is 
  $\ol{\Gamma_0} ; (\ol{b}, \sigma_{\Gamma_0})$ with
  $\sigma_{\Gamma_0}((\mu x. \nu y. x, y), a) \coloneq \{\{b\}\}$ for $a = 0$
  and $\sigma_{\Gamma_0}((\mu x. \nu y. x, y), a) \coloneq \emptyset$ otherwise.
  Applying the $\mu$-rule to this sequent yields $\nu y. \mu x. \nu y. x;
  (\ol{b}, \sigma')$ with $\sigma'((\nu y. \mu x. \nu y. x, y), a) \coloneq
  \{\{b\}\}$ for $a = 2$ and $\sigma'((\nu y. \mu x. \nu y. x, y), a) \coloneq
  \emptyset$ otherwise. On the other hand, applying the $\mu$-rule to $b \rsep
  \Gamma_0$ yields $\rsep \Gamma_1$ with $\Gamma_1 \coloneq (\nu y. \mu x. \nu y. x, y \mapsto \varepsilon)$.
  Thus, an application of the $\RWeak$-rule removing the stack $\{b\}$ from
  $((\nu y. \mu x. \nu y. x, y), 2)$, yielding $\sigma''$ with $\sigma'((\nu y. \mu x. \nu y. x, y), a) \coloneq
  \emptyset$ on all $a \in \FF$, is needed. Furthermore, an application of the
  $\RPop$-rule is required to add an empty stack to $((\nu y. \mu x. \nu y. x, y), 0)$.
  \begin{lrbox}{\mypt}
    \begin{varwidth}{\linewidth}
      \begin{comfproof}
        \AXC{$\nu y. \mu x. \nu y. x; (\emptyset, \sigma_{\Gamma_1})$}
        \LSC{\RPop}
        \UIC{$\nu y. \mu x. \nu y. x; (\emptyset, \sigma'')$}
        \LSC{\RWeak}
        \UIC{$\nu y. \mu x. \nu y. x; (\ol{b}, \sigma')$}
        \LSC{$\mu$}
        \UIC{$\ol{\Gamma_0} ; (\ol{b}, \sigma_{\Gamma_0})$}
      \end{comfproof}
    \end{varwidth}
  \end{lrbox}
  \[
    \inference[$\mu$]{\rsep (\nu y. \mu x. \nu y. x, y \mapsto \varepsilon)}{b \rsep (\mu x. \nu y. x, y \mapsto b)}
    \quad
    \stackrel{\embed}{\leadsto}
    \quad
    \usebox{\mypt}
  \]
  More complicated cases with more $\nu$-variables can be dealt with in an
  analogous manner.

  The $\RMod$-rule simply corresponds to the $\RMod$-rule of $\Rrr(\mML)$. No
  `collapsing'-related complications can arise in its translation as all
  formulas are principal in applications of the $\RMod$-rule. The
  $\RReset_a$-rule directly corresponds to the $\RReset_a$-rule of $\Rrr(\mML)$.
  It is easily observed that \RMerge{} is always simulated by a suitable
  instance of \RWeak{}.
\end{proof}

\begin{corollary}[Soundness]
  If $\FRmML \vdash \Gamma$ then $\mML \vdash \ol{\Gamma}$.
\end{corollary}

Let $\TF$ be a finite fragment of $\mML$. To prove completeness, one constructs
a proof morphism $\search \colon \Sss(\TF) \to \FRmML$ embedding the proof search system
for $\TF$ into $\FRmML$.

\begin{lemma}
  There exists a function $\search \colon \Sss(\Seq_\TF) \to \Seq_{\FRmML}$ with
  \[\search(\Gamma; (\Theta, \sigma)) \coloneq \widehat{\Theta}
    \rsep \Gamma^\sigma\]
  where for any $S \subseteq \Theta$, $\widehat{S} \in \Theta^*$ is the
  duplicate-free sequence of length
  $\abs{S}$, consisting of the elements of $S$ which is strictly sorted
  according to $\Theta$. The notation
  $\Gamma^\sigma \coloneq \{(\varphi, \sigma \uh \varphi) ~|~ \varphi
  \in \Gamma\}$ with $(\sigma \uh \varphi)(x) \coloneq \widehat{\sigma((\varphi, x), 0)}$.

  The function can be extended to a proof morphism $\search \colon \Sss(\TF) \to \FRmML$.
\end{lemma}
\begin{proof}
  Towards this claim, first pick some $\Rrr(\Theta, \sigma) \in \Sss(\TF)$ arranged as follows
  \[
    \inference[$\Rrr(\Theta, \sigma)$]{
      \Gamma_1; (\Theta_1, \sigma_1) \quad
      \ldots \quad
      \Gamma_n; (\Theta_n, \sigma_n)
    }
    {\Gamma; (\Theta, \sigma)}
  \]
  Then there is $R \in \TF$ with $\rho(R) = (\Gamma, \Gamma_1, \ldots,
  \Gamma_n)$ and morphisms $r_i \colon \iota_\FF(\Gamma)
  \to \iota_\FF(\Gamma_i)$ given by
  the trace interpretation. Then for each $i \leq n$ there is $(\Theta, \sigma)
  \gstep{r_i} (\Theta_i, \sigma_i)$ with the expanded sequence
  the expanded sequence
  \[(\Theta, \sigma) \step{R_{\gamma_1}} (\Theta_r^1, \sigma_r^1) \ldots
    \step{R_{\gamma_k}} (\Theta_r^k, \sigma_r^k)
    \step{P} (\Theta_p, \sigma_p) \step{r_i} (\Theta^*_i,
    \sigma^*_i) \step{T} (\Theta_i, \sigma_i) \]
  in which the initial $R_\gamma$- and $P$-steps are shared between all $i \leq
  n$ (see \Cref{lem:greedy-ceil}). Similarly to \Cref{lem:elab}, we may derive
  the following in $\FRmML$:
  \begin{comfproof}
    \AXC{$\widehat{\Theta_1}\rsep \Gamma_1^{\sigma_1}$}
    \DOC{}
    \LIC{\RMerge}
    \UIC{$\widehat{\Theta_1^*}\rsep \Gamma^0_1$}
    \AXC{$\ldots$}
    \AXC{$\widehat{\Theta_n}\rsep \Gamma_n^{\sigma_n}$}
    \DOC{}
    \RIC{\RMerge}
    \UIC{$\widehat{\Theta_n^*}\rsep \Gamma^0_n$}
    \LIC{R}
    \TIC{$\widehat{\Theta_r^k} \rsep \Gamma^{\sigma_r^k}$}
    \DOC{}
    \LIC{$\textsc{Reset}_{\gamma_1}$}
    \UIC{$\widehat{\Theta_0} \rsep \Gamma^{\sigma_0}$}
  \end{comfproof}
  That is, first apply all possible $\textsc{Reset}_{\gamma_i}$-rules, starting at the
  $\Theta_0$-greatest $\gamma_0$. Because $\Gamma^{\sigma}$ annotates $\nu$-variables
  to which $\sigma$ assigns no stack with the empty stack $\varepsilon$, the population
  step does not need to be replicated in the preproof as $\Gamma^{\sigma_r^k}
  = \Gamma^{\sigma_p}$. Continue by applying the rule corresponding
  to $R \in M$. Observe that while the controls resulting from this application
  matches $\Theta_i^*$, the sequents will be some sequent $\Gamma^0_i$ which
  might contain multiple copies of the same formula with different
  annotations (analogously to the `collapsing' issues in
  \Cref{lem:frmml-embed}). In such cases, there is some $(\varphi, x)$ to which
  $\sigma^*_i$ assigns two or more stacks. The greedy run uses a thinning to reestablish
  the property that each quantifier is assigned at most one stack. In $\FRmML$,
  this can be replicated by applying the $\RMerge$-rule to all formulas
  $(\varphi, \sigma), (\varphi, \sigma') \in \Gamma_i^0$. Note that the
  order of applications and choice of which two `$\varphi$-instances' to
  pick for $\RMerge$-applications does not matter as the resulting sequent will
  always be $\widehat{\Theta_i} \rsep \Gamma_i^{\sigma_i}$.

  An argument analogous to that given for $\elab$ in \Cref{lem:elab} shows that
  $\search$ maintains the soundness condition. Crucially, this relies on the
  fact that an accepting Safra board run through a proof in $\Sss(\TF)$ can
  never `enter some state' $((\varphi, x), 2)$. If it did, from that point onwards, the run can
  never reach any $((\psi, x), a)$ with $a < 2$ anymore, meaning only finitely
  $\RReset$s could take place from that point onwards. Thus, `dropping' the
  stacks on $((\varphi, x), 2)$ as the function $\search \colon \Sss(\TF) \to \FRmML$
  does not hinder the maintainance of the soundness condition.
\end{proof}

\begin{corollary}[Completeness]\label{lem:frmml-complete}
  If $\mML \vdash \Gamma$ then $\FRmML \vdash
  \Gamma$.
\end{corollary}

As can be observed, most issues in establishing both $\embed$ and $\search$ as
proof morphisms surround the `collapsing' of formula identities. A simple
solution to circumvent all issues of this kind is be to take the sequents of
$\mML$ to be lists of formulas in which there can be distinct occurrences of the
same formula (and defining $\FRmML$ with similar list-like sequents). We chose
to not do this as this would mark a departure from the presentation of $\mML$
in~\cite{afshariCutfreeCompletenessModal2017}. A central goal of this article is
to give `recipes' for generating reset system which do not require any
modifications to the original cyclic proof system. Doing so for this example
would thus have violated this goal. Nonetheless, we recommend readers who may be
running into similar problems when generating reset systems for their cyclic
proof systems consider making such a modification in order to lighten their
proving load.

\subsubsection{$\BB$-Reset Modal $\mu$-Calulus}
\label{sec:boolean-mu}

We next present a reset proof system $\BRmML$ corresponding to $\iota_\BB(\mML)$
(pronounced ``$\BB$-reset modal $\mu$-calculus'').

The system uses $\mu$-formulas with annotated quantifiers. That is, each
$\nu$-quantifier $\nu^u x. \varphi$ is annotated with a sequence of distinct
characters $u$:
\[\varphi \in \form_\BB \langeq p ~|~ \neg p ~|~ x ~|~ \varphi \wedge \varphi
  ~|~ \varphi \vee \varphi ~|~ \Box \varphi ~|~ \Diamond \varphi ~|~ \mu
  x.\varphi ~|~ \nu^u x. \varphi
  \quad p \in \prop, x \in \var \]
Given $\varphi \in \form_\BB$, write $\ol{\varphi} \in \form$ for the
$\mu$-formula obtained by removing the annotations in $\varphi$. Given a
sequence $\Theta$ of distinct letters, define a
partial function $\merge_\Theta \colon \form_\BB \times \form_\BB \to \form_\BB$ such that
$\merge_\Theta(\varphi, \psi)$ is defined on $\varphi,
\psi \in \form_\BB$ iff $\ol{\varphi} = \ol{\psi}$ and the annotations in
$\varphi, \psi$ are subsequences of $\Theta$. The definition is given
below:
\begin{mathpar}
  \merge_\Theta(a, a) \coloneq a

  \merge_\Theta(\varphi_0 \bullet \varphi_1, \psi_0 \bullet \psi_1) \coloneq
  \merge_\Theta(\varphi_0, \psi_0) \bullet \merge_\Theta(\varphi_1, \psi_1)

  \merge(\ocircle \varphi, \ocircle \psi) \coloneq \ocircle \merge(\varphi, \psi)

  \merge(\nu^u x. \varphi, \nu^v x. \psi) \coloneq \nu^{\min_\Theta(u, v)} x.
  \merge(\varphi, \psi)
\end{mathpar}
where $a \in \prop \cup \var$, $\bullet \in \{\wedge, \vee\}$ and $\ocircle \in \{\neg, \Box,
\Diamond\} \cup \{\mu x ~|~ x \in \var\}$ and $\min_\Theta(u, v)$ is the
minimal sequence according to the ordering $<_\Theta$ defined in \Cref{def:safra-board}.

A letter $a$ is \emph{covered} in $\varphi$ if in every annotation $u$ in
$\varphi$ such that $a$ appears in $u$, $a$ is not at the last position of $u$.
This notion extends to sets $\Gamma$ of annotated formulas. The \emph{reset
  operation} $\sreset{\varphi}{a}$ is defined below:
\begin{mathpar}
  \sreset{p}{a} \coloneq p

  \sreset{\varphi \bullet \psi}{a} \coloneq (\sreset{\varphi}{a}) \bullet (\sreset{\psi}{a})

  \sreset{\ocircle \varphi}{a} \coloneq \ocircle (\sreset{\varphi}{a})

  \sreset{\nu^u x. \varphi}{a} \coloneq \nu^{u'} x. \sreset{\varphi}{a}
  \qquad \text{ where } u' \coloneq
  \begin{cases}
    u & \text{ if } \sigma(x) = uav \\
    \sigma(x) & \text{ otherwise}
  \end{cases}
\end{mathpar}
where $p \in \prop \cup \var$, $\bullet \in \{\wedge, \vee\}$ and $\ocircle \in \{\neg, \Box,
\Diamond\} \cup \{\mu x ~|~ x \in \var\}$.

The \emph{sequents} of $\BRmML$ are expressions $\Theta \rsep \Gamma$,
where the \emph{control} $\Theta$ is a sequence of distinct characters and
$\Gamma$ is a finite set of formulas from $\form_\BB$ such that each quantifier
$\sigma^u x. \varphi$ occurring in $\Gamma$ is annotated with a subsequence of
$\Theta$. 
Write $\varphi^\epsilon$ for a formula $\varphi$ in which all
quantifiers are annotated with the empty sequence $\varepsilon$, extending
this notation to sets of formulas $\Gamma$.
The set of $\BRmML$-sequents is denoted by $\Seq_{\BRmML}$.

\begin{definition}\label{def:brmu}
  The \emph{derivation rules} of $\BRmML$ are given below. Denote by $\Theta'$ the
  control from which all letters not occurring in any annotation in $\Gamma$
  are removed.
  \begin{mathpar}
    \inference[\textsc{Ax}]{}{\Theta \rsep p, \neg p}

    \inference[\textsc{Wk}]{\Theta' \rsep \Gamma}{\Theta \rsep \Gamma, \varphi}

    \inference[$\vee$]{\Theta \rsep \Gamma, \varphi, \psi}{\Theta \rsep \Gamma, \varphi \vee \psi}

    \inference[$\wedge$]{\Theta' \rsep \Gamma, \varphi \qquad \Theta' \rsep
      \Gamma, \psi}{\Theta \rsep \Gamma, \varphi \wedge \psi}

    \inference[\textsc{Mod}]{\Theta \rsep \Gamma, \varphi}{\Theta \rsep \Diamond \Gamma, \Box \varphi}

    \inference[$\mu$]{\Theta \rsep \Gamma, \varphi[\mu x. \varphi / x]}{\Theta \rsep \Gamma, \mu x. \varphi}

    \inference[$\nu$]{\Theta a \rsep \Gamma, \varphi[\nu^{ua} x. \varphi / x]
      \quad a \not\in \Theta}{\Theta \rsep \Gamma, \nu^u x. \varphi}

    \inference[$\RReset_a$]
    {\Theta' \rsep \sreset{\Gamma}{a} \quad a \text{ covered in } \Gamma }
    {\Theta \rsep \Gamma}

    \inference[\(\RMerge\)]{\Theta' \rsep \Gamma, \xi \quad \merge_\Theta(\varphi,
      \psi) = \xi}{\Theta \rsep \Gamma, \varphi, \psi}
  \end{mathpar}
  A $\BRmML$-preproof is a \emph{proof} if
  every pair of bud $t
  \in \dom(\beta)$ and companion $\beta(t)$ has an \emph{invariant}, i.e. there
  exists a letter $a$ such that $a$ occurs in all of the controls
  $\Theta$ between $t$ and $\beta(t)$, the prefix of $a$ in the controls
  $\Theta$ remains constant and the $\textsc{Reset}_a$ rule
  is applied between $t$ and $\beta(t)$.
  A $\BRmML$-proof is \emph{a proof 
  of $\Gamma$} if its root
  is labeled $\varepsilon \rsep \Gamma^\varepsilon$.
  Write $\BRmML \vdash \Gamma $ if there is a proof of $\Gamma$ in $\BRmML$.
\end{definition}

Soundness of $\BRmML$ with regards to $\mML$ is proven by constructing a proof
morphism $\embed \colon \BRmML \to \mML$.

\begin{lemma}
  There exists a function $\embed \colon \Seq_{\BRmML} \to \Rrr(\Seq_{\mML})$ which is defined by
  \[\embed(\Theta \rsep \Gamma) \coloneq \ol{\Gamma} ; (\ol{\Theta}, \sigma_\Gamma) \]
  where $\ol{\Theta}$ is the set $\{u \text{ occurs in } \Theta\}$ ordered
  according to the letters' positions in $\Theta$, $\ol{\Gamma}$ is $\Gamma$
  with all annotations removed and $\sigma_\Gamma$ is defined as below:
  \[
    (\ol{\varphi}, i) \mapsto \{\{a \text{ occurs in } u\} ~|~ \psi \in \Gamma
    \text{ and } \ol{\psi} = \ol{\varphi} \text{ and } \psi \addr i = \nu^u x.
    \xi\} \text{ for each } \varphi \in \Gamma \text{ and } i \in N(\varphi)
  \]
  The function can be extended into a proof morphism $\embed \colon \BRmML \to \mML$.
\end{lemma}
\begin{proof}
  As with $\BRmML$, the main complication in constructing this morphism is that
  for $\Theta \rsep \Gamma \in \Seq'$ there may be two formulas $\varphi, \psi \in
  \Gamma$ with $\ol{\varphi} = \ol{\psi}$, i.e. two instances of the same formula
  with different annotation. The $\embed$-function `collapses' the two annotated
  formulas into one instance of a formula in which each $\nu$-instance is
  annotated by two stacks. Both the kinds of complications which can arise and
  their treatment using the $\mML$-rule $\RWk$ and the $\Rrr(\mML)$-rule $\RWeak$
  are completely analogous to those in $\FRmML$. Hence, we refer the reader to for a
  more exhaustive treatment of these issues \Cref{lem:frmml-embed}.

  Each $\BRmML$ rule is translated to its $\Rrr(\mML)$-counterpart. The rule
  \RMerge{} is always simulated by a suitable instance of \RWeak{}.
\end{proof}

\begin{corollary}[Soundness]
  If $\BRmML \vdash \Gamma$ then $\mML \vdash \ol{\Gamma}$.
\end{corollary}

Let $\TF$ be a finite fragment of $\mML$. To prove completeness, one constructs
a proof morphism $\search \colon \Sss(\TF) \to \BRmML$ embedding the proof search system
for $\TF$ into $\BRmML$.

\begin{lemma}
  There exists a function $\search \colon \Sss(\Seq_{\TF}) \to \Seq_{\BRmML}$ with 
  \[\search(\Gamma; (\Theta, \sigma)) \coloneq \widehat{\Theta}
    \rsep \Gamma^\sigma\]
  where for any $S \subseteq \Theta$, $\widehat{S} \in \Theta^*$ is the
  duplicate-free sequence of length
  $\abs{S}$, consisting of the elements of $S$ which is strictly sorted
  according to $\Theta$. The notation
  $\Gamma^\sigma \coloneq \{\varphi^{\sigma \uh \varphi}_\varepsilon ~|~ \varphi
  \in \Gamma\}$ with $(\sigma \uh \varphi)(a) \coloneq \sigma(\varphi, a)$ and
  $\varphi^{\widehat{\sigma}}_a$ is recursively defined by
  \[
    c^{\widehat{\sigma}}_a \coloneq c
    \qquad
    (\varphi \bullet \psi)^{\widehat{\sigma}}_a \coloneq
    \varphi^{\widehat{\sigma}}_{a0} \bullet \psi^{\widehat{\sigma}}_{a1}
    \qquad
    (\ocircle \varphi)_a^{\widehat{\sigma}} \coloneq \ocircle \varphi_{a0}^{\widehat{\sigma}}
    \qquad
    (\nu x. \varphi)_a^{\widehat{\sigma}} \coloneq
    \begin{cases}
      \nu^{\widehat{S}} x. \varphi^{\widehat{\sigma}}_{a0} & \text{ if } \widehat{\sigma}(i) \coloneq \{S\} \\
      \nu^\varepsilon x. \varphi^{\widehat{\sigma}}_{a0} & \text{ otherwise}
    \end{cases}
  \]
  where $c \in \prop \cup \var$, $\bullet \in \{\wedge, \vee\}$ and $\ocircle \in \{\neg, \Box,
  \Diamond\} \cup \{\mu x ~|~ x \in \var\}$.

  The function can be extended to a proof morphism $\search \colon \Sss(\TF) \to \BRmML$.
\end{lemma}
\begin{proof}
  Towards this claim, first pick some $\Rrr(\Theta, \sigma) \in \Sss(\TF)$ arranged as follows
  \[
    \inference[$\Rrr(\Theta, \sigma)$]{
      \Gamma_1; (\Theta_1, \sigma_1) \quad
      \ldots \quad
      \Gamma_n; (\Theta_n, \sigma_n)
    }
    {\Gamma; (\Theta, \sigma)}
  \]
  Then there is $R \in \TF$ with $\rho(R) = (\Gamma, \Gamma_1, \ldots,
  \Gamma_n)$ and morphisms $r_i \colon \iota_\BB(\Gamma)
  \to \iota_\BB(\Gamma_i)$ given by
  the trace interpretation. Then for each $i \leq n$ there is $(\Theta, \sigma)
  \gstep{r_i} (\Theta_i, \sigma_i)$ with the expanded sequence
  the expanded sequence
  \[(\Theta, \sigma) \step{R_{\gamma_1}} (\Theta_r^1, \sigma_r^1) \ldots
    \step{R_{\gamma_k}} (\Theta_r^k, \sigma_r^k)
    \step{P} (\Theta_p, \sigma_p) \step{r_i} (\Theta^*_i,
    \sigma^*_i) \step{T} (\Theta_i, \sigma_i) \]
  in which the initial $R_\gamma$- and $P$-steps are shared between all $i \leq
  n$ (see \Cref{lem:greedy-ceil}). Similarly to \Cref{lem:elab}, we may derive
  the following in $\BRmML$:
  \begin{comfproof}
    \AXC{$\widehat{\Theta_1}\rsep \Gamma_1^{\sigma_1}$}
    \DOC{}
    \LIC{\RMerge}
    \UIC{$\widehat{\Theta_1^*}\rsep \Gamma^0_1$}
    \AXC{$\ldots$}
    \AXC{$\widehat{\Theta_n}\rsep \Gamma_n^{\sigma_n}$}
    \DOC{}
    \RIC{\RMerge}
    \UIC{$\widehat{\Theta_n^*}\rsep \Gamma^0_n$}
    \LIC{R}
    \TIC{$\widehat{\Theta_r^k} \rsep \Gamma^{\sigma_r^k}$}
    \DOC{}
    \LIC{$\textsc{Reset}_{\gamma_1}$}
    \UIC{$\widehat{\Theta_0} \rsep \Gamma^{\sigma_0}$}
  \end{comfproof}
  That is, first apply all possible $\textsc{Reset}_{\gamma_i}$-rules, starting at the
  $\Theta_0$-greatest $\gamma_0$. Because $\Gamma^{\sigma}$ annotates fixed-points
  to which $\sigma$ `assigns' no stack with $\nu^\varepsilon$, the population
  step does not need to be replicated in the preproof as $\Gamma^{\sigma_r^k}
  = \Gamma^{\sigma_p}$. Continue by applying the rule corresponding
  to $R \in M$. Observe that while the controls resulting from this application
  matches $\Theta_i^*$, the sequents will be some sequent $\Gamma^0_i$ which
  might contain multiple copies of the `same' formula with different
  annotation (similarly to $\sigma_i^*$ assigning multiple stacks to some of its
  arguments). A simple example in which this occurs is if $\Gamma = \varphi,
  \psi \vee \xi$ with $\ol{\varphi} = \ol{\psi}$ but $\varphi \neq \psi$. If the
  $\vee$-rule is applied to this sequent, the resulting sequent $\Gamma_1^0 =
  \varphi, \psi, \xi$ contains two `copies' of $\ol{\varphi}$ (corresponding
  to $\sigma_1^*$ assigning multiple stacks to some of the quantifiers in
  $\ol{\varphi}$). In such cases, the greedy run uses a thinning to reestablish
  the property that each quantifier is assigned at most one stack. In $\BRmML$,
  this can be replicated by applying the $\RMerge$-rule to all formulas
  $\varphi, \psi \in \Gamma_i^0$ with $\ol{\varphi} = \ol{\psi}$. Note that the
  order of applications and choice of which two `$\ol{\varphi}$-instances' to
  pick for $\RMerge$-applications does not matter as the resulting sequent will
  always be $\widehat{\Theta_i} \rsep \Gamma_i^{\sigma_i}$.

  An argument analogous to that given for $\elab$ in \Cref{lem:elab} shows that
$\search$ maintains the soundness condition.
\end{proof}

\begin{corollary}[Completeness]\label{lem:ra-complete}
  If $\mML \vdash \Gamma$ then $\BRmML \vdash
  \Gamma$.
\end{corollary}

\subsubsection{The Jungteerapanich-Stirling System}
\label{sec:js}

The first reset proof system was put forward by
Jungteerapanich~\cite{jungteerapanichTableauSystemsModal2010} for the modal
$\mu$-calculus. The system is a tableaux system which induces a decision
algorithm for satisfiability of $\mu$-sequents. Later, the system was converted
to a regular validity proof system by
Stirling~\cite{stirlingProofSystemNames2013}. The latter system is usually
called the Jungteerapanich-Stirling ($\JS$) in the literature. While the system
is also inspired by the Safra construction (see \cite[Section
4.3.5]{jungteerapanichTableauSystemsModal2010} for details) the final system
is quite bespoke, as we point out later. Soundness and completeness are
proven directly with regards to the semantics of the modal $\mu$-calculus,
rather than the arguments relying on automata theory we employ in this article.
As $\JS$ is well-known in the field of cyclic proof theory, it is of interest to
compare our systems $\FRmML$ and $\BRmML$ to it. Our presentation of $\JS$
slightly differs from that given in~\cite{stirlingProofSystemNames2013} to
better fit with the notation style of this article.

The system $\JS$ assumes some fixed linear ordering $<$ on the variables which
are denoted in capital letters $X, Y, Z$. For the remainder of this section, we
only consider formulas $\varphi$ which are well-named and in which the
subsumption order $<_\varphi$ coincides with the variable ordering $<$. A finite
collection of well-named formulas can always be $\alpha$-renamed such that this
property is fulfilled.
For each variable $X$ there is an infinite supply of
\emph{names} $x_1, x_2, \ldots$ associated with $X$. The names are distinct
between distinct variables.

\emph{Sequents} of $\JS$ are expressions $\Theta \rsep \Gamma$ where $\Gamma$ is
a finite set of annotated $\mu$-formulas $\varphi^u$ and the \emph{control} $\Theta$ is a finite,
repetition-free sequence of names of variables occurring in the formulas
$\varphi^u \in \Gamma$. The
annotations $u$ of $\varphi^u \in \Gamma$ are subsequences $u \sqsubseteq
\Theta$ of $\Theta$.
Furthermore, they must be ordered according to the
ordering $<$ of variables, i.e. if a name $x$ of $X$ and a name $y$ of $Y$
appear in $u$ and $X < Y$ then $x$ must occur before $y$ in $u$.

For $u \sqsubseteq \Theta$ denote by $u \uh X$ the
subsequence of $u$ from which all names corresponding to variables $Y > X$ have
been removed. For two names $x, y \in \Theta$ write $x \sqsubset_\Theta y$
if either $x$ is a name for $X$ and $y$ is a name for $Y$ with $X < Y$ or if $x$
and $y$ are names for the same variable $X$ and $x$ occurs before $y$ in
$\Theta$. This extends to sequences $u, v \sqsubseteq \Theta$, writing $u
\sqsubset_\Theta v$ if $u$ contains the $\sqsubset_\Theta$-least name which
occurs in only one of the two sequences.

\begin{definition}
  The \emph{derivation rules} of $\JS$ are given below. Write $\Theta'$ to
  denote the control $\Theta$ from which all names not occurring in annotations
  in the corresponding $\Gamma$ were removed.
  \begin{mathpar}
    \inference[\RAx]{}{\Theta \rsep \Gamma, p^u, \neg p^v}

    \inference[$\vee$]{\Theta \rsep \Gamma, \varphi^u, \psi^u}{\Theta \rsep \Gamma, \varphi \vee \psi^u}

    \inference[$\wedge$]{\Theta \rsep \Gamma, \varphi^u \quad \Theta \rsep \Gamma, \psi^u}{\Theta \rsep \Gamma, \varphi \wedge \psi^u}

    \inference[$\Box$]{\Theta \rsep \Gamma, \varphi^u}{\Theta \rsep \Diamond \Gamma, \Box \varphi^u}

    \inference[$\mu$]{\Theta' \rsep \Gamma, \varphi[\mu X. \varphi / X]^{u \uh X}}{\Theta \rsep \Gamma, \mu X. \varphi^u}

    \inference[$\nu$]{\Theta' x \rsep \Gamma, \varphi[\nu X. \varphi /
      X]^{(u \uh X)x} \quad x \text{ fresh } X\text{ name}}{\Theta \rsep \Gamma, \nu X. \varphi^u}

    \inference[\RThin]{\Theta' \rsep \Gamma, \varphi^u \quad u \sqsubset_\Theta v}{\Theta \rsep \Gamma, \varphi^u, \varphi^v}

    \inference[$\RReset_{x}$]{\Theta' \rsep \Gamma, \varphi_1^{ux},
      \ldots, \varphi_n^{ux} \quad x \text{ not in } \Gamma}{\Theta \rsep \Gamma, \varphi_1^{uxx_1u_1},
      \ldots, \varphi_n^{uxx_nu_n}}
  \end{mathpar}

  A $\JS$-preproof is a \emph{proof} if
  every pair of bud $t
  \in \dom(\beta)$ and companion $\beta(t)$ has an \emph{invariant}, i.e. there
  exists a name $x$ such that $x$ occurs in all of the controls
  $\Theta$ between $t$ and $\beta(t)$ and the $\textsc{Reset}_x$ rule
  is applied between $t$ and $\beta(t)$.
  A $\JS$-proof is \emph{a proof 
    of $\Gamma$} if its root
  is labeled $\varepsilon \rsep \Gamma^\varepsilon$.
\end{definition}

A $\mu$-formula is \emph{guarded} if a $\Box$ occurs on the paths between a
binder $\nu X$ or $\mu X$ and each of its bound variables $X$. The following is
proven in \cite[Theorem 4]{stirlingProofSystemNames2013}.

\begin{proposition}
  If $\gamma$ is closed and guarded then $\JS \vdash \gamma$ iff $\gamma$ is valid.
\end{proposition}

It is clear at a glance that the systems $\BRmML$ and $\JS$ are quite
different: While $\BRmML$ annotates quantifier instances, the annotations of
$\JS$ concern formulas in $\Gamma$ and their $\nu$-variables (it is easily
observed that the annotations can never be extended with $\mu$-variables). Their
commonality ends at both systems being reset proof systems for the modal
$\mu$-calculus. The 
comparison between $\FRmML$ and $\JS$ will turn out much more revealing.

In a way, the annotations of $\FRmML$ and $\JS$ `track' the same trace
values: The $\nu$-variables of each formula in $\Gamma$ on a per-formula basis.
In $\FRmML$ this is obvious: Each formula $\varphi$ in $\Gamma$ comes with an assignment
$\sigma$ which assigns each of variable $X$ bound by $\nu X$ in $\varphi$ a
subsequence of the control $\Theta$. In $\JS$ this is a more subtle observation:
At first glance, annotations are only per formula $\varphi$. However, these annotations
consist only\footnote{This is not quite accurate. Rather, the names in the
  annotations are for variables which occurred `hereditarily' in the formula.
  For example, the formula $\nu X. \nu Y. Y$, after two applications of the
  $\nu$-rule, will be unfolded to $\nu Y. Y$ annotated by a sequence $xy$, with the obvious variable
  correspondences, even though $X$ does not occur in $\nu Y. Y$ `anymore'.
  However, these `anomalies' will only ever occur as prefixes of annotations
  which `eventually' do not impact proof search anymore.} of variables which
occur in $\varphi$. The $\mu$- and $\nu$-rules ensure that the names in the
annotation are ordered by the global ordering $<$ on variables. In this sense,
the annotation $u$ of $\varphi^u$ can be separated into subsequences of names
$x_1 \ldots x_n$, for each variable $X$ in $\varphi$, tracking the progress  of
each trace value $(\varphi, X)$ (in the sense of the $\iota_\FF$ GTC for
$\mML$).
In that light, it can be observed that the $\mu$-rule of the two
systems is essentially the same: It `cancels' all progress made by
$\nu$-variables subsumed by the unfolded $\mu$-formula by removing their
associated annotations.

The ordering $\sqsubset_\Theta$ is the lexicographic ordering comparing first
according to the global ordering $<$ on variables and subsequently according to
`age ordering' given by positions in $\Theta$. This illustrates the most
significant point of departure between $\FRmML$ and $\JS$: The design of $\JS$
takes into account some deeper insights into the semantics of the
modal $\mu$-calculus, specifically the role of the subsumption ordering
$<_\varphi$ in the validity condition.
Another instance of these insights comes
into play is the $\nu$-rule: When unfolding the $\nu$-quantifier of a formula
$\nu X. \varphi^u$ in $\JS$, the resulting annotation is $(u \uh X)x$, `clearing
off' the annotations of variables subsumed by $X$. This is necessary to ensure
that the names in the annotation remain ordered according to the global variable
order $<$ (otherwise the name $x$ would likely be appended after names of
subsumed variables, disturbing this ordering property). The completeness of this
rule hinges on a semantic insight: If there is a successful trace on a
$\nu$-variable in a formula, all `higher' $\nu$-variables subsuming it must also
have a successful trace. Thus, forgetting the progress of `lower'
$\nu$-variables upon progress in `higher' $\nu$-variables does not endanger
completeness. Partly, Jungteerapanich and Stirling can `get away with this' as
they prove soundness and completeness directly with regards to the semantics of
the modal $\mu$-calculus, `skipping' the automata theoretic considerations we
make in this article. Proving a similar result in the more generic setting of
trace categories with subsumption orders on their trace objects is highly
intricate. The system $\FRmML$ does thus not employ this kind of
`optimization' as this would have marked a departure from our original goal of
generating reset proof systems in a simple, `effortless' manner.

Another such point of difference between $\JS$ and $\FRmML$ is the $\RThin$-rule
of $\JS$. It should be noted that the only reason that $\JS$ only features this
strict thinning rule instead of a more general weakening rule is that it was
originally designed as a tableaux system for proof search in the modal
$\mu$-calculus. The $\RThin$-rule is the weakest rule which yields completeness
for the system (its purpose being analogous to the thinning steps of Safra
boards). However, the addition of a general weakening rule which allowed the
discarding of arbitrary formulas from $\Gamma$ would leave the soundness of
$\JS$ unchanged. The interesting point of comparison between $\JS$ and $\FRmML$
in this regard lies in $\RThin$ of $\JS$ and $\RMerge$ of $\FRmML$ recalled
below.
\[
  \inference[\RThin]{\Theta' \rsep \Gamma, \varphi^u \quad u \sqsubset_\Theta v}{\Theta \rsep \Gamma, \varphi^u, \varphi^v}
  \qquad \qquad
  \inference[\(\RMerge\)]{\Theta' \rsep \Gamma, (\varphi,
    \merge_\Theta(\sigma, \sigma'))}{\Theta \rsep \Gamma, (\varphi, \sigma), (\varphi, \sigma')}
\]
Both rules are included for essentially the same reason: To mirror the thinning
step in the greedy Safra board runs required for completeness. The $\RThin$-rule
once again embodies a semantic insight: It is sufficient to `keep' the formula
which has the `best progress' on the highest variable. To illustrate this point,
suppose $\Theta$ was $x_1y_1x_2$, corresponding to variables $X$ and $Y$ in
$\varphi$, and that the sequent contained two copies of $\varphi$:
$\varphi^{x_1x_2}$ and $\varphi^{x_1y_1}$. In this case, an application of
$\RThin$ would discard $\varphi^{x_1y_1}$, essentially because $x_1x_2$ `has
better $X$-progress'. This is `complete' as `success' in `lower' variables
always entails `success' in `higher' variables as elaborated in the previous
paragraph. On the other hand, the $\RMerge$-rule of $\FRmML$ does once more not
embody this insight. Consider the analogous case: A control $\Theta$ of form $abc$ and
two annotated copies of $\varphi$ given by $(\varphi, (X \mapsto ac, Y \mapsto
\varepsilon))$ and $(\varphi, (X \mapsto a, Y \mapsto b))$. In this case, the
$\RMerge$-rule `keeps' the annotation $X \mapsto ac$ \emph{and} the annotation
$Y \mapsto b$ in the resulting merged annotation of $\varphi$, the `best'
annotation for each separate variable, respectively. This is required to mirror
the thinning steps of greedy Safra board runs, which do the same. While it would
likely also be sound and complete for $\RMerge$ in $\FRmML$ to instead always
`keep' the annotation $\sigma$ which has the `best' annotation for the `highest'
variable, proving this would once again be extremely intricate.

\section{Conclusion}
\label{sec:conclusion}

We have shown that to each cyclic proof system $\RR$ with a soundness condition
specified in terms of an activation algebra $\Aa$, there is an associated cyclic
proof system $\Rrr(\RR)$ with a $\RReset$-based soundness condition. The
construction of $\Rrr(\RR)$ is fully independent of the underlying logic of $\RR$,
only relying on the specification of the global trace condition in terms of
$\Aa$. The equivalence of $\RR$ and $\Rrr(\RR)$ is proven via cyclic proof
system homomorphisms and a proof system $\Sss(\RR)$ tailored to easing proof search.
The method of cyclic proof system homomorphisms allows the
equivalence between $\RR$ and $\Rrr(\RR)$ to be extended to bespoke
$\RReset$-based proof systems $\TS$, as
demonstrated in \Cref{sec:concrete}. This strategy is applied to present equivalent
reset systems for cyclic arithmetic, Gödel's
T and the modal $\mu$-calculus.

\paragraph{Discussion}

Our approach comes with some shortcomings. First, while broad, the scope of
applicability of our results is not universal. We have only demonstrated
how to give corresponding reset systems to cyclic proof systems with
$\TT_\Aa$-specifiable global trace conditions. There are some global trace
conditions which likely cannot be specified this way, for example that given by Hazard
for transfinite expressions~\cite{hazardCyclicProofsTransfinite2022}.
Furthermore, the soundness of the original system $\RR$ must be a global trace
condition for our method to apply. We did not consider other kinds of soundness conditions, such as
induction orders~\cite{sprengerStructureInductiveReasoning2003a}, bouncing
threads~\cite{baeldeBouncingThreadsCircular2022} or trace
manifolds~\cite{brotherstonSequentCalculusProof2006}.
Another shortcoming is that the na\"ive cyclic system $\Rrr(\RR)$ generated from a suitable cyclic
proof system $\RR$ can be `unwieldy'. For each concrete reset system we give in
\Cref{sec:concrete}, some modifications were necessary to make the resulting
system pleasant for human use, specifically in finding a good `syntax' for
sequents of the concrete reset system. More generally, we believe that to
turn $\Rrr(\RR)$ into a `pleasant' system, some amount of human
creativity is still required.
One of the biggest strengths of the method we have described, its independence
of semantic considerations about the logic, is also one of its biggest
drawbacks: The systems generated by our method do not take advantage of semantic
insights into the logic in question and proving the equivalence of our systems
and other reset systems from the literature may take considerable effort. 
Jungteerapanich's reset system for the modal \( \mu
\)-calculus~\cite{jungteerapanichTableauSystemsModal2010} provides an example of
a $\RReset$-system designed using deep semantic insight in a manner $\Rrr(\mML)$
for both trace interpretations of $\mML$ does not (this is discussed in greater
detail in \Cref{sec:js}).

It should be noted that $\Rrr(\RR)$ describes merely \textit{one way} of
designing reset proof systems.
The reset proof system given in~\cite{afshariCyclicProofsFirstorder2022a} provides an
example of the potential for variation.
The reset rule $\text{RS}(\kappa)$ utilised in that system corresponds
 to the following transition on Safra boards: Fix
a board $(\Theta, \sigma) \in \Sb(\Aa, X)$ and pick a covered $\kappa \in
\Theta$. Define the set of $C(\kappa)$ of children of $\kappa$ as
\[
  C(\kappa)
  \coloneq \bigl\{\text{min}_\Theta \{\gamma \in S ~|~ \kappa < \gamma\} \bigm| x \in X, a \in
  \Aa, S \in \sigma(a, x) \text{ and } \kappa \in S\bigr\}.
\]
If every $\gamma \in
C(\kappa)$ is also covered, one may perform a `reset operation' yielding the
board $(\Theta \setminus C(\kappa), \sigma \setminus C(\kappa))$ where $(\sigma
\setminus C(\kappa))(x, a) \coloneq \{S \setminus C(\kappa) \mid S \in \sigma(x,
a)\}$. Replacing the reset transitions of \Cref{def:reset} with this variant of
the reset condition would yield an abstract reset proof system
enjoying the same soundness and completeness properties as our chosen form of $\Rrr(\RR)$, albeit requiring
slight modifications to their proofs. Most likely there are multiple
permissible alternatives to the `reset machinery' we
present in this article. We chose to only
cover one, namely the one the closest to the traditional Safra construction for
Rabin automata.

\paragraph{Related work}

We are aware of three articles designing reset systems for cyclic proof systems:
Jungteerapanich-Stirling~\cite{jungteerapanichTableauSystemsModal2010,stirlingProofSystemNames2013},
Afshari et~al.~\cite{afshariCyclicProofsFirstorder2022a} and Afshari
et~al.~\cite{afshariCyclicProofSystem2023}.

Jungteerapanich~\cite{jungteerapanichTableauSystemsModal2010} and
Stirling~\cite{stirlingProofSystemNames2013} propose reset proof systems for the
modal $\mu$-calculus, respectively for satisfiability and validity. These are
the first reset systems in the cyclic proof theory literature. A comparison
between their validity system and the systems we derive for validity of the
$\mu$-calculus in this article can be found in \Cref{sec:js}. The upshot is that
while also inspired by the Safra construction, their system also incorporates
multiple insights into the semantics of the $\mu$-calculus which our systems
neglect.

Afshari et~al.~\cite{afshariCyclicProofsFirstorder2022a} give a reset proof system for the
first-order $\mu$-calculus. It is based on the cyclic proof system
for the first-order $\mu$-calculus with ordinal approximations put forward by
Sprenger and
Dam~\cite{sprengerGlobalInductionMechanisms2003,sprengerStructureInductiveReasoning2003a}.
The crucial insight underpinning its design is that the mechanism of ordering
ordinal variables in the Sprenger-Dam system is already very similar to the control of
the Jungteerapanich-Stirling system. Thus, the Sprenger-Dam system is
extended into a reset system in a very natural manner.

Afshari et~al.~\cite{afshariCyclicProofSystem2023} give a reset proof system for
full computation tree logic ($\mathrm{CTL}^*$). Their system is a hypersequent
calculus and thus requires a more intricate trace condition. The annotations
used in their reset condition are either empty or one letter. In this, their system falls
in between the `full' reset proof systems, such as $\JS$, that for the
first-order $\mu$-calculus or ours, and the `mere' path condition systems
discussed in a subsequent paragraph.

It should be noted that
all of the aforementioned reset systems were designed by combining automata
theoretic considerations, specifically the Safra construction, and semantic
insights about the logic for which the systems were constructed. In both cases,
this resulted in systems which are more elegant than the systems we generate in
this article. It should however be noted that designing such elegant systems
requires considerable effort. Furthermore, we believe that for many technical
purposes of cyclic proof theory, the na\"ive reset proof systems derived in this
article shall prove sufficient.

The literature also contains articles on cyclic proof systems with path
conditions, i.e. whose trace conditions allow each simple cycle of the preproof
to be considered separately, which are not strictly reset proof systems
because their soundness condition is implemented in a simpler manner. Specifically,
Marti and Venema~\cite{martiFocusSystemAlternationFree2021} demonstrate that for
the alternation-free fragment of the $\mu$-calculus, the
Jungteerapanich-Stirling system can be simplified to a system in which formulas
are annotated with one ``one bit of information'' (said to be \emph{in focus} or
\emph{out of focus}, respectively) and which does not need a $\RReset$-rule. The
resulting system still possesses a path condition.
Rooduijn~\cite{rooduijnCyclicHypersequentCalculi2021} gives a very similar path
condition for cyclic proof systems of modal logics with the master modality.
It should be noted that all positive properties of reset proof systems we have
mentioned in this article, such as their suitability for proof theoretic
investigations and proof search, extend to all cyclic proof systems with path
conditions.

Cyclic proof systems with path conditions, such as reset proof systems, have
proven well-suited to proof theoretic investigations. So far, the results which
employ them are in the areas of
interpolation~\cite{afshariLyndonInterpolationModal2022,afshariUniformInterpolationCyclic2021,martiFocusSystemAlternationFree2021}
and the translation of proof of cyclic proof systems into proofs in non-cyclic
proof systems with suitable induction
axioms~\cite{afshariCutfreeCompletenessModal2017}. However, there seems to be no
reason to assume that cyclic proof systems with path conditions might not also
prove useful in proving other properties, such as $\RCut$-elimination, or
investigations of the computational contents of cyclic proofs.


\paragraph{Future work}

The results of this article open up many avenues of future research. As noted
previously, reset proof systems have proven to be valuable tools in the arsenal
of cyclic proof theory. With reset systems for many more cyclic proof systems
now available `off the shelf', we hope to see more proof theoretic
investigations using reset proof systems in the future. This could proof
especially valuable to the proof theories of logics with features particularly
well-suited to cyclic proof systems, such as fixed-points and inductive
definitions.

This article elaborated on the relationship between the global trace condition
and the reset path condition using the abstract notion of trace put forward
in~\cite{afshariAbstractCyclicProofs2022}. There are further soundness
conditions for cyclic proofs, such as induction
orders~\cite{sprengerStructureInductiveReasoning2003a}, bouncing
threads~\cite{baeldeBouncingThreadsCircular2022} and trace
manifolds~\cite{brotherstonSequentCalculusProof2006}. We hope to `complete the
picture' in the future by investigating these other soundness conditions and
their relationships in this abstract setting. An interesting aspect to explore
in this direction is the fact in some reset proof systems, an induction order
can essentially be `read off' the system's proofs. An example of such a system
is that for Gödel's T given in \Cref{sec:reset-godel}. However, this is not true
for all reset proof systems. For example, the system for Peano arithmetic in
\Cref{sec:ra} does not possess this property.

\printbibliography{}
\end{document}